%% file: template.tex
\pdfoutput=1
\documentclass[12pt,draftcls,onecolumn]{IEEEtran}

\makeatletter

\let\proof\@undefined
\let\endproof\@undefined
\makeatother
\usepackage{amsfonts}
\usepackage{amsmath}
\usepackage{amssymb}
\usepackage{amsthm}
\usepackage{graphicx}
\usepackage{epsfig}
\usepackage{setspace}
\usepackage{balance}
\usepackage{mathtools}

 \DeclareMathOperator*{\argmin}{argmin}
  \DeclareMathOperator*{\tot}{total}
 \DeclareMathOperator*{\worst}{worst}
  \DeclareMathOperator*{\switch}{switch}
 
  \DeclareMathOperator*{\best}{best}

\newcommand{\Nat}{I\!\!N} 
\newcommand{\Real}{I\!\!R} 
\newcommand{\Expectation}{I\!\!E} 
\newcommand{\transpose}{{\!\scriptscriptstyle\mathrm T}}  
 \newtheorem{theorem}{\bf Theorem}

 \newtheorem{lemma}{\bf Lemma}

 \newtheorem{example}{\bf Example}
  \newtheorem*{conditionw}{\bf Condition (W)}
  \newtheorem*{conditiong}{\bf Condition (G)}
  \newtheorem*{conditionb}{\bf Condition (B)}
  \newtheorem*{conditionb2}{\bf Condition (B2)}
  \newtheorem*{conditione}{\bf Condition (E)}

\begin{document}

\title{Energy-Efficient Transmission Scheduling 
with Strict
Underflow Constraints}
%
%
\author{David I Shuman, Mingyan Liu, and Owen Q. Wu
\thanks{David I Shuman and Mingyan Liu are with the Electrical
Engineering and Computer Science Department, University of Michigan, Ann Arbor, MI.
Email: \{dishuman, mingyan\}@umich.edu}
\thanks{Owen Q. Wu is with the Ross School of Business, University of Michigan, Ann Arbor, MI.
Email: owenwu@bus.umich.edu}
\thanks{Part of the work reported here was presented at the \emph{International Symposium on Modeling and Optimization in Mobile, Ad Hoc, and Wireless Networks (WiOpt), April 2008, Berlin, Germany}.}}
\maketitle
\thispagestyle{empty}
\begin{abstract}
\input{abstract}

\end{abstract}





\doublespacing
\input{introduction}

\input{problem}
\input{lowSNR}

\input{highSNR}
\input{two_users}

\input{discussion}

\input{conclusion}

\input{appendix}
\bibliographystyle{IEEEtr}
\bibliography{alloc}

\end{document}

%% file: abstract.tex
We consider a single source transmitting data to one or more receivers/users over a shared wireless channel. Due to random fading, the wireless channel conditions vary with time and from user to user.
Each user has a buffer to store received packets before they are drained. At each time step, the source determines how much power to use for transmission to each user. The source's objective is to allocate power in a manner that minimizes an expected cost measure, while satisfying strict buffer underflow constraints and a total power constraint in each slot. The expected cost measure is composed of costs associated with power consumption from transmission and packet holding costs. The primary application motivating this problem is wireless media streaming. For this application, the buffer underflow constraints prevent the user buffers from emptying, so as to maintain playout quality. In the case of a single user with linear power-rate curves, we show that a modified base-stock policy is optimal 
under the finite horizon, infinite horizon discounted, and infinite horizon average expected cost criteria. For a single user 
with piecewise-linear convex power-rate curves,
we show that a finite generalized base-stock policy is optimal under all three 
expected cost criteria. We also present the sequences of critical numbers that complete the characterization of the optimal control laws in each of these cases when some additional technical conditions are satisfied. We then analyze the structure of the optimal policy for the case of two users. We conclude with a discussion of methods to identify implementable near-optimal policies for the most general case of M users.

%% file: introduction.tex
\section{Introduction}\label{Se:introduction}
In this paper, we examine the problem of energy-efficient transmission scheduling over a wireless channel, subject to underflow constraints. We consider a single source transmitting to one or more receivers/users over a shared wireless channel. Each 
user has a buffer to store received packets before they are drained at a certain rate. The available data rate of the channel varies
with time and from user to user, due to random fading. The
transmitter's goal is to minimize total power consumption by
exploiting the temporal and spatial variation of the channel,
while preventing any user's buffer from emptying.

\subsection{Opportunistic Scheduling and Related Work}
This problem falls into the general class of opportunistic scheduling problems, where the common theme is to exploit the temporal and spatial variation of the channel.\footnote{Opportunistic scheduling problems are also referred to as multi-user variable channel scheduling problems \cite{andrews1}.} At a high level, the 
idea of exploiting the \emph{temporal diversity} of the channel via opportunistic scheduling can be explained as follows. Consider the case of a single sender transmitting to a single receiver with different linear power-rate curves for each possible channel condition. Consider one scheduling policy that transmits data in a just-in-time fashion, without regard to the condition of the time-varying channel. Over the long run, the total power consumption tends toward the power consumption per data packet under the average channel condition times the number of packets sent. If instead, the scheduler aims to send more data when the channel is in a ``good'' state (requiring less power per data packet), and less data when the channel is in a ``bad'' state, the total power consumption should be lower. Much of the challenge for the scheduler lies in determining how good or bad a channel condition is, and how much data to send accordingly.

Similarly, in the case of multiple receivers, the scheduler can exploit the \emph{spatial diversity} of the channel by transmitting only to those receivers who have the best channel conditions in each time slot.
The benefit of increasing system throughput and reducing total power
consumption through such a joint resource allocation policy is
commonly referred to as the \emph{multiuser diversity gain}
\cite{viswanath}. It was introduced in the context of the
analogous \emph{uplink} problem where multiple sources transmit to
a single destination (e.g., the base station) \cite{knopp}. Since,
there has been a wide range of literature on opportunistic
scheduling problems in wireless networks.

Sending more data when the channel is in a good state can increase system throughput and/or reduce total energy consumption; however, in opportunistic scheduling problems, it is often the case that the transmission scheduler has competing quality of service (QoS) interests. For instance,
one QoS interest commonly considered is \emph{fairness}. If, when a singe source is transmitting to multiple receivers, the scheduler only considers total throughput and energy consumption across all users, it may often be the case that it ends up transmitting to a single user or the same small group of users in every slot. This can happen, for instance, if a base station requires less power to send data to a nearby receiver, even when the nearby receiver's channel is in its worst possible
condition and a farther away receiver's channel is in its best possible condition. Thus, fairness constraints are often imposed to ensure that the transmitter sends packets to all receivers.

A number of different fairness conditions have been examined in the literature.
%
For example, \cite{berggren} and
\cite{shroff2} consider \emph{temporal fairness}, where the scheduler must transmit to each receiver for some minimum fraction of the time over the long run. Under the \emph{proportional fairness} considered by \cite{viswanath}
and \cite{holtzman}, the scheduler considers the current channel conditions relative to the average channel condition of each receiver.
Reference \cite{shroff2} considers a more general \emph{utilitarian
fairness}, where the focus is on system performance from the receiver's perspective, rather than on resources consumed by each user. The authors of  \cite{borst} incorporate fairness directly into the objective function by setting relative throughput target values for each receiver and maximizing the minimum relative long-run average throughput.

Another QoS consideration that is important in many applications is delay. Different notions of delay have been incorporated into opportunistic scheduling problems. 
One proxy for delay 
is the \emph{stability} of all of the sender's queues for arriving packets awaiting transmission. The motivation for this criterion is that if none of these queues blows up, then the delay is not ``too bad.'' With stability as an objective, it is common to restrict attention to \emph{throughput optimal} policies, which are scheduling policies that ensure the sender's queues are stable, as long as this is possible for the given arrival process and channel model. References
\cite{andrews2}\nocite{tassiulas}\nocite{shakkottai}-\cite{neely_stable} present such throughput optimal scheduling algorithms, and examine conditions guaranteeing stabilizability in different settings.



When an arriving packet model is used for the data, one can also define \emph{end-to-end delay} as the time between a packet's arrival at the sender's buffer and its decoding by the receiver. A number of opportunistic scheduling studies have considered the \emph{average} end-to-end delay of all packets over a long horizon. For instance, \cite{cruz3}-\nocite{berry}\nocite{goyal}\nocite{wang_thesis}\nocite{rajan}\nocite{berry_yeh}\nocite{djonin}\nocite{bhorkar}\nocite{kittipiyakul1}\nocite{agarwal}\nocite{goyal08}\cite{kittipiyakul2} all consider average delay, either as a constraint or by incorporating it directly into the objective function to be minimized.
However, the average delay criterion allows for the possibility of long delays (albeit with small probability); thus,
for many delay-sensitive applications, \emph{strict} end-to-end delay is often a more appropriate consideration for studies with arriving packet models. In \cite{chen_neely2} and \cite{chen_neely}, Chen, Mitra, and Neely place strict constraints on the end-to-end delay of each packet in a point-to-point system, examine the optimal scheduling policy assuming all future channel conditions are known, and suggest heuristics based on this optimal \emph{offline} scheduling policy for the more realistic \emph{online} case where the scheduler only learns the channel conditions in a causal fashion. Rajan, Sabharwal, and Aazhang also consider strict constraints on the end-to-end delay in an arriving packet model in \cite[Section IV]{rajan}.



A strict constraint on the end-to-end delay of each packet is one particular form of a \emph{deadline} constraint, as each packet has a deadline by which it must be transmitted. This notion can be generalized to impose individual deadlines on each packet, whether 
the packets are arriving over time or are all in the sender's buffer from the beginning. 
References \cite{fu_conf}-\nocite{fu}\nocite{jindal3}\nocite{jindal_gen}\nocite{jindal}\cite{jindal4} consider point-to-point communication when a fixed amount of data is in the sender's buffer at the start of the time horizon and the individual deadlines coincide, so that all packets must be transmitted and received by a common deadline, the end of the time horizon under consideration. In \cite[Section III-D]{fu_conf} and \cite[Section III-D]{fu}, Fu, Modiano, and Tsitsiklis specify the optimal transmission policy when the power-rate curves under each channel condition are linear and the transmitter is subject to a per slot peak power constraint. In \cite{jindal3}-\nocite{jindal_gen}\nocite{jindal}\cite{jindal4}, Lee and Jindal model the power-rate curve under each channel condition as convex, first of the form of the so-called Shannon cost function based on the capacity of the additive white Gaussian noise channel, and then as a convex monomial function.\footnote{In our notation of Section \ref{Se:Sys_Model}, these two cases correspond to power-rate curves of the form $c(z,s)=\frac{2^z-1}{{g}_1(s)}$ and $c(z,s)=\frac{z^\zeta}{g_2(s)}$, respectively, where $c(z,s)$ is the power required to transmit $z$ bits under channel condition $s$, $g_1(\cdot)$ and $g_2(\cdot)$ are known functions, and $\zeta$ is a fixed parameter.}

References \cite{uysal_04} and \cite{tarello} consider opportunistic scheduling problems with multiple receivers and a single deadline constraint at the end of a finite horizon. Packets arrive over time and the emphasis is on offline scheduling policies in \cite{uysal_04}, whereas \cite{tarello} considers a fixed amount of data destined for each receiver, and assumes the data is already in the sender's buffers at the beginning of the horizon. The model of \cite{tarello} is perhaps the closest to our general model for $M$ receivers; however, two key differences are (i) the transmitter is not subject to a power constraint in \cite{tarello}; and (ii) the transmitter can transmit to at most one receiver in each time slot in \cite{tarello}.

In our model, the strict underflow constraints serve as a notion of both fairness and delay. The notion of fairness is that none of the receivers' buffers are allowed to empty, guaranteeing the required level of service to all users. The underflow constraints also serve as a notion of delay, and can be seen as multiple deadline constraints - certain packets must arrive by the end of the first slot, another group by the end of the second slot, and so forth. Therefore, Sections \ref{Se:lowSNR} and \ref{Se:highSNR} of this paper aim to generalize the works of \cite{fu_conf}-\cite{fu} and \cite{jindal3}-\nocite{jindal_gen}\nocite{jindal}\cite{jindal4}, respectively, by considering multiple deadlines in the point-to-point communication problem, rather than a single deadline at the end of the horizon. In addition to better representing some delay-sensitive applications, this extension of the model also allows us to consider infinite horizon problems. We compare related work in opportunistic scheduling problems with deadline constraints further in \cite{chapter}. For more complete surveys of opportunistic scheduling studies in wireless networks, see \cite{shroff1} and \cite{shroff3}.

%




\subsection{Wireless Media Streaming and Related Work}
The primary application we have in mind to motivate this problem is wireless media streaming. For this application, the data are audio/video sequences, and the packets are drained from the receivers' buffers in order to be decoded and played. Enforcing the underflow constraints reduces playout interruptions to the end users. In order to make the presentation concrete, we use the above wireless media streaming terminology throughout the paper.

Transporting multimedia over wireless networks is a promising
application that has seen recent advances \cite{girod}.  At the
same time, a number of resource allocation issues need to be
addressed in order to provide high quality and efficient media
over wireless. First, streaming is in general bandwidth-demanding.
Second, streaming applications tend to have
stringent 
QoS requirements (e.g., they can be
delay and jitter intolerant). Third, it is desirable to operate
the wireless system in an energy-efficient manner. This is obvious
when the source of the media streaming (the sender) is a mobile.
When the media comes from a base station that is not
power-constrained, it is still desirable to conserve power in
order to (i) limit potential interference to other base stations and
their associated mobiles, and (ii) maximize the number of receivers the sender can support.
Of the related work in wireless media streaming, \cite{bambos2} has the
closest setup to our model. The main differences are that
\cite{bambos2} features a loose constraint on underflow (i.e., it
is allowed, but at a cost), as opposed to our tight constraint,
and the two studies adopt different wireless channel models. In
the extension \cite{bambos1}, the receiver may slow down its
playout rate (at some cost) to avoid underflow. In this setting,
the authors investigate the tradeoffs between power consumption
and playout quality, and examine joint power/playout rate control
policies. In our model, the receiver does not have the option to
adjust the playout speeds. Our model also bears resemblance to
\cite{luna}. The first difference here is that \cite{luna} aims to
minimize transmission energy subject to a constant end-to-end
delay constraint on each video frame. A second difference is that
the controller in \cite{luna} must assign various source coding
parameters such as quantization step size and coding mode, whereas
our model assumes a fixed encoding/decoding scheme.

\subsection{Summary of Contribution}
In this paper, we formulate the task of energy-efficient transmission scheduling subject to strict underflow constraints as three different Markov decision problems (MDPs), with the finite horizon discounted expected cost, infinite horizon discounted expected cost, and infinite horizon average expected cost criteria, respectively. These three MDPs feature a continuous component of the state space and a continuous action space at each state. Therefore, unlike \emph{finite} MDPs, they cannot in general be solved exactly via dynamic programming, and suffer from the well-known \emph{curse of dimensionality} \cite{rust,chow}. Our aim in this paper is to analyze the dynamic programming equations in order to (i) determine if there are circumstances under which we can analytically derive optimal solutions to the three problems; 
and (ii) leverage our mathematical analysis and results on the structures of the optimal scheduling policies 
to improve our intuitive understanding of the problems.

We begin by showing that in the case of a single receiver under linear power-rate curves, the optimal policy is an easily-implementable \emph{modified base-stock} policy. In each time slot, it is optimal for the sender to transmit so as to bring the number of packets in the receiver's buffer level after transmission as close as possible to a target level or critical number.\footnote{We use the terms target level and critical number interchangeably throughout the paper.} The target level depends on the current channel condition, with a better channel condition corresponding to a higher target level.
We also show that the strict underflow constraints may cause the scheduler to be less opportunistic than it otherwise would be, and transmit more packets under ``medium'' channel conditions in anticipation of deadline constraints in future time slots.

We then generalize this result in two different directions. First, we relax the assumption that the power-rate curves under each channel condition are linear, and model them as piecewise-linear convex to better approximate more realistic convex power-rate curves. Under piecewise-linear power-rate curves, we show the optimal policy is a \emph{finite generalized base-stock} policy, and provide an intuitive explanation of this structure in terms of multiple target levels in each time slot.
In addition to the structural results on the optimal policy for the case of a single receiver under either linear or piecewise-linear convex power-rate curves, we provide an efficient method to calculate the critical numbers that complete the characterization of the optimal policy when  certain technical conditions are satisfied.

The second generalization of the single receiver model under linear power-rate curves is to a single user transmitting to two receivers over a shared wireless channel. In this case, we state and prove the structure of the optimal policy, and show how the peak power constraint in each slot couples the optimal scheduling of the two receivers' packet streams.

In all three setups, we show the structure of the optimal policy in the finite horizon discounted expected cost problem extends to the infinite horizon discounted and average expected cost problems.

Throughout the analysis, we make a novel connection 
with inventory models that may prove useful in other wireless transmission scheduling problems. Because the inventory models corresponding to our wireless communication models have not been previously examined, our results also represent a contribution to the inventory theory literature.

%

The remainder of this paper is organized as follows. 
In the next
section, we describe the system model, formulate finite and infinite horizon MDPs, 
and relate our model to
models in inventory theory. In Section \ref{Se:lowSNR}, we consider the case of a single receiver under linear power-rate curves. While this case can be considered a special case of the models of Sections \ref{Se:highSNR} and \ref{Se:two_users}, we present it first in order to (i) state additional structural properties of the optimal transmission policy to a single user under linear power-rate curves that are not true in general for the 
 cases discussed in Sections \ref{Se:highSNR} and \ref{Se:two_users}; (ii) highlight some intuitive takeaways that carry over to the generalized models, but are more transparent in the simpler model; and (iii) compare it to related problems in the wireless communications literature.
We analyze the structure of the optimal scheduling policy for the finite horizon 
problem, provide a method to compute the critical numbers that complete the characterization of the optimal policy when some additional technical conditions are met, and provide sufficient conditions for this problem to be equivalent to a previously-studied single deadline problem.
Section \ref{Se:highSNR} generalizes the analysis of Section \ref{Se:lowSNR} to the case of a single receiver under piecewise-linear convex power-rate curves, and also addresses the infinite horizon problems for the case of a single receiver. In Section \ref{Se:two_users}, we analyze the structure of the optimal policy when there are two receivers with linear power-rate curves. We discuss the 
relaxation of the strict underflow constraints and the extension to the general case of $M$ receivers in Section \ref{Se:discussion}. Section \ref{Se:conclusion} concludes the paper.


%% file: problem.tex
\section{Problem Description}\label{Se:problem}

In this section, we present an abstraction of the transmission
scheduling problem outlined in the previous section and formulate
three optimization problems.
While most of 
this paper focuses on
the cases of one and two users, the formulation in this section is
for the more general multi-user (multi-receiver) case, so that we can discuss this more general case in Section \ref{Se:m_users}.

\subsection{System Model and Assumptions} \label{Se:Sys_Model}

We consider a single source transmitting media sequences to
$M$ users/receivers over a shared wireless channel. The sender
maintains a separate buffer for each receiver, 
and is assumed to
always have data to transmit to each receiver.\footnote{This assumption is commonly referred to as the infinite backlog assumption.}  We consider a \emph{fluid} packet model that allows packet to be split, with the receiver reassembling fractional packets.
Each receiver has a playout
buffer at the receiving end, assumed to be infinite.  While in
reality this cannot be the case, it is nevertheless a reasonable
assumption considering the decreasing cost and size of memory, and
the fact that our system model allows holding costs to be assessed on packets in the
receiver buffers.
See Figure \ref{Fig:pro:system_diagram} for a diagram of the system.
\begin{figure}[htbp]
\centerline {
\includegraphics[width=4.0in]{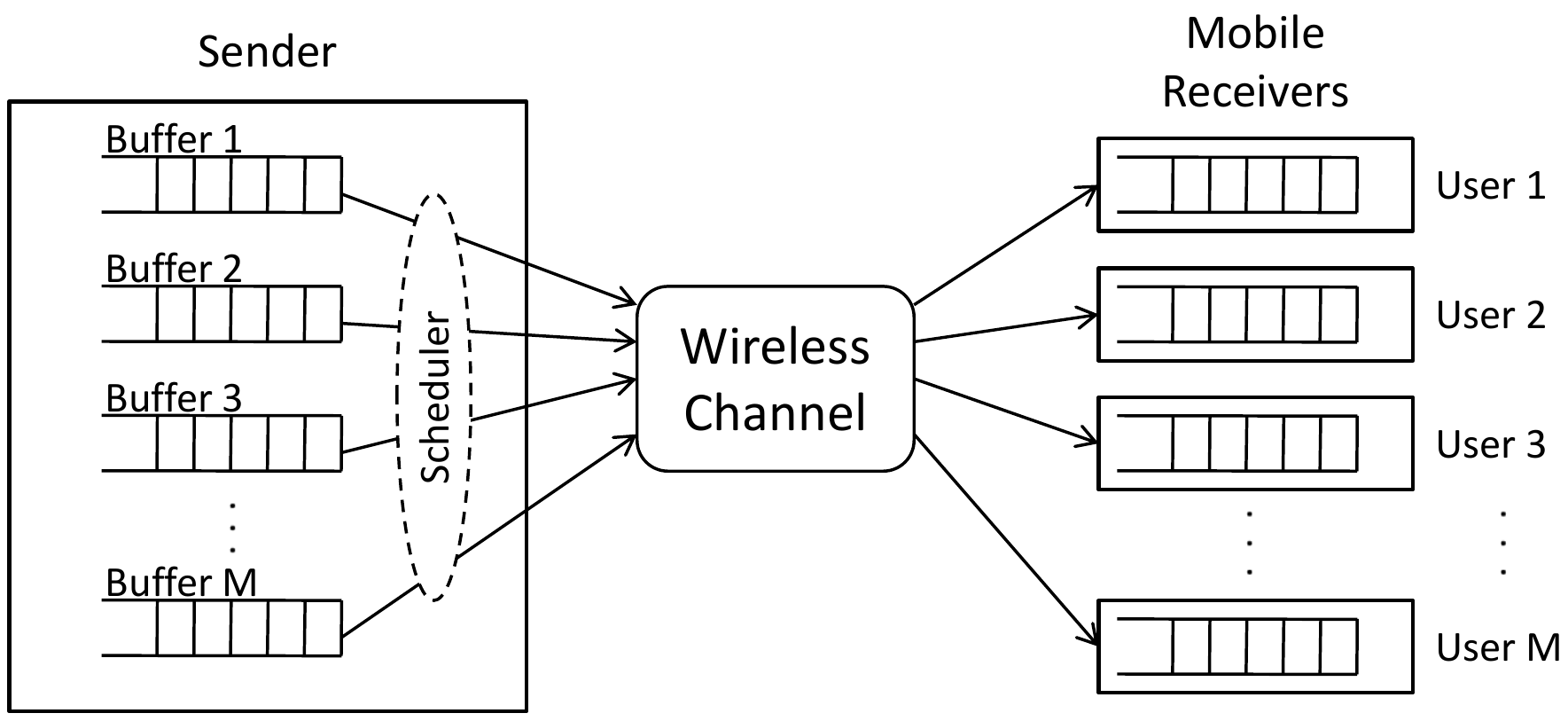}
} \caption{System model.}\label{Fig:pro:system_diagram}
\end{figure}

We consider time evolution in discrete steps, indexed backwards by
$n = N,N-1,\ldots,1$, with $n$ representing the number of slots
remaining in the time horizon. $N$ is the length of the time
horizon, and slot $n$ refers to the time interval $[n,n-1)$.

At the beginning of each time slot, the scheduler allocates some
amount of power (possibly zero) for transmission to each user. 
The
total power consumed in any one slot must not exceed the
fixed power constraint, $P$.
Following transmission and
reception in each slot, a certain number of packets are
removed/purged from each receiver buffer for playing.  The transmitter (or scheduler) knows precisely the
packet requirements of each receiver (i.e., the number of packets removed from the buffer) in each time slot.  This is justified
by the assumption that the transmitter knows the encoding
and decoding schemes used. We assume
that packets transmitted in slot $n$ arrive in time to be used for
playing in slot $n$, and that the users' consumption of packets in each
slot is constant, denoted by $\textbf{d}=\left(d^1,d^2,\ldots,d^M\right)$. 
This
latter assumption is less realistic, but may be justified if the
receiving buffers are drained at a constant rate at the MAC layer,
before packets are decoded by the media players at the application
layer. It is also worth noting that the same techniques we use in
this paper to analyze the 
constant drainage rate
case can be used to
examine the 
case
of time-varying drainage rates. We discuss the extension to the case of time-varying drainage rates 
further in Section \ref{Se:str:finite}. We also assume the receiver buffers are empty at the beginning of the time horizon, and that even when the channels are in their worst possible condition, the maximum power constraint
$P$ is sufficient to transmit enough packets to satisfy one time
slot's packet requirements
for every user. We discuss the relaxation of this assumption in Section \ref{Se:relaxation}.


%


In general, wireless channel conditions are time-varying. Adopting a block fading model, we assume
that the slot duration is within the channel coherence time such
that the channel conditions within a single slot are constant.  User $m$'s
channel condition in slot $n$ is modeled as a random
variable, $S_n^m$. 
We assume that the evolution of a given user's channel condition is independent of all other users' channel conditions and the transmitter's scheduling decisions. We also assume that the transmitter
learns all the channel states through a feedback channel at the beginning
of each time slot, prior to making the scheduling decisions.

We begin by modeling the evolution of
each user's channel condition 
as a finite-state ergodic homogeneous Markov process, $\left\{S_n^m\right\}_{n=N,N-1,\ldots,1}$ with state space ${\cal{S}}^m$.\footnote{Theorems \ref{Th:str:finite},  \ref{Th:pwl:fin}, \ref{Th:one:infinite}, 
\ref{Th:two_users:properties}, \ref{Th:two_users:structure}, and \ref{Th:two:infd} and their proofs remain valid as stated when each user's channel condition is given by a more general homogeneous Markov process that is not necessarily finite-state and ergodic.}
Namely, conditioned on the channel state, $S_n^m$, at time $n$, user $m$'s channel states at \emph{future} times ($n-1, n-2,\ldots$) are independent
of the channel states at \emph{past} times ($n+1,n+2,\ldots$). Note the somewhat unconventional notation that future times are indexed by lower epoch numbers, as
$n$ represents the number of slots
remaining in the time horizon. Modeling time backwards facilitates the analysis of the infinite horizon problems, as will be seen for example in Section \ref{Se:highSNR:infinite}.
It may also be the case that each user's channel condition is independent and identically distributed (IID) from slot to slot. When this is the case, we can often say more about the optimal transmission policy, as will be seen for example in Sections \ref{Se:calculation_low} and \ref{Se:highSNR:computation}.

Associated with each channel condition for a given user is a power-rate function.
If user $m$'s channel is in condition $s^m$, then the transmission of
$r$ units of data to user $m$ incurs a power consumption of $c^m(r,s^m)$.
This power-rate function $c^m(\cdot,s^m)$ is commonly assumed to be linear (in the low SNR regime) or
convex (in the high SNR regime).
In this paper, we consider power-rate functions that are linear or piecewise-linear convex, the latter of which can be used to approximate more general convex power-rate functions. We assume that sending data consumes a strictly positive amount of power, and therefore take the power-rate functions to be strictly increasing under all channel conditions.

The goal of this
study is to characterize the control laws that minimize the
transmission power and packet holding costs over a finite or infinite time horizon,
subject to tight underflow constraints and a maximum power
constraint in each time slot.

\subsection{Notation} \label{Se:notation}
Before proceeding, we introduce some notation. We define $\Real_+ := [0,\infty)$ and $\Nat := \left\{1,2,\ldots\right\}$. A single dot, as in $a \cdot b$, represents scalar multiplication. We use bold font to denote column vectors, such as $\textbf{w} = (w^1,w^2,\ldots,w^M)$. We include a transpose superscript whenever a vector is meant to be a row vector, such as $\textbf{w}^{\transpose}$. The notations $\textbf{w} \preceq \tilde{\textbf{w}}$ and $\textbf{w} \succeq \tilde{\textbf{w}}$ denote component-wise inequalities; i.e., $w^m \leq (\hbox{respectively,} \geq) ~\tilde{w}^m,~\forall m$.  Finally, we use the standard definitions of the meet and join of two vectors. Namely,
\begin{eqnarray*}
&\textbf{w} \wedge \tilde{\textbf{w}} &= \left(w^1,w^2,\ldots,w^M\right) \wedge \left(\tilde{w}^1,\tilde{w}^2,\ldots,\tilde{w}^M\right) \\
&&:=
\Bigl( \min\left\{w^1,\tilde{w}^1\right\},\min\left\{w^2,\tilde{w}^2\right\},\ldots,\min\left\{w^M,\tilde{w}^M\right\} \Bigr)~, \\
\hbox{and  }
&\textbf{w} \vee \tilde{\textbf{w}} &= \left(w^1,w^2,\ldots,w^M\right) \vee \left(\tilde{w}^1,\tilde{w}^2,\ldots,\tilde{w}^M\right) \\ &&:=
\Bigl( \max\left\{w^1,\tilde{w}^1\right\},\max\left\{w^2,\tilde{w}^2\right\},\ldots,\max\left\{w^M,\tilde{w}^M\right\} \Bigr)~.
\end{eqnarray*}

\subsection{Problem Formulation} \label{Se:Prob_Form}
We consider three problems.  Problem
(\textbf{P1}) is the finite horizon discounted expected cost
problem; 
Problem (\textbf{P2}) is the infinite
horizon discounted expected cost problem; and
Problem (\textbf{P3}) is the infinite
horizon average expected cost problem.
The three problems feature
the same information state, action space, system dynamics, and
cost structure, but different optimization criteria.

The information state at time $n$ is the pair
$(\textbf{X}_n,\textbf{S}_n)$, where the random vector \\ 
$\textbf{X}_n=(X^1_n, X^2_n, \cdots, X^M_n)$ 
denotes
the current receiver buffer queue lengths, and \\
$\textbf{S}_n=(S^1_n, S^2_n, \cdots, S^M_n)$ 
denotes
the channel conditions in slot $n$ (recall that $n$ is the number
of steps remaining until the end of the horizon).
%
The dynamics for the receivers' queues are governed by the simple
equation $\textbf{X}_{n-1}=\textbf{X}_n+\textbf{Z}_n-\textbf{d}$ at all times $n = N, N-1, \ldots,
1$, where $\textbf{Z}_n$ is a controlled random vector chosen by the scheduler at each time $n$ that represents the number of packets transmitted to each user in the $n^{th}$ slot.
%
At each time $n$, $\textbf{Z}_n$ must be chosen
to meet the peak power constraint:
\begin{eqnarray*} 
\sum\limits_{m=1}^M c^m(Z_n^m,S_n^m) \leq P~, 
\end{eqnarray*}
and the underflow constraints:  
\begin{eqnarray*}
X_n^m+Z_n^m \geq d^m~,~\forall m \in \{1,2,\ldots,M\}~. 
\end{eqnarray*}
Clearly, the scheduler cannot transmit a negative number of
packets to any user, so it must also be true that $Z_n^m \geq 0$ for all $m$.

We now present the optimization criterion for each problem. In
addition to the cost associated with power consumption from
transmission, we introduce holding costs on
packets stored in each user's playout buffer at the end of a
time slot. The holding costs associated with user $m$
in each slot are described by a convex, nonnegative, nondecreasing function, $h^m(\cdot)$, of the packets remaining in user $m$'s buffer following playout, with $\lim_{x \rightarrow \infty}h^m(x) = \infty$. We assume without loss of generality that $h^m(0)=0$. Possible holding cost models 
include a linear model, $h^m(x)=\hat{h}^m \cdot x$ for some positive constant $\hat{h}^m$, or a barrier-type function such as:
\begin{eqnarray*}
h^m(x)&:=&\left\{
\begin{array}{ll}
   0 , & \mbox{if } ~x \leq \mu \\
   \kappa \cdot (x-\mu) , & \mbox{if } ~x > \mu ~~(\kappa\mbox{ very large}) 
\end{array} \right.~, 
\end{eqnarray*}
\noindent which could represent a finite receiver buffer of length $\mu$.\footnote{Taking $\mu$ to be greater than the time horizon $N$ in the finite horizon expected cost problem is equivalent to not assessing any holding costs in Problem
(\textbf{P1}).} 

In Problem
(\textbf{P1}), we wish to find a transmission policy
$\boldsymbol{\pi}$ that minimizes ${J}_{N,\alpha}^\pi$, the finite horizon discounted expected cost
under policy $\boldsymbol{\pi}$,
defined as:
\begin{eqnarray*} 
{J}_{N,\alpha}^\pi:={\Expectation}^\pi
\Biggl\{\sum\limits_{t=1}^{N}
\sum\limits_{m=1}^M
\alpha^{N-t}\cdot \Bigl\{c^m
\bigl(Z_t^m,S_t^m \bigr)+
h^m\bigl(X_t^m+Z_t^m-d^m\bigr)\Bigr\}\mid{\cal F}_N \Biggr\} ~,
\end{eqnarray*}
where $0\leq\alpha \leq 1$ is the discount factor and ${\cal F}_N$
denotes all information available at the beginning of the time
horizon. For Problem (\textbf{P2}), the discount factor must satisfy $0\leq\alpha < 1$, and the infinite horizon discounted expected cost function for
minimization is defined as:
\begin{eqnarray*}
J_{\infty,\alpha}^\pi:=\lim\limits_{N\rightarrow\infty} {J}_{N,\alpha}^\pi~,
\end{eqnarray*}
For Problem (\textbf{P3}), the average expected cost function for
minimization is defined as:
\begin{eqnarray*}
J_{\infty,1}^\pi:=\limsup\limits_{N\rightarrow\infty} \frac{1}{N}{J}_{N,1}^\pi~.
\end{eqnarray*}

In 
all three cases, we allow the transmission policy $\boldsymbol{\pi}$
to be chosen from the set of all history-dependent randomized and deterministic
control laws, $\boldsymbol{\Pi}$ (see, e.g., \cite[Definition 2.2.3, pg. 15]{lerma2}).

Combining the constraints and criteria, we present the
optimization formulations for Problem (\textbf{P1}) (or
(\textbf{P2}) or
(\textbf{P3})):
\begin{eqnarray*}
&& \inf\limits_{\pi \in \Pi} {J}_{N,\alpha}^\pi ~~~ \Bigl(\mbox{or} ~~ \inf\limits_{\pi \in\Pi}{J}_{\infty,\alpha}^\pi \mbox{ or} ~~ \inf\limits_{\pi \in\Pi}{J}_{\infty,1}^\pi \Bigr) \\
\mbox{s.t.} && \sum\limits_{m=1}^M c^m\left(Z_n^m,S_n^m\right) \leq P,~w.p.1,~\forall n \\ 
&& Z_n^m \geq \max\left\{0,d^m-X_n^m\right\},~w.p.1,~\forall n 
,~\forall m \in \{1,2,\ldots,M\}.
%
\end{eqnarray*}

%
%

Problem (\textbf{P1}) may be solved using standard dynamic programming
(see, e.g., \cite{lerma2, shreve}).
The recursive dynamic programming equations
are given by:\footnote{As will be shown in the proofs of Theorems \ref{Th:pwl:infa} and \ref{Th:two:infa}, our model satisfies the measurable selection condition 3.3.3 of \cite[pg. 28]{lerma2}, justifying the use of $\min$ rather than $\inf$ in the dynamic programming equations.}
\begin{eqnarray} \label{Eq:gen:finDP}
V_n(\textbf{x},\textbf{s}) &=&
\min_{\textbf{z} \in {\cal A}^{\textbf{d}}(\textbf{x},\textbf{s})}
\left\{
\begin{array}{l}
\sum\limits_{m=1}^M \left\{ {c}^m\left(z^m,s^m\right)+{h}^m\left(x^m+z^m-{d}^m\right)\right\} \\
 +\alpha \cdot \Expectation \bigl[V_{n-1}(\textbf{x}+\textbf{z}-\textbf{d},\textbf{S}_{n-1})\bigm| \textbf{S}_n=\textbf{s}\bigr]
\end{array}
\right\} \nonumber \\
 &&~~~~~~~~~~~~~~~~~~~~~~~~~~~~~~~~~~~~~~
n=N,N-1,\ldots,1 \\
%
V_0(\textbf{x},\textbf{s}) &=& 0, ~~\forall \textbf{x} \in \Real_+^M,\forall
\textbf{s} \in {\cal{S}}:={\cal{S}}^1 \times {\cal{S}}^2 \times \ldots \times {\cal{S}}^M, \nonumber
\end{eqnarray}
%
%
where $V(\cdot,\cdot)$ is the value function or expected
cost-to-go,
and the action space is defined as:
\begin{eqnarray}\label{Eq:gen:actionspace}
{\cal A}^{\textbf{d}}(\textbf{x},\textbf{s}) := \biggl\{\textbf{z} \in \Real_{+}^M :
\begin{array}{l}
\textbf{z} \succeq \max\left\{\textbf{0}, \textbf{d}-\textbf{x} \right\} \hbox{ and} \\
\sum\limits_{m=1}^M c^m\left(z^m,s^m\right) \leq P
\end{array}
\biggr\},~~\forall \textbf{x} \in \Real_+^M,\forall
\textbf{s} \in {\cal{S}},
\end{eqnarray}
where the maximum in \eqref{Eq:gen:actionspace} is taken element-by-element (i.e., $z^m \geq \max\left\{0,d^m-z^m\right\}~\forall m$).
Note that our assumption that the maximum power constraint
$P$ is always sufficient to transmit enough packets to satisfy one time
slot's packet requirements
for every user (i.e., $\sum_{m=1}^M c^m\left(d^m,s^m\right) \leq P,~\forall \textbf{s} \in {\cal{S}}$) ensures that the action space ${\cal A}^{\textbf{d}}(\textbf{x},\textbf{s})$ is always non-empty. 



\subsection{Relation to Inventory Theory} \label{Se:Inv_Th}
The model outlined in Section \ref{Se:Sys_Model} corresponds closely to models used in
inventory theory. Borrowing that field's terminology, our
abstraction is a multi-period, single-echelon, multi-item, discrete-time inventory
model with random (linear or piecewise-linear convex) ordering costs, a budget constraint, and
deterministic demands. The items correspond to the streams of data packets, the random ordering costs to the random channel
conditions, the budget constraint to the power available in each
time slot, and the deterministic demands to the packet
requirements for playout.

To the best of our knowledge, this particular problem has not been
studied in the context of inventory theory, but similar problems
have been examined, and some of the techniques from the inventory theory literature are useful in analyzing our model. References \cite{fabian}-\nocite{kalymon}\nocite{kingsman}\nocite{kingsmanThesis}\nocite{magirou}\nocite{magirouLet}\nocite{golabi82}\cite{golabi85} all consider single-item
inventory models with linear ordering costs and random prices. The key result for
the case of deterministic demand of a single item with no resource
constraint is that the optimal policy is a base-stock
policy with different target stock levels for each price. Specifically, for each possible ordering price (translates
into channel condition in our context), there exists a critical
number such that the optimal policy is to fill the inventory
(receiver buffer) up to that critical number if the current level is
lower than the critical number, and not to order (transmit)
anything if the current level is above the critical number. Of the
prior work, Kingsman \cite{kingsman}, \cite{kingsmanThesis} is the
only author to consider a resource constraint, and he imposes a
maximum on the number of items that may be ordered in each slot.
The resource constraint we consider is of a different nature in
that we limit the amount of power available in each slot.
This is equivalent to a limit on the per slot budget
(regardless of the stochastic price realization), rather than a
limit on the number of items that can be ordered.

Of the related work on single-item inventory models with deterministic linear ordering costs and stochastic demand, \cite{fedzipII} and \cite{tayur} are the most relevant; in those studies, however, the resource constraint also amounts to a limit on the number of items that can be ordered in each slot, and is constant over time. References \cite{sobel_70}-\nocite{bensoussan_book}\cite{zahrn} consider single-item inventory models with deterministic piecewise-linear convex ordering costs and stochastic demand. The key result in this setup is that the optimal inventory level after ordering is a piecewise-linear nondecreasing function of the current inventory level (i.e., there are a finite number of target stock levels), and the optimal ordering quantity is a piecewise-linear nonincreasing function of the current inventory level.
Porteus \cite{porteus_90} refers to policies of this form as \emph{finite generalized base-stock policies}, to distinguish them from the superclass of \emph{generalized base-stock policies}, which are optimal when the deterministic ordering costs are convex (but not necessarily piecewise-linear), as first studied in \cite{karlin58}. Under a generalized base-stock policy, the optimal inventory level after ordering is a nondecreasing function of the current inventory level, and the optimal ordering quantity is a nonincreasing function of the current inventory level.

References \cite{evans}-\nocite{decroix}\nocite{shaoxiang}\cite{janakiraman} consider multi-item inventory systems under deterministic ordering costs, stochastic demand, and resource constraints. We discuss related results from these studies in more detail in Section \ref{Se:two_users}. 

We are not aware of any prior work on (i) single-item inventory models with random piecewise-linear convex ordering costs; (ii) exact computation of the critical numbers in any sort of finite generalized base-stock policy; or (iii) multi-item inventory models with random ordering costs and joint resource constraints. Therefore, not only is this connection between wireless transmission scheduling problems and inventory models novel, but the results we present in this paper also represent a contribution to the inventory theory literature.


%% file: lowSNR.tex
\section{Single Receiver with Linear Power-Rate Curves}\label{Se:lowSNR}
In this section, we analyze the finite 
horizon discounted expected cost problem 
when there is only a single receiver ($M=1$), and the power-rate functions under different channel conditions are linear. One such family of power-rate functions is shown in Figure \ref{Fig:low:power-rate}, where there are three possible channel conditions, and a different linear power-rate function associated with each channel condition. Note that due to the power constraint $P$ in each slot, the effective power-rate function is a two-segment piecewise-linear convex function under all channel conditions. We subsequently simplify our notation and use $c_s$ to denote the power consumption per unit of data transmitted when the channel condition is in state $s$. Because there is just a single receiver, we also drop the dependence of the functions and random variables on $m$. We defer the infinite horizon expected cost problems for this case until Section \ref{Se:highSNR:infinite}.
\begin{figure}[htbp]
\centerline {
\includegraphics[width=3.4in]{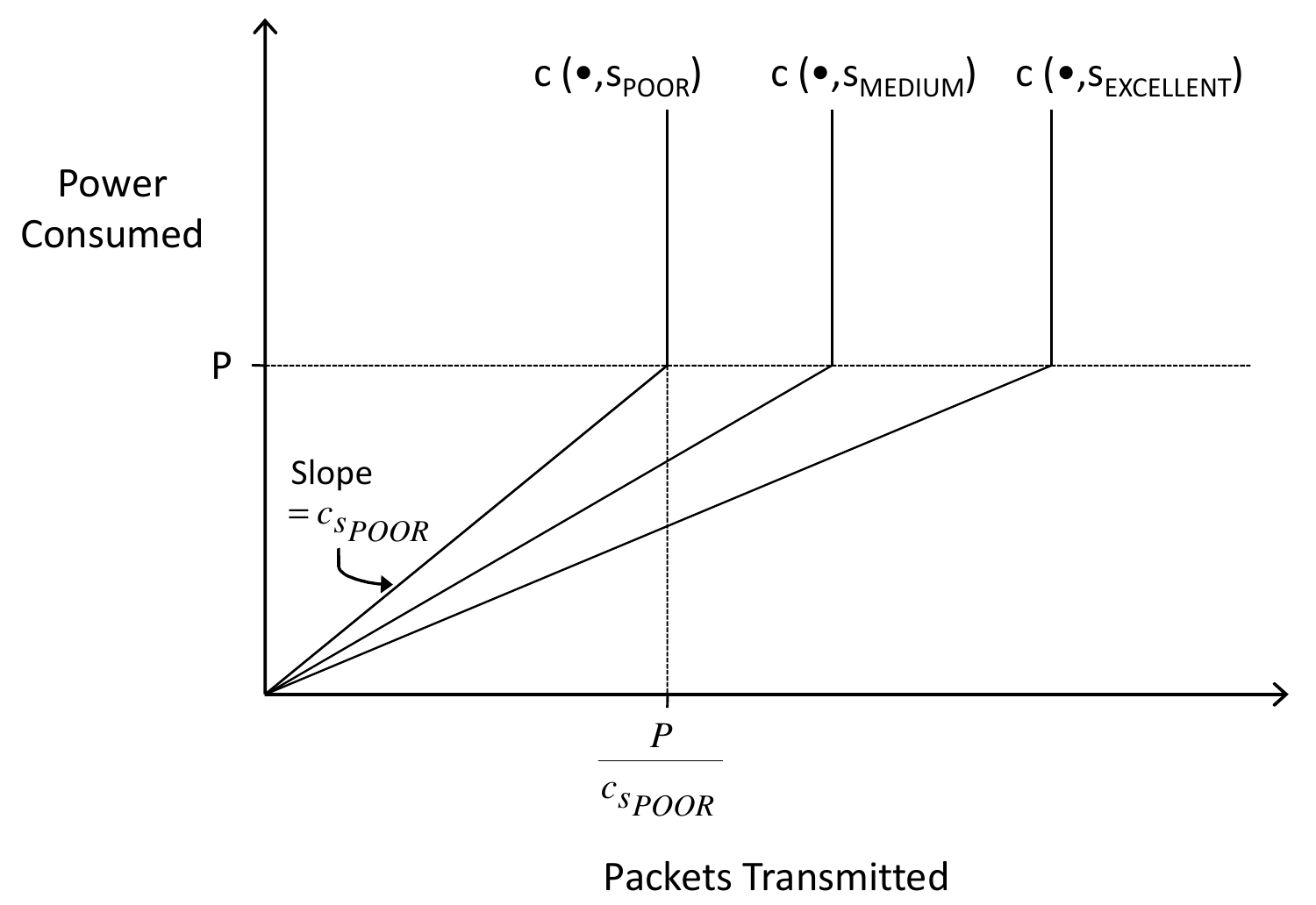}
} \caption{A family of linear power-rate functions. Due to the power constraint, the effective power-rate function, shown above for each of the three channel conditions, is a two-segment piecewise-linear convex function. When the channel condition is $s$, the slope of the first segment is $c_s$.}  \label{Fig:low:power-rate}
\end{figure}


We denote the ``best'' and ``worst'' channel
conditions by $s_{\best}$ and $s_{\worst}$,
respectively, and denote the slopes of the power-rate functions under these respective conditions by $c_{\min}$ and $c_{\max}$. That is,
  \begin{eqnarray*}
  0<c_{s_{\best}}=c_{\min}
  :=\min_{s \in {\cal{S}}}\{c_s\} \leq \max_{s \in {\cal{S}}}\{c_s\}
  =:c_{\max}=
  c_{s_{\worst}}  \leq \frac{P}{d}.
  \end{eqnarray*}
%
%

With these notations in place, the dynamic program \eqref{Eq:gen:finDP} for Problem (\textbf{P1}) becomes:
\begin{eqnarray}\label{Eq:low:DP}
V_n(x,s)
&=&
\min\limits_{\max(0,d-x)\leq z \leq \frac{P}{c_s}}
\left\{
\begin{array}{l}
c_s \cdot z + h(x+z-d) \\
+ \alpha \cdot \Expectation \bigl[V_{n-1}(x+z-d,S_{n-1}) \bigm| S_n =s\bigr]
\end{array}
\right\} \label{Eq:z_act} \\
&=&
\min\limits_{\max(x,d)\leq y \leq x + \frac{P}{c_s}}
\left\{
\begin{array}{l}
c_s \cdot (y-x) + h(y-d) \\
+ \alpha \cdot \Expectation \bigl[V_{n-1}(y-d,S_{n-1}) \bigm| S_n = s \bigr]
\end{array}
\right\}\label{Eq:y_act} \\
&=& -c_s \cdot x + \min\limits_{\max(x,d)\leq y \leq x + \frac{P}{c_s}}
\Bigl\{ g_n(y,s)
\Bigr\},
~~~n=N,N-1,\ldots,1~, \nonumber \\
V_0(x,s) &=& 0,~\forall x \in \Real_+, \forall s \in {\cal{S}}, \nonumber
\end{eqnarray}
where $g_n(y,s):=c_s \cdot y + h(y-d) + \alpha \cdot \Expectation
\bigl[V_{n-1}(y-d,S_{n-1})\mid S_n=s\bigr]$.
Here, the transition from (\ref{Eq:z_act}) to (\ref{Eq:y_act}) is done by a change of variable in the action space from $Z_n$ to $Y_n$, where $Y_n = X_n + Z_n$. The controlled random variable $Y_n$ represents the queue length of the receiver buffer \emph{after} transmission takes place in the $n^{th}$ slot, but \emph{before} playout takes place (i.e., before $d$ packets are removed from the buffer). The restrictions on the action space, $\max(x,d)\leq y \leq x + \frac{P}{c_s}$, ensure: (i) a nonnegative number of packets is transmitted; (ii) there are at least $d$ packets in the receiver buffer following transmission, in order to satisfy the underflow constraint; and (iii) the power constraint is satisfied.

\subsection{Structure of Optimal Policy} 
\label{Se:str:finite}
With the above change of variable in the the action space,
%
%
the expected cost-to-go at time $n$, $V_n(x,s)$, depends on the current buffer level, $x$, only through the fixed term $-c_s \cdot x$ and the action space
; i.e., 
the function $g_n$ does not depend on $x$.
This separation allows us to leverage the inventory theory techniques of showing ``single critical number'' or ``base-stock'' policies, which date as far back as \cite{glicksberg}. The following theorem gives the structure of the optimal transmission policy for the finite horizon discounted expected cost problem.

%

\begin{theorem} \label{Th:str:finite}
For every $n \in
\{1,2,\ldots,N\}$ and $s \in \cal{S}$, define the critical number 
\begin{eqnarray*}
b_n(s) := \min \left\{\hat{y} \in
[d,\infty):~g_n(\hat{y},s)=\min\limits_{y\in [d,\infty)}
g_n(y,s) \right\}~.
\end{eqnarray*}
Then, for Problem (\textbf{P1}) in the case of a single receiver with linear power-rate curves,  the optimal buffer level after transmission 
with $n$ slots remaining is given by:
\begin{eqnarray} \label{Eq:str:y_star}
y_n^*(x,s):=\left\{
\begin{array}{ll}
   x , & \mbox{if } ~x \geq b_n(s) \\
   b_n(s) , & \mbox{if } ~b_n(s)-\frac{P}{c_s}\leq x < b_n(s) \\
    x + \frac{P}{c_s} , & \mbox{if } ~x < b_n(s) - \frac{P}{c_s}  \\
\end{array} \right.,
\end{eqnarray}
or, equivalently, the optimal number of packets to transmit in slot $n$ is given by:
\begin{eqnarray} \label{Eq:str:z_star}
z_n^*(x,s):=\left\{
\begin{array}{ll}
   0 , & \mbox{if } ~x \geq b_n(s) \\
   b_n(s)-x , & \mbox{if } ~b_n(s)-\frac{P}{c_s}\leq x < b_n(s) \\
   \frac{P}{c_s} , & \mbox{if } ~x < b_n(s) - \frac{P}{c_s}  \\
\end{array} \right..
\end{eqnarray}
\noindent Furthermore, for a fixed $s$, $b_n(s)$ is nondecreasing 
in $n$:
\begin{eqnarray} \label{Eq:str:inc_n}
N \cdot d
\geq b_N(s) \geq b_{N-1}(s) \geq \ldots \geq b_1(s)=d~.
\end{eqnarray}
If, in addition, the channel condition is independent and identically distributed from slot to slot, then for a fixed $n$, $b_n(s)$ is nonincreasing 
in $c_s$; i.e., for arbitrary $s^1,s^2 \in {\cal{S}}$ with $c_{s^1} \leq c_{s^2}$, we have:
\begin{eqnarray} \label{Eq:str:inc_c}
n \cdot d \geq b_n(s_{\best}) \geq b_n({s^1}) \geq b_n(s^2) \geq b_n(s_{\worst}) =
d~.
\end{eqnarray}
\end{theorem}
\medskip

The optimal transmission policy in Theorem \ref{Th:str:finite} is
a modified base-stock policy. At time $n$, for each possible
channel condition realization $s$, the critical number $b_n(s)$
describes the target 
number of packets to have in the user's buffer
after transmission in the $n^{th}$ slot. If that number of packets
is already in the buffer, then it is optimal to not transmit any
packets; if there are fewer than the target 
and the available power is
enough to transmit the difference, then it is optimal to do so;
and if there are fewer than the target 
and the available power is not
enough to transmit the difference, then the sender should use the
maximum power to transmit. See Figure \ref{Fig:str:opt_policy} for
diagrams of the optimal policy.

\begin{figure}[htbp]
\centerline {
\includegraphics[width=4.5in]{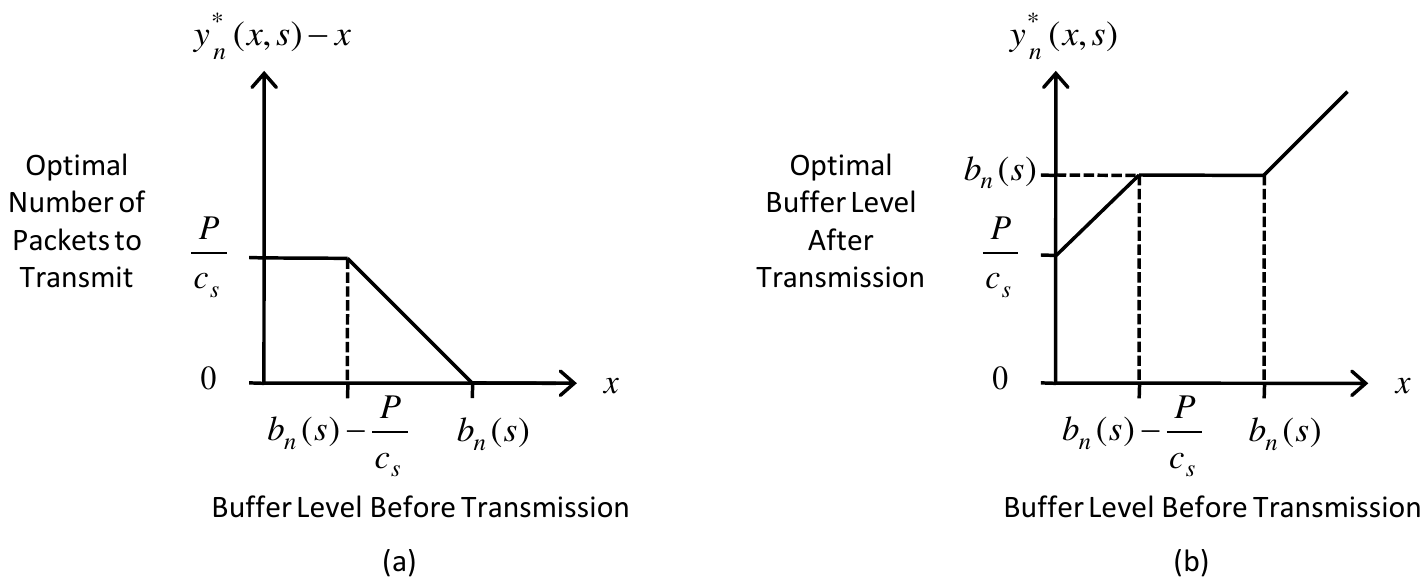}
} \caption{Optimal policy in slot $n$ when the state is $(x,s)$.
(a) depicts the optimal transmission quantity, and (b) depicts the
resulting number of packets available for playout in slot
$n$.}\label{Fig:str:opt_policy}
\end{figure}

Details of the proof of Theorem \ref{Th:str:finite} are included in Appendix A.
The key realization is that for all $n$ and all $s$, $g_n(\cdot,s):\left[d,\infty\right) \rightarrow \Real_+$ is a convex function in $y$, with $\lim_{y\rightarrow\infty}{g_n(y,s)} = \infty $. Thus, for all $n$ and all $s$, $g_n(\cdot,s)$ has a global minimum $b_n(s)$, the target 
number of packets to have in the buffer following transmission in the $n^{th}$ slot. The key idea to show (\ref{Eq:str:inc_n}) is to fix $s \in {\cal{S}}$, view $g_n(y,s)$ as a function of $y$ and $n$, say $f(y,n)$, and show that the function $f(\cdot,\cdot)$ is submodular. From the proof, one can also see that if we relax the stationary (time-invariant) deterministic demand assumption to a nonstationary (time-varying) deterministic demand sequence, $\left\{d_N,d_{N-1},\ldots,d_1\right\}$ (with $d_n \leq \frac{P}{c_{\max}}$ for all $n$), then the structure of the optimal policy is still as stated in (\ref{Eq:str:y_star}). If 
the channel is IID, then the following statement, analogous to (\ref{Eq:str:inc_c}), is true for arbitrary $s^1,s^2 \in {\cal{S}}$ with $c_{s^1} \leq c_{s^2}$:
\begin{eqnarray}  \label{Eq:str:inc_n_nonstationary}
\sum_{i=1}^n d_i \geq b_n(s_{\best}) \geq b_n(c_{s^1}) \geq b_n(c_{s^2}) \geq b_n(s_{\worst}) =
d_n~,~\forall n \in \left\{1,2,\ldots,N\right\}~.
\end{eqnarray}
\noindent However, (\ref{Eq:str:inc_n}), the monotonicity of critical numbers over time for a fixed channel condition, is not true in general under nonstationary deterministic demand. As one counterexample, (\ref{Eq:str:inc_n_nonstationary}) says that under an IID channel, the critical numbers for the worst possible channel condition are equal to the single period demands. Therefore, if the demand sequence is not monotonic, the sequence of critical numbers, $\left\{b_n\left(s_{\worst}\right) \right\}_{n=1,2,\ldots,N}$, is not monotonic.

\input{calculation_low}

%% file: calculation_low.tex
\subsection{Computation of the Critical Numbers}\label{Se:calculation_low}
In this section, we consider the special case where the channel condition is independent and identically distributed from slot to slot,
the holding cost function is linear (i.e., $h(x)=h \cdot x$ for some $h \geq 0$), and the following technical condition is satisfied: for each possible channel condition 
$s$, $\frac{P}{c_s}=l \cdot d$ for some $l \in \Nat$; i.e., the maximum number of packets that can be transmitted in any slot covers exactly the playout requirements of some integer number of slots. Under these 
three assumptions, we can completely characterize the optimal transmission policy.

\begin{theorem} \label{Th:one:finite}
Define the threshold $\gamma_{n,j}$ for $n\in\{1,2,\ldots,N\}$
and $j\in\Nat$ recursively, as follows:
\begin{itemize}
\item[(i)] If $j=1$, $\gamma_{n,j}=\infty$; \item[(ii)] If $j>n$,
$\gamma_{n,j}=0$; \item[(iii)] If $2\leq j \leq n$,
\end{itemize}
\begin{eqnarray}\label{Eq:one:iii}
 \gamma_{n,j} = -h + \alpha\cdot\left(
 \begin{array}{l}
    \sum\limits_{s:~c_s\geq \gamma_{n-1,j-1}} p(s)\cdot\gamma_{n-1,j-1}
    + \sum\limits_{s:~c_s< \gamma_{n-1,j-1}} p(s)\cdot c_s \\
~~~~~+ \sum\limits_{s:~c_s < \gamma_{n-1,j-1+L(s)}}
p(s)\cdot\left[\gamma_{n-1,j-1+L(s)}-c_s\right]
\end{array}
\right),
\end{eqnarray}
\medskip

\noindent where $p(s)$ is the probability of the channel being in
state $s$ in a time slot, and $L(s):=\frac{P}{d \cdot c_s}$.
%
%
%
%
%
%
%
%
%
%
%
%
%
%
%
For each $n\in\{1,2,\ldots,N\}$ and $s\in{\cal{S}}$, if
$\gamma_{n,j+1}\leq c_s< \gamma_{n,j}$, define
$b_n(s) := j \cdot d$.
%
%
%
%
%
%
%
%
The optimal control strategy for Problem (\textbf{P1}) is then given by
$\boldsymbol{\pi}^*=\bigl\{y_N^*,y_{N-1}^*,\ldots,y_1^*\bigr\}$, where
\begin{eqnarray}\label{Eq:one:optimal_control}
y_n^*(x,s):=\left\{
\begin{array}{ll}
   x , & \mbox{if } ~x \geq b_n(s) \\
   b_n(s) , & \mbox{if } ~b_n(s)-\frac{P}{c_s}\leq x < b_n(s) \\
    x + \frac{P}{c_s} , & \mbox{if } ~x < b_n(s) - \frac{P}{c_s}  \\
\end{array} \right. .
\end{eqnarray}
\end{theorem}
\vspace{5 mm}
\noindent Note that with $n$ slots remaining, $0= \gamma_{n,n+1} \leq
\gamma_{n,n} \leq \gamma_{n,n-1} \leq~\ldots~\leq \gamma_{n,2}
\leq \gamma_{n,1} = \infty$, so $b_n(s)$ is well-defined.

Compared to using standard numerical techniques to \emph{approximately} solve the dynamic program and find a near-optimal
policy, the above result not only sheds more insight on the
structural properties of the problem and its \emph{exactly}-optimal solution, but
also offers a computationally simpler method. In particular, the
optimal policy is completely characterized by the thresholds $\left\{\gamma_{n,j}\right\}_{{n\in \{1,2,\ldots,N\},~j\in \Nat}}$.
Calculating these thresholds recursively, as
described in Theorem \ref{Th:one:finite}, requires $O(N^2 \left|{\cal S}\right|)$ operations, which is considerably simpler
from a computational standpoint than approximately solving the dynamic
program \cite{rust,chow}.



To prove Theorem \ref{Th:one:finite}, we show by backwards induction that it is worse to transmit either
fewer or more packets than the number suggested by the policy
$\boldsymbol{\pi}^*$.
The detailed proof 
is omitted, as Theorem \ref{Th:one:finite} is a special case of Theorem \ref{Th:pwl_calc};
however, we discuss some intuition behind the proof and the thresholds here. 

The reason for the technical condition
regarding the maximum number of packets that can be transmitted in any slot is as follows. The optimal action at all times (in general, without the technical condition) is either to transmit enough packets to fill the buffer up to a level satisfying the playout requirements of some number of future slots, or to transmit at maximum power. When the technical condition is satisfied, transmitting at maximum power also results in filling the buffer up to a level satisfying the playout requirements of some number of future slots. Thus, under the optimal policy, all realizations 
result in the buffer level at the end of every time slot being some integer multiple of the demand, $d$. This fact makes it easier to compute the thresholds $\left\{\gamma_{n,j}\right\}_{{n\in \{1,2,\ldots,N\},~j\in \Nat}}$.

An intuitive explanation of the recursion \eqref{Eq:one:iii} is as follows. The threshold
$\gamma_{n,j}$ may be interpreted as the per packet power cost at
which, with $n$ slots remaining in the horizon, the expected
cost-to-go of transmitting packets to cover the user's playout
requirements for the next $j-1$ slots is the same as the expected
cost-to-go of transmitting packets to cover the user's
requirements for the next $j$ slots. That is, $\gamma_{n,j}$ should satisfy:
\begin{eqnarray*}
\alpha \cdot \Expectation\left[V_{n-1}\Bigl((j-1)\cdot d,S_{n-1}\Bigr) \right] + \gamma_{n,j} \cdot d + h \cdot d =  \alpha \cdot \Expectation\left[V_{n-1}\Bigl((j-2)\cdot d,S_{n-1}\Bigr) \right],
\end{eqnarray*}
which is equivalent to:
\begin{align}
&\gamma_{n,j} \nonumber \\
&= -h + \frac{\alpha}{d} \cdot \Expectation\Biggl[
V_{n-1}\Bigl((j-2)\cdot d,S_{n-1}\Bigr) - V_{n-1}\Bigl((j-1)\cdot d,S_{n-1}\Bigr)
\Biggr] \label{Eq:thresh_intuition0} \\
&= -h + \frac{\alpha}{d} \cdot \sum\limits_{s \in {\cal{S}}} p(s) \cdot \Biggl[V_{n-1}\Bigl((j-2)\cdot d,s\Bigr)-V_{n-1}\Bigl((j-1)\cdot d,s\Bigr)\Biggr] \nonumber 
\end{align}
\begin{align}
&= -h + \frac{\alpha}{d} \cdot \left\{
\begin{array}{l}
\smashoperator[r]{\sum\limits_{s:~b_{n-1}(s) \leq (j-2) \cdot d}}~~~p(s) \cdot \left\{-h \cdot d + \alpha \cdot \Expectation\left[
\begin{array}{l}
V_{n-2}\Bigl((j-3)\cdot d,S_{n-2}\Bigr) \\
- V_{n-2}\Bigl((j-2)\cdot d,S_{n-2}\Bigr)
\end{array}
\right] \right\} \\
+{}\smashoperator[r]{\sum\limits_{s:~(j-2) \cdot d < b_{n-1}(s) \leq \bigl(j-2+L(s)\bigr) \cdot d}}~~~p(s) \cdot c_s \cdot d  \\
+{}\smashoperator[r]{\sum\limits_{s:~b_{n-1}(s) > \bigl(j-2+L(s)\bigr) \cdot d}}~~~p(s) \cdot \left\{
-h \cdot d
+ \alpha \cdot \Expectation\left[
\begin{array}{l}
V_{n-2}\Bigl(\bigl(j-3+L(s)\bigr)\cdot d,S_{n-2}\Bigr) \\
- V_{n-2}\Bigl(\bigl(j-2+L(s)\bigr)\cdot d,S_{n-2}\Bigr)
\end{array}
\right]\right\}
\end{array}
\right\} \label{Eq:thresh_intuition1} \\
&= -h + \alpha \cdot \left\{
\begin{array}{l}
{\sum\limits_{s:~b_{n-1}(s) \leq (j-2) \cdot d}}~p(s) \cdot \gamma_{n-1,j-1}\\
+
{\sum\limits_{s:~(j-2) \cdot d < b_{n-1}(s) \leq \bigl(j-2+L(s)\bigr) \cdot d}}~p(s) \cdot c_s  \\
+~~~~{\sum\limits_{s:~b_{n-1}(s) > \bigl(j-2+L(s)\bigr) \cdot d}}~~~~~p(s) \cdot \gamma_{n-1,j-1+L(s)}
\end{array}
\right\} \label{Eq:thresh_intuition2} \\
&= -h + \alpha\cdot\left\{
 \begin{array}{l}
    \sum\limits_{s:~c_s\geq \gamma_{n-1,j-1}} p(s)\cdot\gamma_{n-1,j-1} \\
    + \sum\limits_{s:~\gamma_{n-1,j-1+L(s)} \leq c_s< \gamma_{n-1,j-1}} p(s)\cdot c_s \\
~~~~~+ \sum\limits_{s:~c_s < \gamma_{n-1,j-1+L(s)}}
p(s)\cdot \gamma_{n-1,j-1+L(s)}
\end{array}
\right\}. \label{Eq:thresh_intuition3}
\end{align}
Here, \eqref{Eq:thresh_intuition1} follows from the structure of the optimal control action \eqref{Eq:str:y_star}.
If the channel condition $s$ in the $(n-1)^{st}$ slot is such that $b_{n-1}(s) \leq (j-2) \cdot d$, then no packets are transmitted when the starting buffer level is either $(j-2) \cdot d$ or $(j-1) \cdot d$, and the respective buffer levels at the beginning of slot $n-2$ are $(j-3) \cdot d$ and $(j-2) \cdot d$. The instantaneous costs resulting from the two starting buffer levels differ by $-h \cdot d$. When $(j-2) \cdot d < b_{n-1}(s) \leq \bigl(j-2+L(s)\bigr) \cdot d$, the power constraint is not tight starting from $(j-1) \cdot d$, so the buffer level after transmission is the same starting from $(j-2) \cdot d$ or $(j-1) \cdot d$. The instantaneous costs resulting from the two starting buffer levels differ by $c_s \cdot d$, as an extra $d$ packets are transmitted if the starting buffer is $(j-2) \cdot d$. Finally, when $b_{n-1}(s) > \bigl(j-2+L(s)\bigr) \cdot d$, the power constraint is tight starting from both $(j-2) \cdot d$ and $(j-1) \cdot d$. Therefore, the instantaneous cost difference is $-h \cdot d$, and the respective buffer levels at the beginning of slot $n-2$ are $(j-3+L(s)) \cdot d$ and $(j-2+L(s)) \cdot d$.
Equation \eqref{Eq:thresh_intuition2} follows from \eqref{Eq:thresh_intuition0}, with $n-1,j-1$ substituted for $n,j$, and \eqref{Eq:thresh_intuition3} follows from the definition that $b_n(s)=j \cdot d$ if $\gamma_{n,j+1}\leq c_s < \gamma_{n,j}$.


Comparing the threshold $\gamma_{n,j}$ defined in (\ref{Eq:one:iii}) to the corresponding threshold in the
unrestricted (no power constraint) single user problem
\cite{kingsman, golabi85}, the only difference is the
third term of the right-hand side of (\ref{Eq:one:iii}):
\begin{eqnarray*}
\alpha \cdot \sum_{\{s:~c_s < \gamma_{n-1,j-1+L(s)}\}} p(s) \cdot
\left[\gamma_{n-1,j-1+L(s)}-c_s\right],
\end{eqnarray*}
which is absent in the unrestricted case. For all
$n\in\{1,2,\ldots,N\}$ and $j\in\Nat$, this term is nonnegative.
Thus, for a fixed $n$ and $j$, the threshold in the restricted
case is at least as high as the corresponding threshold in the
unrestricted case. It follows that the optimal stock-up level
$b_n(s)$ is also at least as high in the restricted case for all
$n\in\{1,2,\ldots,N\}$ and $s\in{\cal{S}}$. The intuition behind
this difference is that the sender should transmit more packets
under the same (medium) conditions, because it is not able to take
advantage of the best channel conditions to the same extent due to
the power constraint.
\subsection{Sufficient Conditions for Equivalence with the Single Deadline Problem}
In 
\cite[Section III-D]{fu}, Fu, Modiano, and Tsitsiklis consider the related single user problem of transmitting a given amount of data with minimum energy by a fixed deadline. They also represent the fading channel
by a linear power-rate function with a different slope in each channel condition, and consider a power constraint $P$ in each slot.
%
There is just a single explicit underflow constraint (the deadline) in their problem; 
however, because the terminal cost is set to $\infty$ if all the data is not transmitted by the deadline, the scheduler must transmit enough data in each slot so that it can still complete the job if the channel is in the worst possible condition in all subsequent slots. Thus, if $d_{\tot}$ is the total amount of data that must be sent by the deadline and $d_{\worst}$ is the amount that can be sent in a slot under the worst channel condition, the transmitter must have sent at least $d_{\tot}-d_{\worst}$ packets by the beginning of the last slot, at least $d_{\tot}-2 \cdot d_{\worst}$ packets by the beginning of the second to last slot, and so forth.\footnote{An unstated assumption in the formulation in \cite[Section III-D]{fu} is that $d_{\worst}$ times the horizon length must be at least as large as $d_{\tot}$.} So there are in fact implicit constraints on how much data must be transmitted by the end of slots $N-\left\lceil \frac{d_{\tot}}{d_{\worst}} \right\rceil + 1$, $N-\left\lceil \frac{d_{\tot}}{d_{\worst}} \right\rceil + 2$, $\ldots$, $N-2$, $N-1$. With this interpretation, we believe that our Theorem \ref{Th:one:finite} is equivalent to Theorem 3 and its corollary in \cite{fu} in the special case that, in addition to the hypotheses of our Theorem \ref{Th:one:finite}, $\alpha=1$, $h=0$, and $L\left(s_{\worst}\right)=1$. For, when these conditions are met, the implicit constraints in \cite{fu} coincide exactly with the explicit underflow constraints in our problem. Of course, when these three conditions are not satisfied, the two problems are quite different. For a more detailed comparison of these two problems, see \cite{chapter}.

\subsection{Intuitive Takeaways on the Role of the Strict Underflow Constraints}
As mentioned earlier, the main idea of energy-efficient communication over a fading channel via opportunistic scheduling is to minimize power consumption by transmitting more data when the channel is in a ``good'' state, and less data when the channel is in a ``bad'' state.
However, in order to comply with the underflow or deadline constraints, the transmitter may be forced to send data under poor channel conditions.

One intuitive takeaway from the analysis is that it is better to anticipate the need to comply with these constraints in future slots by sending more packets (than one would without the deadlines) under ``medium'' channel conditions in earlier slots. Doing so is a way to manage the risk of being stuck sending a large amount of data over a poor channel to meet an imminent deadline constraint. Another intuitive takeaway is that the closer the deadlines and the more deadlines it faces, the less ``opportunistic'' the scheduler can afford to be. In summary, both the underflow constraints and the power constraints shift the definition of what constitutes a ``good'' channel, and how much data to send accordingly. For more detailed comparisons of single-receiver opportunistic scheduling problems highlighting the role of the deadline constraints, see \cite{chapter}.

%% file: highSNR.tex
\section{Single Receiver with Piecewise-Linear Convex Power-Rate Curves}\label{Se:highSNR}
In this section, we analyze Problems (\textbf{P1}), (\textbf{P2}), and (\textbf{P3}) 
when there is only a single receiver ($M=1$), and the power-rate functions under different channel conditions are piecewise-linear convex. Note that this is a generalization of the case considered in Section \ref{Se:lowSNR}. 
%

We assume without loss of generality that under each channel condition $s$, the power-rate function has $K+1$ segments, and thus the power consumed in transmitting $z$ packets under channel condition $s$ can be represented as follows:
\begin{eqnarray*}
&&c(z,s) = z \cdot \tilde{c}_0(s) + \sum\limits_{k=0}^{K-1} \Bigl\{\bigl(\tilde{c}_{k+1}(s)-\tilde{c}_k(s) \bigr)\cdot \max\bigl\{z-\tilde{z}_k(s), 0\bigr\} \Bigr\}~, \hbox{ where } \\
&&0 < \tilde{c}_0(s) \leq \tilde{c}_1(s) \leq \cdots \leq \tilde{c}_K(s)~, \hbox{ and } \\
&&0 =\tilde{z}_{-1}(s) < \tilde{z}_0(s) < \tilde{z}_1(s) < \cdots < \tilde{z}_{K-1}(s)<\tilde{z}_{K}(s)=\infty~.
\end{eqnarray*}
The terms $\left\{\tilde{c}_k(s)\right\}_{k \in \left\{0,1,\ldots,K \right\}}$ represent the slopes of the segments of $c(\cdot,s)$, and the terms $\left\{\tilde{z}_k(s)\right\}_{k \in \left\{0,1,\ldots,K-1 \right\}}$ represent the points at which the slopes of $c(\cdot,s)$ change. An example of a family of such power-rate functions is shown in Figure \ref{Fig:power-rate}. For each channel condition $s \in {\cal S}$, we define the maximum number of packets that can be transmitted without exceeding the per slot power constraint $P$ as:
\begin{eqnarray*}
\tilde{z}_{\max}(s):=\{z: c(z,s)=P\}~.
\end{eqnarray*}
Note that $\tilde{z}_{\max}(s)$ is well-defined due to the strictly increasing nature of $c(\cdot,s)$. Recall that we assume $\tilde{z}_{\max}(s)\geq d,~\forall s \in {\cal S}$. We also assume without loss of generality that $\tilde{z}_{\max}(s)>\tilde{z}_{K-1}(s),~\forall s \in {\cal S}$.
\begin{figure}[htbp]
\centerline {
\includegraphics[width=3.4in]{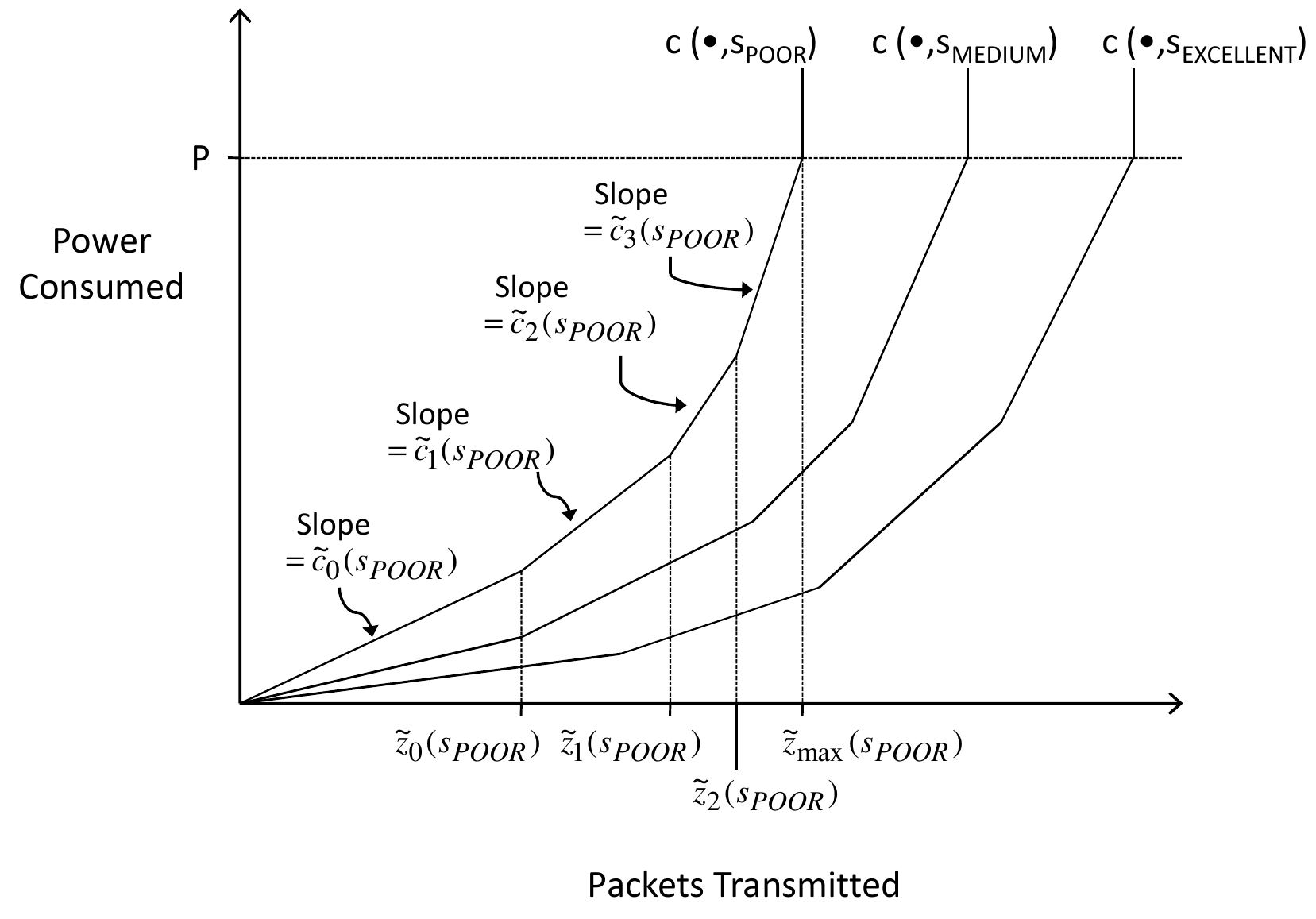}
} \caption{A family of piecewise-linear convex power-rate functions. Like Figure \ref{Fig:low:power-rate}, we incorporate the power constraint into each curve to show the effective power-rate curve. As an example, the power-rate function $c(\cdot,s_{POOR})$ is completely characterized by the sequence of slopes $\left\{\tilde{c}_k(s_{POOR})\right\}_{k \in \left\{0,1,2,3\right\}}$ and the sequence of points where the slopes change $\left\{\tilde{z}_k(s_{POOR})\right\}_{k \in \left\{0,1,2 \right\}}$. The maximum number of packets that can be transmitted in a slot when the channel condition is $s_{POOR}$ is $\tilde{z}_{\max}(s_{POOR})$.}  \label{Fig:power-rate}
\end{figure}

In this case, the dynamic program \eqref{Eq:gen:finDP} for Problem (\textbf{P1}) becomes:
\begin{eqnarray}\label{Eq:high:DP}
V_n(x,s)
&=&
\min\limits_{\bigl\{\max(0,d-x) \leq z \leq \tilde{z}_{\max}(s)\bigr\}}
\left\{
\begin{array}{l}
c(z,s) + h(x+z-d) \\
+ \alpha \cdot \Expectation \bigl[V_{n-1}(x+z-d,S_{n-1}) \bigm| S_n =s\bigr]
\end{array}
\right\} \nonumber \\
&=& \min\limits_{\bigl\{\max(0,d-x) \leq z \leq \tilde{z}_{\max}(s)\bigr\}}
\Bigl\{
c(z,s) + \tilde{g}_n(x+z,s) 
\Bigr\},~~
n=N,N-1,\ldots,1  \\
V_0(x,s) &=& 0,~\forall x \in \Real_+, \forall s \in {\cal{S}}~, \nonumber
\end{eqnarray}
where $\tilde{g}_n(y,s):=h(y-d) + \alpha \cdot \Expectation \left[V_{n-1}(y-d,S_{n-1}) \middle|S_n=s\right]$.

\subsection{Structure of Optimal Policy for the Finite Horizon Discounted Expected Cost Problem}
We showed in Theorem \ref{Th:str:finite} that the the optimal transmission policy to a single receiver in the case of linear power-rate curves  is
a modified base-stock policy characterized by a single critical level for each channel condition.
In this section, we generalize this result to the case of piecewise-linear power-rate curves, and show that
the optimal receiver buffer level after transmission (respectively, optimal number of packets to transmit) is no longer a three-segment piecewise-linear nondecreasing (respectively, nonincreasing) function of the starting buffer level as in Figure \ref{Fig:str:opt_policy}, but a more general piecewise-linear nondecreasing (respectively, nonincreasing) function.

%

\label{Se:highSNR:finite}
\medskip

\begin{theorem} \label{Th:pwl:fin}
In Problem (\textbf{P1}) with a single receiver under piecewise-linear convex power-rate curves, for every $n \in \{1,2,\ldots,N\}$ and $s \in \cal{S}$, there exists a nonincreasing sequence of critical numbers $\bigl\{b_{n,k}(s)\bigr\}_{k \in \left\{0,1,\ldots,K\right\}}$ such that the optimal 
number of packets to transmit with $n$ slots remaining is given by:
\begin{eqnarray}\label{Eq:pwl:opt_structure}
z_n^*(x,s) := \left\{
\begin{array}{ll}
   \tilde{z}_{k-1}(s) , & \mbox{if } ~b_{n,k}(s)-\tilde{z}_{k-1}(s) < x \leq b_{n,k-1}(s)-\tilde{z}_{k-1}(s)~, \\
      &~~~~~~~~~~~~~~~~~~~~~~~~~~~~~~~~~~~~~~~~~~~k \in \{0,1,\ldots,K\} \\
   b_{n,k}(s)-x , & \mbox{if } ~b_{n,k}(s)-\tilde{z}_k(s) < x \leq b_{n,k}(s)-\tilde{z}_{k-1}(s)~, \\
      &~~~~~~~~~~~~~~~~~~~~~~~~~~~~~~~~~~~~~~~k \in \{0,1,\ldots,K-1\} \\
   b_{n,K}(s)-x , & \mbox{if } ~b_{n,K}(s)-\tilde{z}_{\max}(s) < x \leq b_{n,K}(s)-\tilde{z}_{K-1}(s) \\
   \tilde{z}_{\max}(s) , & \mbox{if } ~0 \leq x \leq b_{n,K}(s)-\tilde{z}_{\max}(s)
   \end{array} \right. ,
\end{eqnarray}
where $b_{n,-1}(s):= \infty,~\forall s \in {\cal S}$. The optimal receiver buffer level after transmission is given by $y_n^*(x,s)=x+z_n^*(x,s)$.
\end{theorem}
\medskip

The optimal transmission policy in Theorem \ref{Th:pwl:fin} is a finite generalized base-stock policy.
It can be interpreted as follows.
Under each channel condition $s$, there is a target level or critical number associated with each segment of the associated piecewise-linear convex power-rate curve shown in Figure \ref{Fig:power-rate}. If the starting buffer level is below the critical number associated with the first segment, $b_{n,0}(s)$, 
the scheduler should try
to bring the buffer level as close as possible to the target, $b_{n,0}(s)$. 
If the maximum number of packets sent at this per packet power cost, $\tilde{z}_{0}(s)$, does not suffice to reach the critical number $b_{n,0}(s)$, then 
those $\tilde{z}_{0}(s)$ packets are scheduled, and the next segment of the power-rate curve is considered. This second segment has a slope of $\tilde{c}_1(s)$ and an associated critical number $b_{n,1}(s)$, which is no higher than $b_{n,0}(s)$, the first critical number. If the starting buffer level plus the $\tilde{z}_{0}(s)$ already-scheduled packets brings the buffer level above $b_{n,1}(s)$, then no more packets are scheduled for transmission. Otherwise, it is optimal to transmit so as to bring the buffer level as close as possible to $b_{n,1}(s)$, by transmitting up to $\tilde{z}_{1}(s)-\tilde{z}_{0}(s)$ additional packets at a cost of $\tilde{c}_1(s)$ power units per packet. This process continues with the sequential consideration of each segment of the power-rate curve. At each successive iteration, the target level is lower and the starting buffer level, updated to include already-scheduled packets, is higher. The process continues until the buffer level reaches or exceeds a critical number, or the full power $P$ is consumed. Note that this sequential consideration is not actually done online, but only meant to provide an intuitive explanation of the optimal policy.
See Figure \ref{Fig:fgbs} for diagrams of the structure of the optimal finite generalized base-stock policy.
\begin{figure}[htbp]
\centerline {
\includegraphics[width=6in]{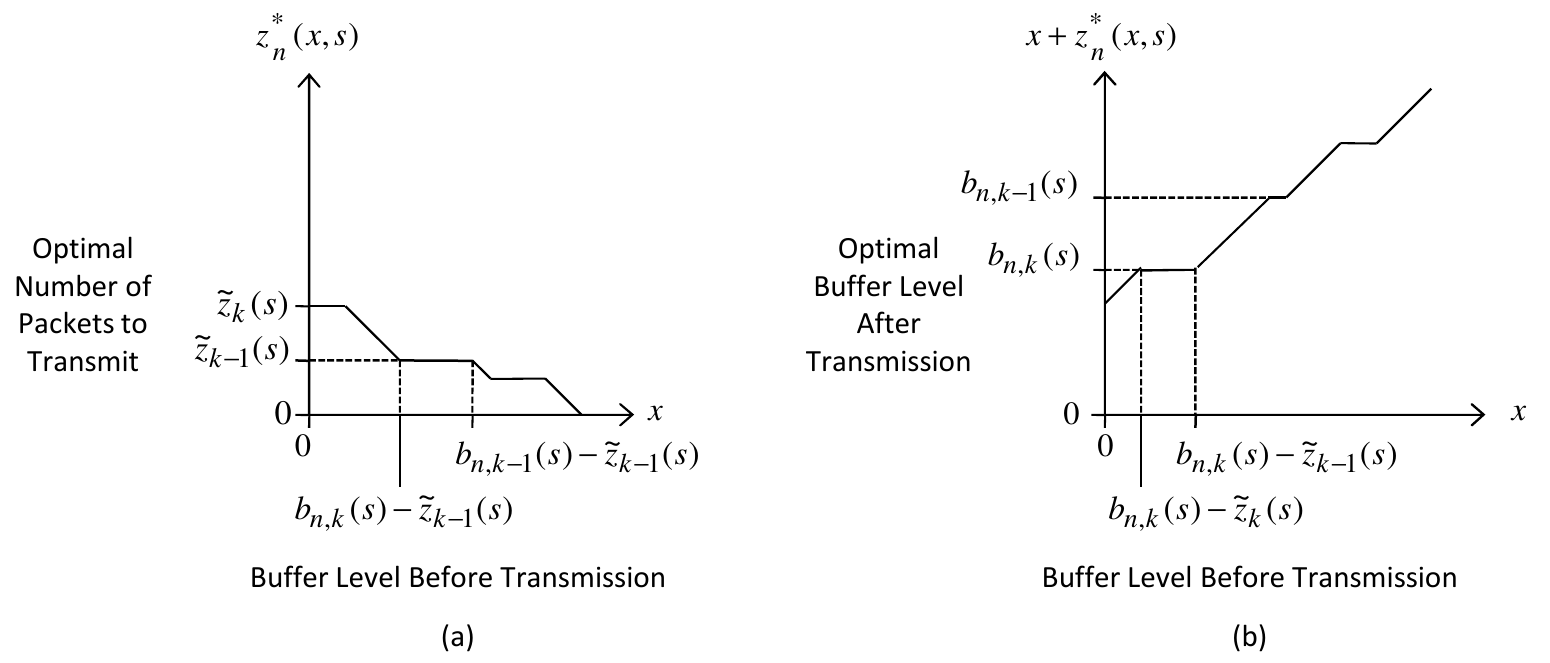}
} \caption{Optimal transmission policy in slot $n$ when the state is $(x,s)$. (a) depicts the optimal transmission quantity, and (b) depicts the resulting number of packets available for playout in slot $n$.}  \label{Fig:fgbs}
\end{figure}

\subsection{Computation of Critical Numbers} \label{Se:highSNR:computation}
\medskip

While finite generalized base-stock policies have been considered in the inventory literature for almost three decades, 
we are not aware of any previous studies that explicitly compute the critical numbers for any model where such a policy is optimal. In this section, we compute the critical numbers under each channel condition when technical conditions similar to those of Section \ref{Se:calculation_low} are satisfied.
We consider the special case when the channel condition is independent and identically distributed from slot to slot;
the holding cost function is linear (i.e., $h(x)=h \cdot x$);
and the following technical condition on the power-rate functions is satisfied for each possible channel condition $s \in {\cal S}$: $\tilde{z}_{\max}(s) = \tilde{l}_{\max} \cdot d$ for some  $\tilde{l}_{\max} \in \Nat$, and for every $k \in \{0,1,\ldots,K-1\}$, $\tilde{z}_k(s) = \tilde{l}_k \cdot d$ for some $\tilde{l}_k \in \Nat$; i.e., the slopes of the effective power-rate functions only change at integer multiples of the 
drainage rate $d$. Under these conditions, we can completely characterize the optimal transmission policy.

As in Theorem \ref{Th:one:finite}, we recursively define a set of thresholds, and use them to determine
the
critical numbers, $\left\{b_{n,k}(s)\right\}_{k \in \{-1,0,\ldots,K\}}$, for each channel condition, at each time.
\begin{theorem} \label{Th:pwl_calc}
Define the thresholds $\tilde{\gamma}_{n,j}$ for $n\in\{1,2,\ldots,N\}$
and $j\in\Nat$ recursively, as follows:
\begin{itemize}
\item[(i)] If $j=1$, $\tilde{\gamma}_{n,j}=\infty$; \item[(ii)] If $j>n$,
$\tilde{\gamma}_{n,j}=0$; \item[(iii)] If $2\leq j \leq n$,
\end{itemize}
\begin{eqnarray}\label{Eq:high:iii}
 \tilde{\gamma}_{n,j} = -h + \alpha\cdot\left(
 \begin{array}{l}
    \sum\limits_{s:~\tilde{c}_0(s)\geq \tilde{\gamma}_{n-1,j-1}} p(s)\cdot\tilde{\gamma}_{n-1,j-1} \\
    ~~~~~+
    \sum\limits_{k=0}^{K-1}
    \left\{
    \begin{array}{l}
    \sum\limits_{s:~\tilde{\gamma}_{n-1,j-1+\tilde{L}_k(s)}\leq \tilde{c}_k(s)< \tilde{\gamma}_{n-1,j-1+\tilde{L}_{k-1}(s)}} p(s)\cdot \tilde{c}_k(s) \\
~~~~~+ \sum\limits_{s:~\tilde{c}_k(s) < \tilde{\gamma}_{n-1,j-1+\tilde{L}_k(s)} \leq \tilde{c}_{k+1}(s)}
p(s) \cdot \tilde{\gamma}_{n-1,j-1+\tilde{L}_k(s)}
\end{array}
\right\}
\\
~~~~~+ \sum\limits_{s:~\tilde{\gamma}_{n-1,j-1+\tilde{L}_{\max}(s)} \leq \tilde{c}_K(s) < \tilde{\gamma}_{n-1,j-1+\tilde{L}_{K-1}(s)}}
p(s) \cdot \tilde{c}_K(s) \\
~~~~~+ \sum\limits_{s:~\tilde{c}_K(s) < \tilde{\gamma}_{n-1,j-1+\tilde{L}_{\max}(s)}}
p(s) \cdot \tilde{\gamma}_{n-1,j-1+\tilde{L}_{\max}(s)} \\
\end{array}
\right),
\end{eqnarray}
\medskip

\noindent where $p(s)$ is the probability of the channel being in
state $s$ in a time slot, $\tilde{L}_k(s):=\frac{\tilde{z}_k(s)}{d}$ for all $s \in {\cal S}$ and $k \in \{0,1,\ldots,K-1\}$, and $\tilde{L}_{\max}(s):=\frac{\tilde{z}_{\max}(s)}{d}$ for all $s \in {\cal S}$.
%
%
%
%
%
%
%
%
%
%
%
%
%
%
%
For each $n\in\{1,2,\ldots,N\}$ and $s\in{\cal{S}}$, define $b_{n,-1}(s):=\infty$ and for all $k \in \{0,1,\ldots,K\}$, if
$\tilde{\gamma}_{n,j+1}\leq \tilde{c}_k(s)< \tilde{\gamma}_{n,j}$, define
$b_{n,k}(s) := j \cdot d$.
%
%
%
%
%
%
%
%
The optimal control strategy for Problem (\textbf{P1}) is then given by
$\boldsymbol{\pi}^*=\bigl\{z_N^*,z_{N-1}^*,\ldots,z_1^*\bigr\}$, where for all $n \in \{N,N-1,\ldots,1\}$,
$z_n^*(x,s)$ is given by \eqref{Eq:pwl:opt_structure}.
\end{theorem}
\medskip
\medskip

It is straightforward to check that Theorem \ref{Th:pwl_calc} is in fact a generalization of Theorem \ref{Th:one:finite}. To see this, set $K=0$ so that the summation from $k=0$ to $k=K-1$ on the right-hand side of \eqref{Eq:high:iii} drops out. Then $\tilde{\gamma}_{n,j}$ in \eqref{Eq:high:iii} is the same as ${\gamma}_{n,j}$ in \eqref{Eq:one:iii}, $\tilde{c}_0(s)$ corresponds to $c_s$ in \eqref{Eq:one:iii}, $b_{n,0}(s)$ corresponds to $b_n(s)$, $\tilde{z}_{\max}(s)$ corresponds to $\frac{P}{c_s}$, $\tilde{L}_{\max}(s)$ corresponds to $L(s)$, and $\tilde{L}_{K-1}(s)=0$. The resulting optimal transmission policies are also the same.

In Theorem \ref{Th:pwl_calc}, the threshold
$\tilde{\gamma}_{n,j}$ may again be interpreted as the per packet power cost at
which, with $n$ slots remaining in the horizon, the expected
cost-to-go of transmitting packets to cover the user's playout
requirements for the next $j-1$ slots is the same as the expected
cost-to-go of transmitting packets to cover the user's
requirements for the next $j$ slots. The intuition behind the recursion \eqref{Eq:high:iii} is similar to the detailed explanation given in Section
\ref{Se:calculation_low}. Namely, we can start with equation \eqref{Eq:thresh_intuition0} and expand out the right-hand side based on the known structure of the optimal policy, until, after a fair bit of algebra, the result is \eqref{Eq:high:iii}. A detailed proof of Theorem \ref{Th:pwl_calc} is included in Appendix A.



\subsection{Structure of the Optimal Policy for the Infinite Horizon Discounted Expected Cost Problems} \label{Se:highSNR:infinite}
\medskip

In this section, we show that the optimal policy for the infinite horizon discounted expected cost problem 
is the natural extension of the optimal policy for the finite horizon discounted expected cost problem;
namely, it is a finite generalized base-stock policy 
characterized by \emph{time-invariant} sequences of critical numbers for each channel condition. These time-invariant sequences of critical numbers for the infinite horizon discounted expected cost problem are equal to the limit of the finite horizon sequences of critical numbers as the time horizon $N$ goes to infinity.
\medskip

\begin{theorem} \label{Th:one:infinite}
~~~~~~~~~~~~~~~~~~~~
\begin{itemize}
\item[(a)] For a fixed $x \in\Real_+$ and $s\in{\cal{S}}$, $V_n(x,s)$ is nondecreasing in $n$. Moreoever, $\lim\limits_{n\rightarrow\infty} V_n(x,s)$ exists and
is finite, $\forall x \in\Real_+, \forall s\in{\cal{S}}$.
\medskip

\item[(b)] Define $V_{\infty}(x,s):= \lim\limits_{n\rightarrow\infty}
V_n(x,s).$ Then 
$V_{\infty}(x,s)$ is convex in $x$ for any
fixed $s\in{\cal{S}}$.

\item[(c)] Define $\tilde{g}_{\infty}(y,s):=
h(y-d) + \alpha \cdot \Expectation \left[V_{\infty}\left(y-d,S^{\prime}\right) \mid S=s \right]
$, where $S^{\prime}$ is the channel condition in the subsequent slot.
Then $\tilde{g}_{n}(y,s)$ converges monotonically to $\tilde{g}_{\infty}(y,s),\forall y\in [d,\infty),\forall s\in{\cal{S}}$;  
$\tilde{g}_{\infty}(y,s)$ is convex in $y$ for any
fixed $s\in{\cal{S}}$; and $\lim\limits_{y\rightarrow\infty} \tilde{g}_{\infty}(y,s) = \infty,
\forall s\in{\cal{S}}$.

\item[(d)] Define $b_{\infty,-1}(s) := \infty$ and
\begin{eqnarray*}
b_{\infty,k}(s) :=
\max\Bigl\{d, \inf\bigl\{b \bigm| \tilde{g}_{\infty}^{\prime+}(b,s) \geq -\tilde{c}_k(s) \bigr\} \Bigr\}~,~\forall k \in \{0,1,\ldots,K\}~,
\end{eqnarray*}
where $\tilde{g}_{\infty}^{\prime+}(b,s)$ represents the right derivative: 
\begin{eqnarray*}
\tilde{g}_{\infty}^{\prime+}(b,s) := \lim\limits_{y \downarrow b} \frac{\tilde{g}_{\infty}(y,s)-\tilde{g}_{\infty}(b,s)}{y-b}~.
\end{eqnarray*}
Then $b_{\infty,k}(s)=\lim\limits_{n\rightarrow\infty}
b_{n,k}(s)$ for all $k \in \{-1,0,1,\ldots,K\}$.

\item[(e)] $V_{\infty}(x,s)$ satisfies the
$\alpha$-discounted cost optimality equation ($\alpha$-DCOE):
\begin{eqnarray}\label{Eq:pwl:inf_functional}
V_{\infty}(x,s)
&=&
\min\limits_{\bigl\{\max(0,d-x) \leq z \leq \tilde{z}_{\max}(s)\bigr\}}
\left\{
\begin{array}{l}
c(z,s) + h(x+z-d) \\
+ \alpha \cdot \Expectation \bigl[V_{\infty}(x+z-d,S^{\prime}) \bigm| S =s\bigr]
\end{array}
\right\} \nonumber \\
&=& \min\limits_{\bigl\{\max(0,d-x) \leq z \leq \tilde{z}_{\max}(s)\bigr\}}
\Bigl\{
c(z,s) + \tilde{g}_{\infty}(x+z,s)
\Bigr\},~\forall x \in \Real_+,\forall s \in {\cal S},
\end{eqnarray}
and the minimum on the right hand side of (\ref{Eq:pwl:inf_functional}) is achieved by:
\begin{eqnarray*}
z_{\infty}^*(x,s) := \left\{
\begin{array}{ll}
   \tilde{z}_{k-1}(s) , & \mbox{if } ~b_{\infty,k}(s)-\tilde{z}_{k-1}(s) < x \leq b_{\infty,k-1}(s)-\tilde{z}_{k-1}(s)~, \\
   &~~~~~~~~~~~~~~~~~~~~~~~~~~~~~~~~~~~~~~~~~~~k \in \{0,1,\ldots,K\} \\
   b_{\infty,k}(s)-x , & \mbox{if } ~b_{\infty,k}(s)-\tilde{z}_k(s) < x \leq b_{\infty,k}(s)-\tilde{z}_{k-1}(s)~,\\
   &~~~~~~~~~~~~~~~~~~~~~~~~~~~~~~~~~~~~~~k \in \{0,1,\ldots,K-1\} \\
   b_{\infty,K}(s)-x , & \mbox{if } ~b_{\infty,K}(s)-\tilde{z}_{\max}(s) < x \leq b_{\infty,K}(s)-\tilde{z}_{K-1}(s) \\
   \tilde{z}_{\max}(s) , & \mbox{if } ~0 \leq x \leq b_{\infty,K}(s)-\tilde{z}_{\max}(s)
   \end{array} \right.
\end{eqnarray*}

%
\item[(f)] The optimal stationary policy for Problem (\textbf{P2}) in the case of a single receiver with piecewise-linear convex power-rate curves is given by
$\boldsymbol{{\pi_{\infty}^*}}=(z_{\infty}^*,z_{\infty}^*,\ldots)$.
\end{itemize}
\end{theorem}
\medskip
\noindent A detailed proof, 
 which follows the
logic conveyed in the statement of the theorem, is included in Appendix B. As a special case of Theorem \ref{Th:one:infinite}, the optimal policy in Problem (\textbf{P2}) for the case discussed in Section \ref{Se:lowSNR} of a single receiver with linear power-rate curves is given by $\boldsymbol{{\pi_{\infty}^*}}=(z_{\infty}^*,z_{\infty}^*,\ldots)$, where:
\begin{eqnarray*}
z_{\infty}^*(x,s):=\left\{
\begin{array}{ll}
   0 , & \mbox{if } ~x \geq b_{\infty}(s) \\
   b_{\infty}(s)-x , & \mbox{if } ~b_{\infty}(s)-\frac{P}{c_s}\leq x < b_{\infty}(s) \\
   \frac{P}{c_s} , & \mbox{if } ~x < b_{\infty}(s) - \frac{P}{c_s}  \\
\end{array} \right.~,
\end{eqnarray*}
and $b_{\infty}(s):=\lim\limits_{n\rightarrow\infty} b_n(s)$.

\subsection{Structure of the Optimal Policy for the Infinite Horizon Average Expected Cost Problems} \label{Se:pwl:average}
\medskip
In this section we use the \emph{vanishing discount approach} to show that the finite generalized base-stock structure is also optimal for the infinite horizon average expected cost problem, \textbf{(P3)}. We show that an optimal policy for the infinite horizon average expected cost problem exists and can be represented as the limit as the discount factor increases to one of optimal policies identified in Section \ref{Se:highSNR:infinite} for the infinite horizon discounted expected cost problem. 

In Section \ref{Se:highSNR:infinite}, we suppressed the dependence of the value functions and optimal policies on the discount factor, $\alpha$. Here, we make this dependence explicit by including the discount factor in the subscript labeling of the value functions and optimal policies for the infinite horizon discounted expected cost problem. For example, the value function defined in (b) of Theorem \ref{Th:one:infinite} is now denoted by $V_{\infty,\alpha}(x,s)$.

\begin{theorem} \label{Th:pwl:infa}
For all $\alpha \in [0,1)$, define:
\begin{eqnarray*}
m_{\infty,\alpha}&:=&\inf\limits_{\substack{x \in \Real_+ \\ s \in {\cal S}}} V_{\infty,\alpha}(x,s), \\ 
\rho^*&:=&\lim\limits_{\alpha \nearrow 1} (1-\alpha) \cdot m_{\infty,\alpha},~\hbox{and} \\ 
w_{\infty,\alpha}(x,s)&:=& V_{\infty,\alpha}(x,s) - m_{\infty,\alpha},~\forall x \in \Real_+,~\forall s \in {\cal S}.
\end{eqnarray*}
Then:
\begin{itemize}
\item[(a)] There exists a continuous function $w_{\infty,1}(\cdot,\cdot)$ and a selector $z_{\infty,1}^*(\cdot,\cdot)$ that satisfy the ACOE:
\begin{align*} 
\rho^* + w_{\infty,1}(x,s) 
&=
\min\limits_{\bigl\{\max(0,d-x) \leq z \leq \tilde{z}_{\max}(s)\bigr\}}
\left\{
\begin{array}{l}
c(z,s) + h(x+z-d) \\
+\Expectation \bigl[w_{\infty,1}(x+z-d,S^{\prime}) \bigm| S =s\bigr]
\end{array}
\right\} \\
&=c\Bigl(z_{\infty,1}^*(x,s),s\Bigr) + h\Bigl(x+z_{\infty,1}^*(x,s)-d\Bigr) \\
&~~+\Expectation \left[w_{\infty,1}\Bigl(x+z_{\infty,1}^*(x,s)-d,S^{\prime}\Bigr) \middle| S =s\right],~\forall x \in \Real_+,~\forall s \in {\cal S}.
\end{align*}
\item[(b)] The stationary policy $\boldsymbol{{\pi_{\infty,1}^*}}=(z_{\infty,1}^*,z_{\infty,1}^*,\ldots)$ is optimal for Problem (\textbf{P3}) in the case of a single receiver with piecewise-linear convex power-rate curves.
\item[(c)] The resulting optimal average cost beginning from any initial state $({x},{s}) \in \Real_+ \times {\cal S}$ is $\rho^*$.
\item[(d)] For every 
increasing sequence of discount factors $\{\alpha(l)\}_{l=1,2,\ldots}$ approaching 1, there exists a subsequence $\{\alpha(l_i)\}_{i=1,2,\ldots}$ approaching 1
such that:
\begin{eqnarray*}
w_{\infty,1}({x},{s})&=\lim\limits_{i \rightarrow \infty} w_{\infty,\alpha(l_i)}({x},{s}),~\forall {x} \in \Real_+,~\forall {s} \in {\cal S}.
\end{eqnarray*}
Therefore, for every ${s} \in {\cal S}$, $w_{\infty,1}({x},{s})$ is convex in ${x}$.
\item[(e)] For every $({x},{s}) \in \Real_+ \times {\cal S}$ and increasing sequence of discount factors $\{\alpha(l)\}_{l=1,2,\ldots}$ approaching 1, there exists a subsequence $\{\alpha(l_i)\}_{i=1,2,\ldots}$ approaching 1 and a sequence $\{{x}(i)\}_{i=1,2,\ldots}$ approaching ${x}$ such that:
\begin{eqnarray*}
    {z}_{\infty,1}^*({x},{s})&=\lim\limits_{i \rightarrow \infty} {z}_{\infty,\alpha(l_i)}^*({x}(i),{s})~.
\end{eqnarray*}
\item[(f)] A stationary finite generalized base-stock policy is average cost optimal in the case of piecewise-linear convex power-rate curves, and a stationary modified base-stock policy is average cost optimal in the case of linear power-rate curves. 
\end{itemize}

\end{theorem}
Thus, the structure of the optimal policy is the same for all three problems, (\textbf{P1}), (\textbf{P2}), and (\textbf{P3}).
The proof of Theorem \ref{Th:pwl:infa} is discussed in Appendix C.


\subsection{General Convex Power-Rate Curves}
As mentioned in Section \ref{Se:Sys_Model}, in general, the power-rate curve under each possible channel condition is convex. It can be shown that under convex power-rate curves at each time, the optimal number of packets to send is a nonincreasing function of the starting buffer level. However, without any further structure on the power-rate curves, it is not computationally tractable to compute such optimal policies, known as generalized base-stock policies (a superclass of the \emph{finite} generalized base-stock policies discussed above). This is why we have chosen to analyze piecewise-linear convex power-rate curves, which can be used to approximate general convex power-rate curves. More specifically, our analysis suggests approximating the general convex power-rate curves by piecewise-linear convex power-rate curves where the slopes change at integer multiples of the demand $d$, in order to be able to apply Theorem \ref{Th:pwl_calc} to compute the critical numbers in an extremely efficient manner. Doing so represents an approximation at the modeling stage followed by an exact solution, as compared to modeling the power-rate curves as more general convex functions and having to approximate the solution. Finally, we note that increasing the number of segments used to model the piecewise-linear convex functions leads to a better approximation, but comes at the cost of some extra complexity in implementing the optimal policy, as the scheduler needs to store at least one critical number for each segment of each power-rate curve.

%% file: two_users.tex
\section{Two Receivers with Linear Power-Rate Curves}\label{Se:two_users}
In this section, we analyze the finite and infinite horizon discounted expected cost problems when there are two receivers ($M=2$), and the power-rate functions under different channel conditions are linear for each user. Each user $m$'s channel condition evolves as a homogeneous 
Markov process, $\left\{S_n^m\right\}_{n=N,N-1,\ldots,1}$. As discussed in Sections \ref{Se:introduction} and \ref{Se:problem}, the time-varying channel conditions of the two users are independent of each other, and the transmission scheduler can exploit this spatial diversity. Like Section \ref{Se:lowSNR}, we denote the power consumption per unit of data transmitted to receiver $m$ under channel condition $s^m$ by $c_s^m$. The row vector of these per unit power consumptions is given by $\textbf{c}_{\textbf{s}}^{\transpose}$, so that the total power consumption in slot $n$ is given by $\sum_{m=1}^2 c^m(Z_n^m,S_n^m) = \textbf{c}_{\textbf{s}}^{\transpose} \textbf{Z}_n$. We denote the total holding costs $\sum_{m=1}^2 h^m(X_n^m+Z_n^m-d^m)$ by ${h}(\textbf{X}_n+\textbf{Z}_n-\textbf{d})$.


With these notations, the dynamic program \eqref{Eq:gen:finDP} for Problem (\textbf{P1}) becomes:
\begin{eqnarray}
V_n(\textbf{x},\textbf{s}) &=&
\min_{\textbf{z} \in {\cal A}^{\textbf{d}}(\textbf{x},\textbf{s})}
\left\{
\begin{array}{l}
\textbf{c}_{\textbf{s}}^{\transpose} \textbf{z} +{h}(\textbf{x}+\textbf{z}-\textbf{d}) \\
 +\alpha \cdot \Expectation \bigl[V_{n-1}(\textbf{x}+\textbf{z}-\textbf{d},\textbf{S}_{n-1})\bigm| \textbf{S}_n=\textbf{s}\bigr]
\end{array}
\right\} \label{Eq:two:finDP1} \\
&=&
\min_{\textbf{y} \in \tilde{\cal A}^{\textbf{d}}(\textbf{x},\textbf{s})}
\left\{
\begin{array}{l}
\textbf{c}_{\textbf{s}}^{\transpose} [\textbf{y}-\textbf{x}] +{h}(\textbf{y}-\textbf{d}) \\
 +\alpha \cdot \Expectation \bigl[V_{n-1}(\textbf{y}-\textbf{d},\textbf{S}_{n-1})\bigm| \textbf{S}_n=\textbf{s}\bigr]
\end{array}
\right\} \label{Eq:two:finDP2} \\
&=& -\textbf{c}_{\textbf{s}}^{\transpose} \textbf{x} + \min_{\textbf{y} \in \tilde{\cal A}^{\textbf{d}}(\textbf{x},\textbf{s})}
\Bigl\{
G_n(\textbf{y},\textbf{s}) \Bigr\}
~~~~~~n=N,N-1,\ldots,1~, \nonumber \\
%
V_0(\textbf{x},\textbf{s}) &=& 0, ~~\forall \textbf{x} \in \Real_+^2,\forall
\textbf{s} \in {\cal{S}}:={\cal{S}}^1 \times {\cal{S}}^2, \nonumber
\end{eqnarray}
where
\begin{align}
&G_n(\textbf{y},\textbf{s}):=\textbf{c}_{\textbf{s}}^{\transpose} \textbf{y} +{h}(\textbf{y}-\textbf{d})+\alpha \cdot \Expectation \bigl[V_{n-1}(\textbf{y}-\textbf{d},\textbf{S}_{n-1})\bigm| \textbf{S}_n=\textbf{s}\bigr], \nonumber \\
&~~~~~~~~~~~~~~~~~~~~~~~~~~~~~~~~~~~~~~~~~~~~~~~~~~~~~~~\forall \textbf{y} \in [d^1,\infty) \times [d^2,\infty),\forall \textbf{s} \in {\cal{S}},\hbox{ and } \nonumber \\
&\tilde{\cal A}^{\textbf{d}}(\textbf{x},\textbf{s}):=\biggl\{\textbf{y} \in \Real_{+}^2 :
\textbf{y} \succeq \textbf{d} \vee \textbf{x} 
\hbox{ and } 
\textbf{c}_{\textbf{s}}^{\transpose} [\textbf{y}-\textbf{x}] \leq P
\biggr\},~~\forall \textbf{x} \in \Real_+^2,\forall
\textbf{s} \in {\cal{S}}. \label{Eq:two:actionspace}
\end{align}
The transition from \eqref{Eq:two:finDP1} to \eqref{Eq:two:finDP2} follows again from a change of variable in the action space from $\textbf{Z}_n$ to $\textbf{Y}_n$, where $\textbf{Y}_n = \textbf{X}_n + \textbf{Z}_n$. The controlled random vector $\textbf{Y}_n$ represents the queue lengths of the receiver buffers \emph{after} transmission takes place in the $n^{th}$ slot, but \emph{before} playout takes place (i.e., before $d^m$ packets are removed from user $m$'s buffer). The restrictions on the action space, $\textbf{y} \succeq \textbf{d} \vee \textbf{x}
\hbox{ and } 
\textbf{c}_{\textbf{s}}^{\transpose} [\textbf{y}-\textbf{x}] \leq P$, ensure: (i) a nonnegative number of packets is transmitted to each user; (ii) there are at least $d^m$ packets in user $m$'s receiver buffer following transmission, in order to satisfy the underflow constraint; and (iii) the power constraint is satisfied.

Without the per slot peak power constraint, this $M$-dimensional problem would be separable, and could be solved by solving $M$ instances of the one-dimensional problem of Section \ref{Se:lowSNR}; however, the joint power constraint couples the queues.\footnote{This problem therefore falls into the class of \emph{weakly coupled stochastic dynamic programs} \cite{hawkins, adelman_weak}.} As a result, the optimal transmission quantity to one receiver depends on the other receivers' queue length, as the following example shows.

\begin{example} \label{Ex:two_coupling}
Assume receiver 1's channel is currently in a ``poor'' condition, receiver 2's channel is currently in a ``medium'' condition, and receiver 2's buffer contains enough packets to satisfy the demand for the next few slots. We consider two different scenarios for receiver 1's buffer level to show how the optimal transmission quantity to receiver 2 depends on receiver 1's buffer level. 
In Scenario 1, receiver 1's buffer already contains many packets. In this scenario, 
%
%
it may be beneficial for the scheduler to wait for receiver 2 to have a better channel condition, because it will be able to take full advantage of an ``excellent'' condition when it comes. In Scenario 2, receiver 1's queue only contains enough packets for playout in the current slot.
It may be optimal to transmit some packets to receiver 2 in the current slot in this scenario. To see this, note that
 even if receiver 2 experiences the best possible channel condition in the next slot, 
 the scheduler will need to allocate some power to receiver 1 in order to prevent receiver 1's buffer from emptying. 
 Therefore, the scheduler anticipates not being able to take full advantage of receiver 2's ``excellent'' condition in the next slot, and 
 may compensate by sending some 
 packets in the current slot under the ``medium'' condition.
 \end{example}

\subsection{Structure of Optimal Policy for the Finite Horizon Discounted Expected Cost Problem}

Before proceeding to the structure of the optimal transmission policy, we state some key properties of the value functions in the following theorem.
\begin{theorem} \label{Th:two_users:properties}
With two receivers and linear power-rate curves, the following statements are true for $n=1,2,\ldots, N$, and for all $\textbf{s} \in {\cal{S}}$:
\begin{itemize}
\item[(i)] 
$V_{n-1}(\textbf{x},\textbf{s})$ is convex in $\textbf{x}$.

\item[(ii)] 
$V_{n-1}(\textbf{x},\textbf{s})$ is supermodular in $\textbf{x}$; i.e., for all $\bar{\textbf{x}},\tilde{\textbf{x}} \in \Real_+^2$,
\begin{eqnarray*}
V_{n-1}(\bar{\textbf{x}},\textbf{s})+V_{n-1}(\tilde{\textbf{x}},\textbf{s}) \leq
V_{n-1}(\bar{\textbf{x}} \wedge
\tilde{\textbf{x}},\textbf{s})+V_{n-1}(\bar{\textbf{x}} \vee
\tilde{\textbf{x}},\textbf{s})~.
\end{eqnarray*}

\item[(iii)] 
$G_n(\textbf{y},\textbf{s})$ is convex in $\textbf{y}$.

\item[(iv)] 
$G_n(\textbf{y},\textbf{s})$ is supermodular in $\textbf{y}$;
i.e., for all $\bar{\textbf{y}},\tilde{\textbf{y}} \in 
\left[d^1,\infty\right) \times \left[d^2,\infty\right)$,
\begin{eqnarray*}
G_n(\bar{\textbf{y}},\textbf{s})+G_n(\tilde{\textbf{y}},\textbf{s}) \leq
G_n({\bar{\textbf{y}}} \wedge
\tilde{\textbf{y}},\textbf{s})+G_n({\bar{\textbf{y}}} \vee
\tilde{\textbf{y}},\textbf{s})~.
\end{eqnarray*}
\item[(v)] 
$y_n^1 <
{\hat{y}_n}^1$ implies:
\begin{eqnarray*}
\inf \left\{\argmin_{y_n^2 \in 
[d^2,\infty)}
\biggl\{G_n\left(y_n^1,y_n^2,s^1,s^2\right)\biggr\}
\right\} \geq \inf \left\{\argmin_{y_n^2 \in
[d^2,\infty)}
\biggl\{G_n\left({\hat{y}_n}^1,y_n^2,s^1,s^2\right)\biggr\}
\right\}
\end{eqnarray*}
and $y_n^2 < {\hat{y}_n}^2$ implies:
\begin{eqnarray*}
\inf \left\{\argmin_{y_n^1 \in
[d^1,\infty)}
\biggl\{G_n\left(y_n^1,y_n^2,s^1,s^2\right)\biggr\}
\right\} \geq \inf \left\{\argmin_{y_n^1 \in
[d^1,\infty)}
\biggl\{G_n\left(y_n^1,{\hat{y}_n}^2,s^1,s^2\right)\biggr\}
\right\}.
\end{eqnarray*}
\end{itemize}
\end{theorem}
\medskip

A detailed proof is included in Appendix A. Because
$-\textbf{c}_{\textbf{s}}^{\transpose} \textbf{x}$ is supermodular in $\textbf{x}$, the key part of the induction step in the proof of (ii) is to show that $\min_{\textbf{y} \in
\tilde{\cal{A}}^{\textbf{d}}(\textbf{x},\textbf{s})} \left\{G_{n-1}(\textbf{y},\textbf{s})\right\}$ is also supermodular in $\textbf{x}$. Denoting $\argmin_{\textbf{y} \in
\tilde{\cal{A}}^{\textbf{d}}(\textbf{x},\textbf{s})} \left\{G_{n-1}(\textbf{y},\textbf{s})\right\}$ by $\textbf{y}^*(\textbf{x},\textbf{s})$, we do this constructively by showing that for all $\bar{\textbf{x}},\tilde{\textbf{x}} \in \Real_+^2$:
\begin{align}
&\min_{\textbf{y} \in
\tilde{\cal{A}}^{\textbf{d}}(\bar{\textbf{x}},\textbf{s})} \left\{G_{n-1}(\textbf{y},\textbf{s})\right\}
+\min_{\textbf{y} \in
\tilde{\cal{A}}^{\textbf{d}}(\tilde{\textbf{x}},\textbf{s})} \left\{G_{n-1}(\textbf{y},\textbf{s})\right\} \nonumber \\
&\leq~
G_{n-1}(\bar{\textbf{y}},\textbf{s})+G_{n-1}(\tilde{\textbf{y}},\textbf{s}) \nonumber \\
&\leq~
G_{n-1}\Bigl(\textbf{y}^*(\bar{\textbf{x}} \wedge \tilde{\textbf{x}},\textbf{s}),\textbf{s}\Bigr)+G_{n-1}\Bigl(\textbf{y}^*(\bar{\textbf{x}} \vee \tilde{\textbf{x}},\textbf{s}),\textbf{s}\Bigr) \label{Eq:two:keyStep} \\
&=~
\min_{\textbf{y} \in
\tilde{\cal{A}}^{\textbf{d}}(\bar{\textbf{x}} \wedge \tilde{\textbf{x}},\textbf{s})} \left\{G_{l-1}(\textbf{y},\textbf{s})\right\}
+\min_{\textbf{y} \in
\tilde{\cal{A}}^{\textbf{d}}(\bar{\textbf{x}} \vee \tilde{\textbf{x}},\textbf{s})} \left\{G_{l-1}(\textbf{y},\textbf{s})\right\} \nonumber ,
\end{align}
for a specific choice of $\bar{\textbf{y}} \in \tilde{\cal{A}}^{\textbf{d}}(\bar{\textbf{x}},\textbf{s})$ and $\tilde{\textbf{y}} \in \tilde{\cal{A}}^{\textbf{d}}(\tilde{\textbf{x}},\textbf{s})$. The difficulty is cleverly constructing $\bar{\textbf{y}}$ and $\tilde{\textbf{y}}$, depending on the relative locations of $\bar{\textbf{x}}$, $\tilde{\textbf{x}}$, $\textbf{y}^*(\bar{\textbf{x}} \wedge \tilde{\textbf{x}})$, and $\textbf{y}^*(\bar{\textbf{x}} \vee \tilde{\textbf{x}})$, so as to ensure \eqref{Eq:two:keyStep} is true.

%

It follows from Theorem \ref{Th:two_users:properties} that the structure of the optimal transmission policy for the finite horizon discounted expected cost problem is given by the following theorem.
\begin{theorem}\label{Th:two_users:structure}
For every $n \in
\{1,2,\ldots,N\}$ and $\textbf{s} \in {\cal{S}}^1 \times {\cal{S}}^2$, define the nonempty set of global minimizers of $G_n(\cdot,\textbf{s})$:
\begin{eqnarray*}
{\mathcal B}_n(\textbf{s}):= \left\{\hat{\textbf{y}} \in
[d^1,\infty) \times [d^2, \infty):~G_n(\hat{\textbf{y}},\textbf{s})=\min\limits_{\textbf{y}\in [d^1,\infty) \times [d^2,\infty)}
G_n(\textbf{y},\textbf{s}) \right\}~.
\end{eqnarray*}
Define also
\begin{eqnarray*}
b_n^1\left(\textbf{s}\right):= \min\Bigl\{y^1 \in [d^1,\infty): (y^1,y^2) \in {\mathcal B}_n(\textbf{s}) \hbox{ for some } y^2 \in [d^2,\infty) \Bigr\}~,
\end{eqnarray*}
and
\begin{eqnarray*}
b_n^2\left(\textbf{s}\right):=\min\Bigl\{y^2 \in [d^2,\infty): \bigl(b_n^1\left(\textbf{s}\right),y^2\bigr) \in {\mathcal B}_n(\textbf{s}) \Bigr\}~.
\end{eqnarray*}
Then the vector $\textbf{b}_n(\textbf{s}) = \Bigl(b_n^1\left(\textbf{s}\right),b_n^2\left(\textbf{s}\right) \Bigr) \in {\mathcal B}_n(\textbf{s})$ is a global minimizer of $G_n(\cdot,\textbf{s})$.
Define also the functions:
\begin{align*}
f_n^1(x^2,\textbf{s})&:= \inf \left\{\argmin_{y^1 \in
[d^1,\infty)}
\biggl\{G_n\left(y^1,x^2,s^1,s^2\right)\biggr\}
\right\},\hbox{ for }x^2 \in [d^2,\infty),
\hbox{ and } \\
f_n^2(x^1,\textbf{s})&:= \inf \left\{\argmin_{y^2 \in
[d^2,\infty)}
\biggl\{G_n\left(x^1,y^2,s^1,s^2\right)\biggr\}
\right\},\hbox{ for }x^1 \in [d^1,\infty).
\end{align*}
Note that by construction, $f_n^1\bigl(b_n^2(\textbf{s}),\textbf{s}\bigr)=b_n^1(\textbf{s})$ and $f_n^2\bigl(b_n^1(\textbf{s}),\textbf{s}\bigr)=b_n^2(\textbf{s})$. Partition 
$\Real_+^2$ into the following seven regions:
\begin{align*}
{\mathcal R}_{I}(n,\textbf{s})&:=  \Bigl\{ \textbf{x} \in \Real_+^2 : \textbf{x} \succeq \bigl(f_n^1 (x^2,\textbf{s}),f_n^2(x^1,\textbf{s}) \bigr) \hbox{ and } \textbf{x} \neq  \textbf{b}_n(\textbf{s}) \Bigr\}\\
{\mathcal R}_{II}(n,\textbf{s})&:= \Bigl\{ \textbf{x} \in \Real_+^2 : \textbf{x} \preceq \textbf{b}_n(\textbf{s}) \hbox{ and } \textbf{c}_{\textbf{s}}^{\transpose}\left[\textbf{b}_n(\textbf{s}) - \textbf{x} \right] \leq P \Bigr\}\\
{\mathcal R}_{III-A}(n,\textbf{s})&:= \Bigl\{ \textbf{x} \in \Real_+^2 : x^2 > b_n^2(\textbf{s}) \hbox{ and } f_n^1(x^2,\textbf{s})-\frac{P}{c_{s^1}}\leq x^1 < f_n^1(x^2,\textbf{s}) \Bigr\}\\
{\mathcal R}_{III-B}(n,\textbf{s})&:= \Bigl\{ \textbf{x} \in \Real_+^2 : x^1 > b_n^1(\textbf{s}) \hbox{ and } f_n^2(x^1,\textbf{s})-\frac{P}{c_{s^2}}\leq x^2 < f_n^2(x^1,\textbf{s}) \Bigr\} \\
{\mathcal R}_{IV-A}(n,\textbf{s})&:= \Bigl\{ \textbf{x} \in \Real_+^2 : x^2 > b_n^2(\textbf{s}) \hbox{ and } x^1 < f_n^1(x^2,\textbf{s})-\frac{P}{c_{s^1}} \Bigr\} \\
{\mathcal R}_{IV-B}(n,\textbf{s})&:=\Bigl\{ \textbf{x} \in \Real_+^2 : \textbf{x} \preceq \textbf{b}_n(\textbf{s}) \hbox{ and } \textbf{c}_{\textbf{s}}^{\transpose}\left[\textbf{b}_n(\textbf{s}) - \textbf{x} \right] > P \Bigr\} \\
{\mathcal R}_{IV-C}(n,\textbf{s})&:= \Bigl\{ \textbf{x} \in \Real_+^2 : x^1 > b_n^1(\textbf{s}) \hbox{ and } x^2 < f_n^2(x^1,\textbf{s})-\frac{P}{c_{s^2}} \Bigr\}~,
\end{align*}
and define ${\mathcal R}_{IV}(n,\textbf{s}):={\mathcal R}_{IV-A}(n,\textbf{s}) \cup {\mathcal R}_{IV-B}(n,\textbf{s}) \cup {\mathcal R}_{IV-C}(n,\textbf{s})$.

Then for Problem (\textbf{P1}) in the case of two receivers with linear power-rate curves, for all $\textbf{x} \notin {\mathcal R}_{IV}(n,\textbf{s})$, an optimal control action with $n$ slots remaining is given by:
\begin{eqnarray} \label{Eq:two:y_star}
\textbf{y}_n^*(\textbf{x},\textbf{s}):=\left\{
\begin{array}{ll}
   \textbf{x} , & \mbox{if } ~ \textbf{x} \in  {\mathcal R}_{I}(n,\textbf{s})\\
   \textbf{b}_n(\textbf{s}) , & \mbox{if } ~ \textbf{x} \in {\mathcal R}_{II}(n,\textbf{s}) \\
   \Bigl(f_n^1(x^2,\textbf{s}) ,x^2 \Bigr) , & \mbox{if } ~ \textbf{x} \in  {\mathcal R}_{III-A}(n,\textbf{s})\\
   \Bigl(x^1 ,f_n^2(x^1,\textbf{s}) \Bigr) , & \mbox{if } ~ \textbf{x} \in  {\mathcal R}_{III-B}(n,\textbf{s})\\
\end{array} \right..
\end{eqnarray}

For all $\textbf{x} \in {\mathcal R}_{IV}(n,\textbf{s})$, there exists an optimal control action with $n$ slots remaining, $\textbf{y}_n^*(\textbf{x},\textbf{s})$, which satisfies:
\begin{eqnarray}\label{Eq:two:full}
\textbf{c}_{\textbf{s}}^{\transpose}\left[\textbf{y}_n^*(\textbf{x},\textbf{s}) - \textbf{x} \right] = P~.
\end{eqnarray}
\end{theorem}

A detailed proof is included in Appendix A. Equation \eqref{Eq:two:full} says that it is optimal for the transmitter to allocate the full power budget for transmission when the vector of receiver buffer levels at the beginning of slot $n$ falls in region ${\mathcal R}_{IV}(n,\textbf{s})$. We cannot say anything in general about the optimal allocation (split) of the full power budget between the two receivers when the starting buffer levels lie in region ${\mathcal R}_{IV}(n,\textbf{s})$. Figure \ref{Fig:two:opt_policy} shows the partition of $\Real_+^2$ into the seven regions, and a
diagram of the structure of the optimal transmission policy. Note that the figure shows the seven regions of the optimal policy for a \emph{fixed realization} of the pair of channel conditions. Under different pairs of channel realizations, the seven regions have the same general form, but the targets $\textbf{b}_n(\textbf{s})$ are shifted and the boundary functions $f_n^1(x^2,\textbf{s})$ and $f_n^2(x^1,\textbf{s})$ are different.

\begin{figure}[htbp]
\centerline {
\includegraphics[width=6in]{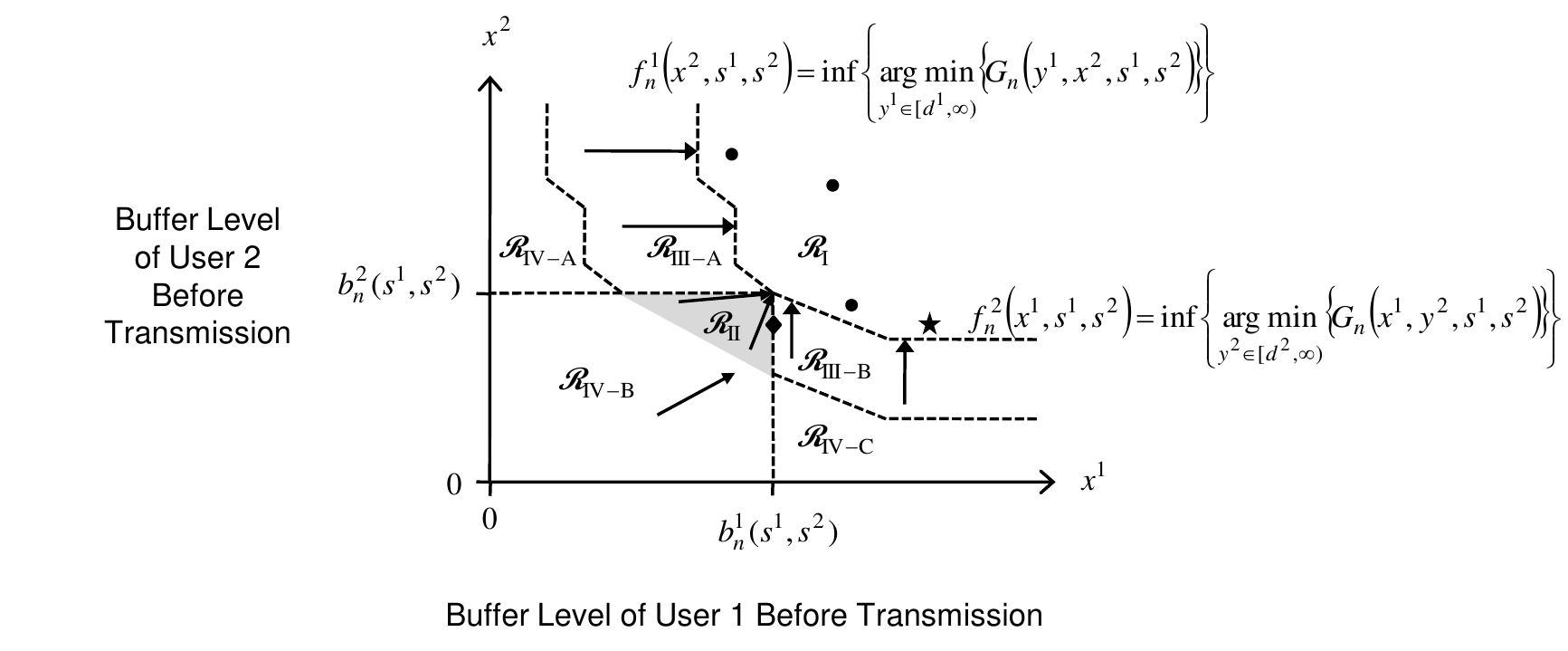}
} \caption{Optimal transmission policy for the two receiver case in slot $n$ when the state is $(\textbf{x},\textbf{s})$. The seven regions described in Theorem \ref{Th:two_users:structure} are labeled. The tails of the arrows represent the vectors of the receiver buffer levels at the beginning of slot $n$, and the heads of the arrows represent the vectors of the receiver buffer levels after transmission but before playout in slot $n$ under the optimal transmission policy. In region ${\mathcal R}_{I}(n,\textbf{s})$, a single dot represents that it is optimal to not transmit any packets to either user. The $\bigstar$ and $\blacklozenge$ represent possible starting buffer levels for Scenarios 1 and 2, respectively, in Example \ref{Ex:two_coupling}.}  \label{Fig:two:opt_policy}
\end{figure}

In some sense, the structure of the optimal policy outlined in Theorem \ref{Th:two_users:structure} can be interpreted as an extension of the modified base-stock policy for the case of a single receiver outlined in Theorem \ref{Th:str:finite}. Namely, under each channel condition at each time, there is a critical number for each receiver $\bigl(b_n^m(\textbf{s})\bigr)$ such that it is optimal to bring both receivers' buffer levels up to those critical numbers if it is possible to do so $\bigl(\hbox{region }{\mathcal R}_{II}(n,\textbf{s})\bigr)$, and it is optimal to not transmit any packets if both receivers' buffer levels start beyond their critical numbers $\bigl(\hbox{region }{\mathcal R}_{I}(n,\textbf{s})\bigr)$. However, this extended notion of the modified base-stock policy only captures the optimal behavior in two of the seven regions, and does not account for the coupling behavior between users that arises through the joint power constraint. For instance, possible starting buffer levels for Scenario 1 and Scenario 2 
in Example \ref{Ex:two_coupling} are illustrated in Figure \ref{Fig:two:opt_policy} by the $\bigstar$ and $\blacklozenge$, respectively. Even though the buffer level of receiver 2 before transmission is the same under both scenarios, the optimal transmission quantity to receiver 2 is different under the two scenarios due to the different starting buffer levels of receiver 1.

\subsection{Structure of the Optimal Policy for the Infinite Horizon Discounted Expected Cost Problems} \label{Se:two:infd}
In this section, we show that the structure of the optimal \emph{stationary (or time-invariant)} policy for the infinite horizon discounted expected cost problem is the same as the structure of the optimal policy for the finite horizon discounted expected cost problem. Moreover, the boundaries of the seven regions of the finite horizon optimal policy shown in Figure \ref{Fig:two:opt_policy} converge to the boundaries of the seven regions of the infinite horizon discounted expected cost optimal policy as the time horizon $N$ goes to infinity.
\begin{theorem} \label{Th:two:infd}
Define:
 \begin{itemize}
 \item[(i)] $V_{\infty}(\textbf{x},\textbf{s}):=\lim\limits_{n \rightarrow \infty}V_n(\textbf{x},\textbf{s})$, for all $\textbf{x} \in \Real_+^2$ and $\textbf{s} \in {\cal S}$ (this limit exists).
 \item[(ii)] $G_{\infty}(\textbf{y},\textbf{s}):=\textbf{c}_{\textbf{s}}^{\transpose} \textbf{y} +{h}(\textbf{y}-\textbf{d})+\alpha \cdot \Expectation \bigl[V_{\infty}(\textbf{y}-\textbf{d},\textbf{S}^{\prime})\bigm| \textbf{S}=\textbf{s}\bigr]$, for all $\textbf{y} \in [d^1,\infty) \times [d^2, \infty)$ and $\textbf{s} \in {\cal S}$.
 \item[(iii)]  ${\mathcal B}_{\infty}(\textbf{s}):= \left\{\hat{\textbf{y}} \in
[d^1,\infty) \times [d^2, \infty):~G_{\infty}(\hat{\textbf{y}},\textbf{s})=\min\limits_{\textbf{y}\in [d^1,\infty) \times [d^2,\infty)}
G_{\infty}(\textbf{y},\textbf{s}) \right\}$.
\item[(iv)] $b_{\infty}^1\left(\textbf{s}\right):= \min\Bigl\{y^1 \in [d^1,\infty): (y^1,y^2) \in {\mathcal B}_{\infty}(\textbf{s}) \hbox{ for some } y^2 \in [d^2,\infty) \Bigr\}$.
\item[(v)] $b_{\infty}^2\left(\textbf{s}\right):=\min\Bigl\{y^2 \in [d^2,\infty): \bigl(b_{\infty}^1\left(\textbf{s}\right),y^2\bigr) \in {\mathcal B}_{\infty}(\textbf{s}) \Bigr\}$.
    \item[(vi)] $\textbf{b}_{\infty}(\textbf{s}) := \Bigl(b_{\infty}^1\left(\textbf{s}\right),b_{\infty}^2\left(\textbf{s}\right) \Bigr)$.
    \item[(vii)] The functions
    \begin{align*}
f_{\infty}^1(x^2,\textbf{s})&:= \inf \left\{\argmin_{y^1 \in
[d^1,\infty)}
\biggl\{G_{\infty}\left(y^1,x^2,s^1,s^2\right)\biggr\}
\right\},\hbox{ for }x^2 \in [d^2,\infty),
\hbox{ and } \\
f_{\infty}^2(x^1,\textbf{s})&:= \inf \left\{\argmin_{y^2 \in
[d^2,\infty)}
\biggl\{G_{\infty}\left(x^1,y^2,s^1,s^2\right)\biggr\}
\right\},\hbox{ for }x^1 \in [d^1,\infty).
\end{align*}
\item[(viii)] The seven regions ${\mathcal R}_{I}(\infty,\textbf{s}) - {\mathcal R}_{IV-C}(\infty,\textbf{s})$, defined in the same way as in Theorem \ref{Th:two_users:structure}, with $n$ replaced by $\infty$.
\end{itemize}
Then
\begin{itemize}
\item[(a)] $V_{\infty}(\textbf{x},\textbf{s})$ satisfies the
$\alpha$-discounted optimality equation ($\alpha$-DCOE):
\begin{eqnarray} \label{Eq:adcoe}
V_{\infty}(\textbf{x},\textbf{s})=
\min_{\textbf{y} \in \tilde{\cal A}^{\textbf{d}}(\textbf{x},\textbf{s})}
\left\{
\begin{array}{l}
\textbf{c}_{\textbf{s}}^{\transpose}[\textbf{y}-\textbf{x}]
+\textbf{h}\left(\textbf{y}-\textbf{d}\right) \\
 +\alpha \cdot \Expectation \bigl[V_{\infty}(\textbf{y}-\textbf{d},\textbf{S}^{\prime})\bigm| \textbf{S}=\textbf{s}\bigr]
 \end{array}
\right\}~,~~\forall \textbf{x} \in \Real_+^2,\forall
\textbf{s} \in {\cal{S}}~.
\end{eqnarray}
\item[(b)] An optimal stationary policy for Problem (\textbf{P2}) in the case of two receivers with linear power-rate curves is given by
$\boldsymbol{{\pi_{\infty}^*}}=(\textbf{y}_{\infty}^*,\textbf{y}_{\infty}^*,\ldots)$, where
\begin{eqnarray*}
\textbf{y}_{\infty}^*(\textbf{x},\textbf{s}):=\left\{
\begin{array}{ll}
   \textbf{x} , & \mbox{if } ~ \textbf{x} \in  {\mathcal R}_{I}(\infty,\textbf{s})\\
   \textbf{b}_{\infty}(\textbf{s}) , & \mbox{if } ~ \textbf{x} \in {\mathcal R}_{II}(\infty,\textbf{s}) \\
   \Bigl(f_{\infty}^1(x^2,\textbf{s}) ,x^2 \Bigr) , & \mbox{if } ~ \textbf{x} \in  {\mathcal R}_{III-A}(\infty,\textbf{s})\\
   \Bigl(x^1 ,f_{\infty}^2(x^1,\textbf{s}) \Bigr) , & \mbox{if } ~ \textbf{x} \in  {\mathcal R}_{III-B}(\infty,\textbf{s})\\
\end{array} \right.,
\end{eqnarray*}
and for all $\textbf{x} \in {\mathcal R}_{IV}(\infty,\textbf{s})$, there exists an optimal control action, $\textbf{y}_{\infty}^*(\textbf{x},\textbf{s})$, which satisfies:
\begin{eqnarray*}
\textbf{c}_{\textbf{s}}^{\transpose}\left[\textbf{y}_{\infty}^*(\textbf{x},\textbf{s}) - \textbf{x} \right] = P~.
\end{eqnarray*}
\item[(c)] $\lim\limits_{n \rightarrow \infty}\textbf{b}_{n}(\textbf{s}) = \textbf{b}_{\infty}(\textbf{s})$ for all $\textbf{s} \in {\cal S}$.
\item[(d)]
$\lim\limits_{n \rightarrow \infty}f_{n}^1(x^2,\textbf{s}) = f_{\infty}^1(x^2,\textbf{s})$ for all $x^2 \in [d^2,\infty)$ and $\textbf{s} \in {\cal S}$.
\item[(e)] $\lim\limits_{n \rightarrow \infty}f_{n}^2(x^1,\textbf{s}) = f_{\infty}^2(x^1,\textbf{s})$ for all $x^1 \in [d^1,\infty)$ and $\textbf{s} \in {\cal S}$.
\end{itemize}
\end{theorem}
A detailed proof of Theorem \ref{Th:two:infd} is included in Appendix B.
\subsection{Structure of the Optimal Policy for the Infinite Horizon Average Expected Cost Problems}
In this section, we again use the vanishing discount approach to show that the structure of the optimal policy for the finite horizon expected cost and infinite horizon discounted expected cost problems extends to the infinite horizon average expected cost problem. As in Section \ref{Se:pwl:average}, we make explicit the dependence of the value functions and optimal policies from 
the corresponding infinite horizon discounted expected cost problem on the discount factor, $\alpha$.
\begin{theorem} \label{Th:two:infa}
For all $\alpha \in [0,1)$, define:
\begin{eqnarray}
m_{\infty,\alpha}&:=&\inf\limits_{\substack{\textbf{x} \in \Real_+^2 \\ \textbf{s} \in {\cal S}}} V_{\infty,\alpha}(\textbf{x},\textbf{s}), \label{Eq:mdef}\\
\rho^*&:=&\lim\limits_{\alpha \nearrow 1} (1-\alpha) \cdot m_{\infty,\alpha},~\hbox{and} \label{Eq:rho_conv} \\
w_{\infty,\alpha}(\textbf{x},\textbf{s})&:=& V_{\infty,\alpha}(\textbf{x},\textbf{s}) - m_{\infty,\alpha},~\forall \textbf{x} \in \Real_+^2,~\forall \textbf{s} \in {\cal S}.  \label{Eq:wdef}
\end{eqnarray}
Then:
\begin{itemize}
\item[(a)] There exists a continuous function $w_{\infty,1}(\cdot,\cdot)$ and a selector $\textbf{y}_{\infty,1}^*(\cdot,\cdot)$ that satisfy the ACOE:
\begin{align} \label{Eq:inf_acoe}
\rho^* + w_{\infty,1}(\textbf{x},\textbf{s}) 
&=
\min_{\textbf{y} \in \tilde{\cal A}^{\textbf{d}}(\textbf{x},\textbf{s})}
\left\{
\begin{array}{l}
\textbf{c}_{\textbf{s}}^{\transpose}[\textbf{y}-\textbf{x}]
+\textbf{h}\left(\textbf{y}-\textbf{d}\right) \\
 +\Expectation \bigl[w_{\infty,1}(\textbf{y}-\textbf{d},\textbf{S}^{\prime})\bigm| \textbf{S}=\textbf{s}\bigr]
 \end{array}
\right\} \\
&=\textbf{c}_{\textbf{s}}^{\transpose}\Bigl[\textbf{y}_{\infty,1}^*(\textbf{x},\textbf{s})-\textbf{x}\Bigr] + \textbf{h}\Bigl(\textbf{y}_{\infty,1}^*(\textbf{x},\textbf{s})-\textbf{d}\Bigr) \nonumber \\
&~~+\Expectation \left[w_{\infty,1}\Bigl(\textbf{y}_{\infty,1}^*(\textbf{x},\textbf{s})-\textbf{d},\textbf{S}^{\prime}\Bigr) \middle| \textbf{S} =\textbf{s}\right],~\forall \textbf{x} \in \Real_+^2,~\forall \textbf{s} \in {\cal S}. \nonumber
\end{align}
\item[(b)] The stationary policy $\boldsymbol{{\pi_{\infty,1}^*}}=(\textbf{y}_{\infty,1}^*,\textbf{y}_{\infty,1}^*,\ldots)$ is optimal for Problem (\textbf{P3}) in the case of two receivers with linear power-rate curves.
\item[(c)] The resulting optimal average cost beginning from any initial state $(\textbf{x},\textbf{s}) \in \Real_+^2 \times {\cal S}$ is $\rho^*$.
\item[(d)] For every 
increasing sequence of discount factors $\{\alpha(l)\}_{l=1,2,\ldots}$ approaching 1, there exists a subsequence $\{\alpha(l_i)\}_{i=1,2,\ldots}$ approaching 1
such that:
\begin{eqnarray*}
w_{\infty,1}(\textbf{x},\textbf{s})&=\lim\limits_{i \rightarrow \infty} w_{\infty,\alpha(l_i)}(\textbf{x},\textbf{s}),~\forall \textbf{x} \in \Real_+^2,~\forall \textbf{s} \in {\cal S}.
\end{eqnarray*}
Therefore, for every $\textbf{s} \in {\cal S}$, $w_{\infty,1}(\textbf{x},\textbf{s})$ is convex and supermodular in $\textbf{x}$.
\item[(e)] For every $(\textbf{x},\textbf{s}) \in \Real_+^2 \times {\cal S}$ and increasing sequence of discount factors $\{\alpha(l)\}_{l=1,2,\ldots}$ approaching 1, there exists a subsequence $\{\alpha(l_i)\}_{i=1,2,\ldots}$ approaching 1 and a sequence $\{\textbf{x}(i)\}_{i=1,2,\ldots}$ approaching $\textbf{x}$ such that:
\begin{eqnarray*}
    \textbf{y}_{\infty,1}^*(\textbf{x},\textbf{s})&=\lim\limits_{i \rightarrow \infty} \textbf{y}_{\infty,\alpha(l_i)}^*(\textbf{x}(i),\textbf{s})~.
\end{eqnarray*}
\item[(f)] There exists an optimal stationary policy with the same structure as statement (b) in Theorem \ref{Th:two:infd}.
\end{itemize}
\end{theorem}
A detailed proof of Theorem \ref{Th:two:infa} is included in Appendix C.

\subsection{Discussion}

At first glance, the structure of the optimal policy described in Theorem \ref{Th:two_users:structure} may also seem analogous to the structure of the optimal policy for the two-item resource-constrained inventory problem with \emph{deterministic} prices and \emph{stochastic} demands (i.e., the reverse of our problem), originally studied by Evans in \cite{evans}, and revisited in \cite{decroix}-\nocite{shaoxiang}\cite{janakiraman}. The structure of the optimal control action at each time for that problem can also be described in terms of seven regions that look essentially the same as those shown in Figure \ref{Fig:two:opt_policy}.\footnote{In the case of deterministic prices and stochastic demands, the boundaries of the regions do not depend on the ordering price (corresponding to the channel conditions $\textbf{s}$ in our case), because the vector of ordering prices is deterministic.}  However, there are two fundamental differences that distinguish these two problems.

First, the function $\tilde{G}_n(\cdot)$ in the deterministic price and stochastic demand inventory problem that corresponds to our function $G_n(\cdot,\textbf{s})$ has an additional structural property that Chen calls \emph{$\mu$-difference monotone} \cite{shaoxiang}.
This property is equivalent to the function $\tilde{G}_n(\cdot)$ not only being supermodular, but also submodular with respect to a partial order introduced by Antoniadou in \cite{antoniadou_thesis,antoniadou} called the \emph{direct value order} (see \cite{shuman_thesis} for further details). This functional property leads to two additional structural results on the optimal control action: (i) when the initial vector of inventories (corresponds to the vector of receivers' buffer levels in our problem) is in region ${\mathcal{R}}_{IV-B}(n)$, there exists an optimal control action such that $\textbf{y}_n^*(\textbf{x}) \preceq \textbf{b}_n$; and (ii) when the initial vector of inventories is in region ${\mathcal{R}}_{IV-A}(n)$ (respectively, ${\mathcal{R}}_{IV-C}(n)$), there exists an optimal control action that includes not ordering any of item 2 (respectively, item 1), corresponding to not transmitting any packets to user 2 (respectively, user 1) in our problem. Due to the time-varying channel conditions, this property does not hold for our function $G_n(\cdot,\textbf{s})$, and these two additional statements on the structure of the optimal policy are not true in general for our problem, as shown by the following example.

\begin{example} \label{Ex:one}
Consider a single sender transmitting to two statistically identical receivers, whose channel conditions are IID over time and independent of each other. The power-rate curves are linear, and the possible per packet power costs are 1.750 (best possible channel condition), 2.000, 2.001, and 2.100 (worst possible channel condition). The associated probabilities of each user experiencing these channel conditions are 0.4, 0.4, 0.1, and 0.1, respectively. The total power constraint in each slot is $P=4.2$, and 1 packet is removed from each receiver's buffer at the end of each time slot (i.e., $\textbf{d}=(1,1)$). We consider a finite horizon problem with the discount rate $\alpha=1$, and no holding costs. We are interested in the optimal control action with $T=3$ time slots remaining, and the current channel conditions are such that it costs 2.000 units of power to transmit a packet to user 1, and 2.001 units of power to transmit a packet to user 2.

Exactly solving the dynamic program shows that the unique global minimizer of the function $G_3(\cdot,\cdot,\textbf{s}_3)$ is the vector $(\frac{101}{75},\frac{101}{75})$. However, if the vector of starting receiver buffer levels at time $T=3$ is $\textbf{x}_3=(0.2,0.2)$, the unique optimal scheduling decision in the slot is to transmit 0.8 packets to user 2, and use the remaining power for transmission to user 1, which results in 1.2996 packets being sent to user 1. A diagram of this optimal control action is shown in Figure \ref{Fig:two:example}. The interesting thing to note here is that despite being power-constrained (the vector of starting buffer levels is in Region ${\cal R}_{IV-B}$), the unique optimal scheduling decision calls for filling user 1's buffer beyond its critical number $b_3^1(\textbf{s}_3)=\frac{101}{75}$. That is, the optimal scheduling decision brings the buffer levels from Region ${\cal R}_{IV-B}$ to Region ${\cal R}_{III-B}$ rather than Region ${\cal R}_{II}$.
\end{example}
\begin{figure}[htbp]
\centerline {
\includegraphics[width=5in]{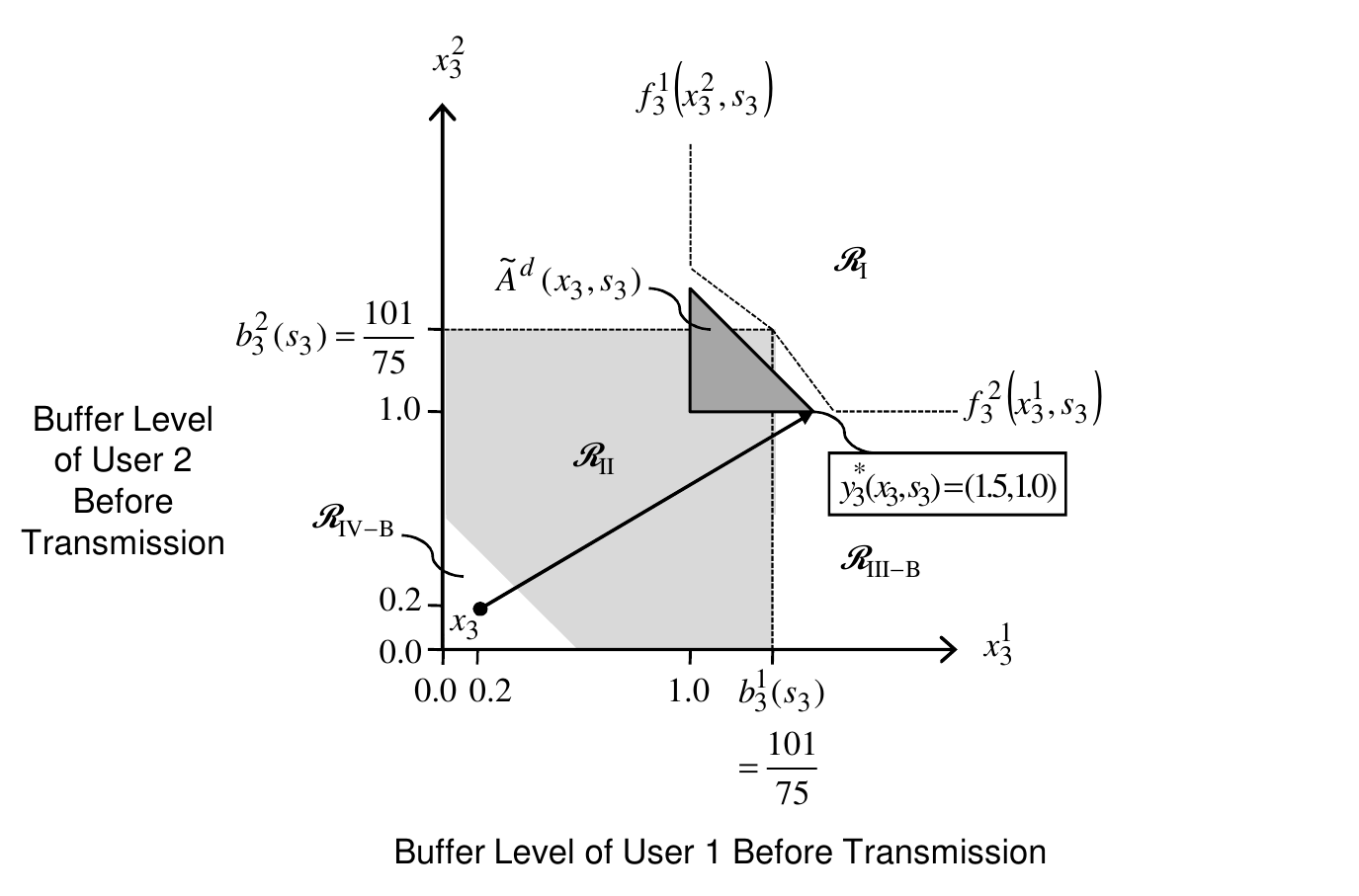}
} \caption{Optimal scheduling decision with 3 slots remaining in Example \ref{Ex:one}. The action space is represented by the triangle $\tilde{{\cal A}}^{\textbf{d}}(\textbf{x}_3,\textbf{s}_3)$. The critical vector $\textbf{b}_3(\textbf{s}_3)$ is not reachable from the starting buffer levels $\textbf{x}_3=(0.2,0.2)$. The unique optimal control action is to choose $\textbf{y}_3(\textbf{x}_3,\textbf{s}_3)$ (the buffer levels after transmission but before playout) to be (1.5, 1.0). The interesting feature of the example is that even though $\textbf{x}_3 \preceq \textbf{b}_3(\textbf{s}_3)$, we have $\textbf{y}_3^*(\textbf{x}_3,\textbf{s}_3) \npreceq \textbf{b}_3(\textbf{s}_3)$.}  \label{Fig:two:example}
\end{figure}

The second fundamental difference is also a consequence of the time-varying channel conditions in our model.
In the infinite horizon version of the two-item inventory problem with deterministic prices and stochastic demands, the critical numbers are time-invariant. Combined with the above property that it is optimal to not order inventory so as to move out of regions ${\mathcal R_{II}}$ and ${\mathcal R_{IV-B}}$, the time-invariant critical numbers mean that the region ${\mathcal R_{II}} \cup {\mathcal R_{IV-B}}$ (i.e., the lower-left square below the critical vector) is a ``stability'' region. Eventually, the vector of inventories enters this region under the optimal ordering policy, and once it does, it never leaves. This behavior both simplifies the analysis and opens the door for new mathematical techniques, such as analyzing \emph{shortfall} to compute the critical numbers \cite{tayur, janakiraman}. In our Problems (\textbf{P2}) and (\textbf{P3}), even though the boundaries of the seven regions for each possible channel condition are time-invariant, no such stability region exists, because the critical numbers vary over time due to the time-varying channel conditions. This makes it significantly more difficult to determine optimal and near-optimal policies.


%% file: discussion.tex
\section{Extensions} \label{Se:discussion}
In this section, we discuss 
the relaxation of the strict underflow constraints and the extension to the general case of $M$ receivers.

\subsection{Relaxation of the Strict Underflow Constraints}\label{Se:relaxation}
In some applications, it may not be the case that the peak power per slot is always sufficient to transmit one slot's worth of packets to each receiver, even under the worst channel conditions. In this case, a more appropriate model is to relax the strict underflow constraints, and allow underflow at a cost. One way to model this situation is to allow the receivers' queues to be negative, with a negative buffer level representing the number of packets that the playout process is behind. Then, in addition to the holding costs assessed on positive buffer levels, shortage costs are assessed on negative buffer levels. With some minor alterations to the proofs, it is straightforward to show that as long as the shortage cost function is a convex function of the negative buffer level, the structural results of Theorems \ref{Th:str:finite}, \ref{Th:pwl:fin} and \ref{Th:two_users:structure} are essentially unchanged by the relaxation of the strict underflow constraints to loose underflow constraints with penalties on underflow. This is not too surprising as the strict underflow constraint case we consider can be thought of as the limiting case as the penalties on underflow go to infinity.\footnote{Tracking the number of packets that the playout process is behind in this manner corresponds to the \emph{complete backlogging} assumption in inventory theory. An alternate model is to say that a packet is of no use once it misses its deadline, penalize missed packets, and keep the receiver queue length at zero. This model corresponds to the \emph{lost sales} assumption in inventory theory.}

\subsection{Extension to the General Case of $M$ Receivers} \label{Se:m_users}

Our ongoing work 
includes examining the extension to the most general case of $M$ receivers. It is unlikely that the structure of the optimal policy in this case has a simple, intuitive, and implementable form. Therefore, our approach is to find lower bounds on the value function and a feasible policy whose expected cost is as close as possible to these bounds. One simple lower bound to the value function can be found by relaxing the per slot peak power constraint of $P$ units of total power allocated to all users, and allowing up to $P$ units of power to be allocated to \emph{each} receiver in a single slot (for a total of up to $M\cdot P$). The advantage of this technique is that it is easy to compute the lower bound, as the $M$-dimensional problem separates into $M$ instances of the 1-dimensional problem we know how to solve from Section \ref{Se:lowSNR}. However, the resulting bound is likely to be loose. A second lower bounding method we are investigating 
is the information relaxation method of Brown, Smith, and Sun \cite{brown}. The main idea there is to assume the scheduler has access to future channel states (corresponding to the \emph{non-causal} or \emph{offline} model often considered in the literature), but penalize the scheduler for using this information. A clever choice of the penalty function often leads to tight lower bounds on the value function. A third method is the Lagrangian relaxation method discussed in \cite{hawkins,adelman_weak}. For our problem, this method is equivalent to relaxing the per slot peak power constraint to an average power constraint (i.e., the scheduler may allocate more than $P$ units of power in some slots, but the average power consumed per slot over the duration of the horizon cannot exceed $P$). Like the first method we mentioned, the resulting relaxed problem under this method can be separated into $M$ instances of a 1-dimensional problem, this time with an average power constraint of $\frac{P}{M}$ instead of a strict power constraint of $P$ for each receiver. A fourth lower bounding method is the linear programming approach to approximate dynamic programming discussed in \cite{adelman_weak}, \cite{seidmann}, and \cite{defarias}. The idea there is to formulate the dynamic program as a linear program,
and approximate the value functions as linear combinations of a set of basis functions. For a more in-depth comparison of the Lagrangian relaxation and approximate linear programming approaches, see \cite{adelman_weak}. Once lower bounds to the value function are determined from any of these methods, feasible policies can be generated based on our structural results or via one-step greedy optimization with the lower bounds substituted into the right-hand side of the dynamic programming equation. 

These same numerical techniques are most likely also the best way to approximate the boundaries of the seven regions of the two receiver optimal policy, and determine a near-optimal split of the power $P$ between the two receivers when the vector of starting receiver buffer levels is in the power-constrained region ${\mathcal R}_{IV}(n,\textbf{s})$.

The results we have presented in this paper 
are useful not only in terms of the intuition they provide, but also in generating feasible policies for the most general case of $M$ receivers and solving subproblems resulting from the relaxation methods described above.


%% file: conclusion.tex
\section{Conclusion}\label{Se:conclusion}
In this paper, we considered the problem of transmitting data to one or more receivers 
over a shared wireless channel in a manner that minimizes power consumption and prevents the
receivers' buffers from emptying. 
We showed that under 
the finite horizon
discounted expected cost, infinite horizon discounted expected cost, and infinite horizon average expected cost criteria, the optimal transmission
policy to a single receiver under linear power-rate curves has a modified base-stock structure. When the power-rate curves are generalized to piecewise-linear power-rate curves, the optimal transmission policy to a single receiver has a finite generalized base-stock structure.
For the special case when holding costs are linear, the stochastic process
representing the channel condition evolution over time is IID, 
and
the maximum number of packets that can be transmitted at any given marginal power cost in a slot is an integer multiple of the drainage rate of the receiver's buffer, we presented an efficient method to compute the critical numbers that fully characterize the modified base-stock and finite generalized base-stock policies. 

We also analyzed the structure of the optimal transmission policy for the case of two receivers. In some sense, the structure of the optimal policy was shown to be an extension of the modified base-stock policy; however, the peak power constraint couples the optimal scheduling of the two data streams, and the time-varying channel conditions may result in 
counterintuitive optimal scheduling decisions that 
are not possible 
in the analogous inventory theory problems.

The extension to the most general case of $M$ receivers is quite complex, and it is likely that numerical approximation techniques need to be used to develop further insights on the nature of the optimal policy. We presented a few possible approaches that constitute ongoing work in that regard.



%

%% file: appendix.tex
\section{Appendix A - Finite Horizon Proofs}\label{Se:appendix1}

\subsection{Proof of Theorem \ref{Th:str:finite}} \label{Se:fin_proof}
\input{appendix_structure_finite_proof}
\subsection{Proof of Theorem \ref{Th:pwl:fin}}
\input{appendix_pwl_finite_proof}

\subsection{Proof of Theorem \ref{Th:pwl_calc}}
\input{appendix_special_finite_proof}

\subsection{Proof of Theorem \ref{Th:two_users:properties}}
\input{appendix_two_user_properties_proof}
\subsection{Proof of Theorem \ref{Th:two_users:structure}}
\input{appendix_two_user_structure_proof}

\section{Appendix B - Infinite Horizon Discounted Expected Cost Proofs}\label{Se:appendix2}

\subsection{Proof of Theorem \ref{Th:one:infinite}}
\input{appendix_structure_infinite_proof}
%

\subsection{Proof of Theorem \ref{Th:two:infd}}
\input{appendix_two_infinite_discounted_proof}

\section{Appendix C - Infinite Horizon Average Expected Cost Proofs}\label{Se:appendix3}
In this section, we prove Theorem \ref{Th:two:infa} using the vanishing discount approach (see, e.g., \cite{lerma2}).
The proof of Theorem \ref{Th:pwl:infa} is nearly identical, and we note the few key differences.
\subsection{Proof of Theorem \ref{Th:two:infa}}
\input{appendix_two_inf_average_proof}


%% file: appendix_structure_finite_proof.tex
Before proceeding to the proof of Theorem \ref{Th:str:finite}, we present
a lemma due to Karush \cite{karush},
which is presented in \cite[pp.~237--238]{porteus}.
\medskip


\begin{lemma}[Karush, 1959] \label{Le:str:karush}
Suppose that $f:\Real \rightarrow \Real$ and that $f$ is convex on
$\Real$. For $v\leq w$, define $\tilde{f}(v,w):=\min\limits_{z \in
[v,w]}{f(z)}.$ Then it follows that:
\begin{itemize}
\item[(a)] $\tilde{f}$ can be expressed as $\tilde{f}(v,w) = F_1(v) + F_2(w)$, where
$F_1$ is convex nondecreasing and $F_2$ is convex nonincreasing on $\Real$.
\item[(b)] Suppose that $S$ is a minimizer of $f$ over $\Real$.
Then $\tilde{f}$ can be expressed as:
\begin{eqnarray*}
\tilde{f}(v,w) = \left\{
\begin{array}{ll}
f(v), & \mbox{ if } S \leq v \\
f(S), & \mbox{ if } v \leq S \leq w \\
f(w), & \mbox{ if } w \leq S
\end{array}
~~.\right.
\end{eqnarray*}
\end{itemize}
\end{lemma}
%
%
%
~\\

{\itshape Proof of Theorem \ref{Th:str:finite}: } We present the proof in three parts.

\emph{\textbf{Part I - Modified Base-Stock Structure:~}}
Recall the dynamic programming equation (\ref{Eq:y_act}): 
\begin{eqnarray*}
V_n(x,s)
=-c_s \cdot x + \min\limits_{\max(x,d)\leq y \leq x + \frac{P}{c_s}}
\left\{ g_n(y,s)
\right\}~,~n=N,N-1,\ldots,1~, 
\end{eqnarray*}
where $g_n(y,s):=c_s \cdot y + h(y-d) +\alpha \cdot \Expectation
\bigl[V_{n-1}(y-d,S_{n-1})\mid S_n=s\bigr]$. We now show by
induction on $n$ that the following statements are true for every
$n \in \left\{1,2,\ldots,N\right\}$ and all
$s \in {\cal{S}}$:
\begin{itemize}
\item[(i)] $g_n(y,s)$ is convex in $y$ on $[d,\infty)$.
\item[(ii)] $\lim_{y\rightarrow\infty}{g_n(y,s)} = \infty $.
\item[(iii)]$V_n(x,s)$ is convex in $x$ on $\Real_+$.
\end{itemize}
\medskip

\noindent \underline{Base Case}: $n=1$ \\
\noindent Let $s_1 \in {\cal{S}}$ be arbitrary. We have
$g_1(y,s_1)=c_{s_1} \cdot y + h(y-d)$, which clearly satisfies (i) and (ii). $y_1^*(x,s_1)=\max(x,d)$ and thus $V_1(x,s_1)=c_{s_1} \cdot (d-x)^+ +
h\Bigl((x-d)^+\Bigr)$, which is convex in $x$. We conclude (i)-(iii) are true at time $n=1$, for all $s \in {\cal{S}}$.
\medskip

\noindent \underline{Induction Step}: We now assume (i)-(iii) are true for $n=m-1$ and all $s
\in {\cal{S}}$, and show they hold for $n=m$ and an arbitrary $s_m \in {\cal{S}}$. Let ${s}_{m-1} \in {\cal{S}}$ also be
arbitrary. $V_{m-1}(y-d,{s}_{m-1})$ is convex in $y$,
so $g_m(y,s_m)$ is 
convex in $y$ as
it is the sum of an affine function, $c_{s_m} \cdot y$, a convex function, $h(y-d)$, and a
nonnegative weighted sum/integral of convex functions, $\alpha \cdot
\Expectation \bigl[V_{m-1}(y-d,S_{m-1})\mid S_m=s_m\bigr]$ (see, e.g., \cite[Section
3.2]{boyd} for the relevant results on convexity-preserving
operations). To show (ii) for $n=m$, we have
$\lim\limits_{y\rightarrow\infty}{g_m(y,s_m)} \geq
\lim\limits_{y\rightarrow\infty}{c_{s_m} \cdot y}= \infty$,
where the inequality follows from $V_{m-1}(x,{s}_{m-1})\geq 0,
\forall x \in \Real_+, \forall {s}_{m-1} \in {\cal{S}}$ and $h(y-d)\geq 0$.
Moving on to (iii), we have:
\begin{eqnarray*}
V_m(x,s_m)&=&-c_{s_m} \cdot x + \min\limits_{\max(x,d)\leq
y \leq x + \frac{P}{c_{s_m}}} \left\{ g_m(y,s_m) \right\} \\
&=& -c_{s_m}
\cdot x + F_1(\max(x,d)) + F_2(x+\frac{P}{c_{s_m}}),
\end{eqnarray*}
\noindent where, by Lemma \ref{Le:str:karush}, $F_1$ is convex
nondecreasing and $F_2$ is convex nonincreasing. $F_1(\max(x,d))$ is also convex
in $x$, as it is the composition of a convex increasing function
with a convex function, and 
$V_m(x,s_m)$ is therefore convex in
$x$. 
%
\noindent This concludes the induction step, and we conclude
(i)-(iii) are true for all $n \in \{1,2,\ldots,N\}$.

Next, we define the critical numbers $b_n(s)$ for all $n \in \{1,2,\ldots,N\}$ and $s \in {\cal{S}}$:
\begin{eqnarray*}
b_n(s) := \min \left\{\hat{y} \in
[d,\infty):~g_n(\hat{y},s)=\min\limits_{y\in [d,\infty)}
g_n(y,s) \right\}~.
\end{eqnarray*}
\noindent Note that by properties (i) and (ii) from the above induction, the minimum of $g_n(\cdot,s)$ over $[d,\infty)$ is achieved, and the set of minimizers over $[d,\infty)$ is a nonempty closed, convex set. Thus, $b_n(s)$ is well-defined.
The form of
$y_n^*(x,s)$, (\ref{Eq:str:y_star}), then follows from part (b) of Karush's
result, Lemma \ref{Le:str:karush}, with $g_n(y,s)$ playing the
role of $f$, $\max(x,d)$ the role of $v$, $x + \frac{P}{c_s}$ the
role of $w$, and $b_n(s)$ the role of $S$.

\emph{\textbf{Part II - Monotonicity of Thresholds in Time:~}}
In this section, we prove (\ref{Eq:str:inc_n}). We showed above that the optimal action with one time slot remaining is $y_1^*(x,s) = \max(x,d)$, for all $s \in {\cal{S}}$. This is precisely the policy suggested by
(\ref{Eq:str:y_star}) with $b_1(s)=d$, as $\frac{P}{c_s}$ is at least as great as $d$. Thus, we conclude the far right equality in (\ref{Eq:str:inc_n}) holds: $b_1(s) = d,~\forall s \in {\cal{S}}$.

In order to show the far left inequality in (\ref{Eq:str:inc_n}), we claim more generally that $b_n(s) \leq n \cdot d$, for all $n$ and $s$. This follows from a simple interchange argument, as all packets transmitted beyond $n\cdot d$ incur transmission costs and holding costs for the duration of the horizon; however, they do not satisfy the playout requirements in any remaining slot. Thus, a policy that transmits enough packets to fill the buffer up to $n \cdot d$ at time $n$ is strictly superior to a policy that transmits more packets.

Next, we prove:
\begin{eqnarray}\label{Eq:str_pr:mon_n}
b_{n+1}(s)\geq b_n(s),~\forall s \in {\cal{S}},~ \forall n \in \left\{1,2,\ldots,N-1\right\}~.
\end{eqnarray}
\noindent By Topkis' Theorem 2.8.1 \cite[pg. 76]{topkis}, in order to show (\ref{Eq:str_pr:mon_n}), it suffices to show that for all $s \in {\cal{S}}$, $n \in \left\{1,2,\ldots,N-1\right\}$, and $y^1,y^2 \in [d,(n+1)\cdot d]$, $y^1>y^2$ implies:
\begin{eqnarray} \label{Eq:str_pr:sub}
g_{n+1}\left(y^1,s\right)-g_{n}\left(y^1,s\right) \leq g_{n+1}\left(y^2,s\right) - g_{n}\left(y^2,s\right)~.
\end{eqnarray}
We let $s \in {\cal{S}}$ be arbitrary, and proceed by induction on the time slot $n$.   \\
\noindent \underline{Base Case}: $n=1$ \\
For all $y \in [d,2d]$,
\begin{eqnarray*}
g_2\left(y,s\right)-g_1\left(y,s\right)&=& \alpha \cdot \Expectation \left[V_1\left(y-d,S_1\right) \mid S_2=s\right]\\
&=&\alpha \cdot \Expectation\left[c_{S_1} \middle| S_2=s\right] \cdot (2d-y)~,
\end{eqnarray*}
which is decreasing in $y$ as $\Expectation\left[c_{S_1} \middle| S_2=s\right] > 0$.
%

\noindent \underline{Induction Step}: We assume that (\ref{Eq:str_pr:sub}) is true for all $n=1,2,\ldots,m-1$ and $s \in {\cal{S}}$. We wish to show it is true for $n=m$. Let 
$y^1,y^2 \in [d,(m+1)\cdot d]$ be arbitrary, with $y^1 > y^2$. Also, let $\hat{s} \in {\cal{S}}$ be arbitrary. Define:
\begin{eqnarray*}
{\beta}_1 &:=&  \min \left\{\argmin_{\max(y^1-d,d) \leq \hat{y} \leq y^1 -d + \frac{P}{{c}_{\hat{s}}}} \left\{g_{m-1}(\hat{y},\hat{s}) \right\} \right\} \\
\hbox{and } ~~{\beta}_2 &:=& \min \left\{
\argmin_{\max(y^2-d,d)\leq \hat{y} \leq y^2 -d + \frac{P}{c_{\hat{s}}}}
\left\{g_{m}(\hat{y},c_{\hat{s}}) \right\} \right\}~.
\end{eqnarray*}
\noindent Note that:
\begin{align}
&\max\left(y^1-d,d\right) \leq \beta_1 \leq \beta_1 \vee \beta_2 \leq y^1-d+\frac{P}{c_{\hat{s}}}~,\hbox{ and } \label{Eq:str_pr:b1} \\
&\max\left(y^2-d,d\right) \leq \beta_1 \wedge \beta_2 \leq \beta_2 \leq y^2-d+\frac{P}{c_{\hat{s}}}~. \label{Eq:str_pr:b2}
\end{align}
Then we have:
\begin{align}
&\min\limits_{\max(y^1-d,d)\leq \hat{y} \leq y^1-d+\frac{P}{c_{\hat{s}}}}\left\{g_{m}\left(\hat{y},\hat{s}\right) \right\}
-\min\limits_{\max(y^1-d,d)\leq \hat{y} \leq y^1-d+\frac{P}{c_{\hat{s}}}}\left\{g_{m-1}\left(\hat{y},\hat{s}\right) \right\} \nonumber \\
&\leq g_{m}\left(\beta_1 \vee \beta_2,\hat{s}\right) - g_{m-1}\left(\beta_1,\hat{s}\right)  \label{Eq:str_pr:min_cond} \\
&\leq g_{m}\left(\beta_2,\hat{s}\right) - g_{m-1}\left(\beta_1 \wedge \beta_2,\hat{s}\right) \label{Eq:str_pr:min_condzz}\\
&\leq
\min\limits_{\max(y^2-d,d)\leq \hat{y} \leq y^2-d+\frac{P}{c_{\hat{s}}}}\left\{g_{m}\left(\hat{y},\hat{s}\right) \right\}
-\min\limits_{\max(y^2-d,d)\leq \hat{y} \leq y^2-d+\frac{P}{c_{\hat{s}}}}\left\{g_{m-1}\left(\hat{y},\hat{s}\right) \right\}~. \label{Eq:str_pr:min_condzzz}
\end{align}
\noindent Equation \eqref{Eq:str_pr:min_cond} follows from \eqref{Eq:str_pr:b1} and \eqref{Eq:str_pr:min_condzzz} follows from \eqref{Eq:str_pr:b2}. If $\beta_2 \geq \beta_1$, \eqref{Eq:str_pr:min_condzz} holds with equality. Otherwise, it follows from the induction hypothesis.
  Since $\hat{s}$ was arbitrary, (\ref{Eq:str_pr:min_condzzz}) holds for all $\hat{s} \in {\cal{S}}$. Therefore, combined with the fact that the Markov process $\left\{S_n\right\}_{n=N,N-1,\ldots,1}$ is homogeneous, (\ref{Eq:str_pr:min_condzzz}) implies:
\begin{eqnarray} \label{Eq:str_pr:exp_min_cond}
\begin{array}{l}
\Expectation \biggl[\min\limits_{\max\left(y^1-d,d\right)\leq \hat{y} \leq y^1-d+\frac{P}{c_{S_{m}}}}\left\{g_{m}\left(\hat{y},S_{m} \right) \right\} \mid S_{m+1} = s \biggr] \\
~~-~\Expectation \biggl[\min\limits_{\max\left(y^1-d,d\right)\leq \hat{y} \leq y^1-d+\frac{P}{c_{S_{m-1}}}}\left\{g_{m-1}\left(\hat{y},S_{m-1} \right) \right\} \mid S_m = s \biggr] \\
\leq ~
\Expectation \biggl[\min\limits_{\max\left(y^2-d,d\right)\leq \hat{y} \leq y^2-d+\frac{P}{c_{S_{m}}}}\left\{g_{m}\left(\hat{y},S_{m} \right) \right\} \mid S_{m+1} = s \biggr] \\
~~-~
\Expectation \biggl[\min\limits_{\max\left(y^2-d,d\right)\leq \hat{y} \leq y^2-d+\frac{P}{c_{S_{m-1}}}}\left\{g_{m-1}\left(\hat{y},S_{m-1} \right) \right\} \mid S_m = s \biggr]~.
\end{array}
\end{eqnarray}
Finally, we have:
\begin{align}
&g_{m+1}(y^1,s)-g_m(y^1,s) \nonumber \\
&= \alpha \cdot \Expectation\left[V_m(y^1-d,S_m)\middle|S_{m+1}=s\right]-\alpha \cdot \Expectation\left[V_{m-1}(y^1-d,S_{m-1})\middle|S_{m}=s\right] \nonumber \\
&=  \alpha \cdot \Expectation \biggl[\min\limits_{\max\left(y^1-d,d\right)\leq \hat{y} \leq y^1-d+\frac{P}{c_{S_{m}}}}\left\{g_{m}\left(\hat{y},S_{m} \right) \right\} \mid S_{m+1} = s \biggr] \nonumber \\
&~~-\alpha \cdot \Expectation \biggl[\min\limits_{\max\left(y^1-d,d\right)\leq \hat{y} \leq y^1-d+\frac{P}{c_{S_{m-1}}}}\left\{g_{m-1}\left(\hat{y},S_{m-1} \right) \right\} \mid S_m = s \biggr] \label{Eq:step1of} 
\end{align}
\begin{align}
&\leq \alpha \cdot \Expectation \biggl[\min\limits_{\max\left(y^2-d,d\right)\leq \hat{y} \leq y^2-d+\frac{P}{c_{S_{m}}}}\left\{g_{m}\left(\hat{y},S_{m} \right) \right\} \mid S_{m+1} = s \biggr] \nonumber \\
&~~-\alpha \cdot \Expectation \biggl[\min\limits_{\max\left(y^2-d,d\right)\leq \hat{y} \leq y^2-d+\frac{P}{c_{S_{m-1}}}}\left\{g_{m-1}\left(\hat{y},S_{m-1} \right) \right\} \mid S_m = s \biggr] \label{Eq:step1ofzzzz} \\
&= \alpha \cdot \Expectation\left[V_m(y^2-d,S_m)\middle|S_{m+1}=s\right]-\alpha \cdot \Expectation\left[V_{m-1}(y^2-d,S_{m-1})\middle|S_{m}=s\right] \label{Eq:step1ofzzzzz} \\
&= g_{m+1}(y^2,s)-g_m(y^2,s)~. \nonumber
\end{align}
Here, \eqref{Eq:step1of} and \eqref{Eq:step1ofzzzzz} follow from the fact that $\Expectation \left[c_{S_{m-1}} \mid S_m =s \right] = \Expectation \left[c_{S_{m}} \mid S_{m+1} =s \right]$, and \eqref{Eq:step1ofzzzz} follows from \eqref{Eq:str_pr:exp_min_cond}. This completes the induction step, and the proof of (\ref{Eq:str:inc_n}).
\emph{\textbf{Part III - Monotonicity of Thresholds in the Channel Condition:~}}
Finally, we show (\ref{Eq:str:inc_c}), the monotonicity of the thresholds in the channel condition, when the channel condition process is IID. The far left inequality follows from the same interchange argument described above, showing $b_n(s) \leq n \cdot d$ for all $s$ and $n$.
%
We now show the far
right equality of (\ref{Eq:str:inc_c}), $b_n(s_{\worst})=d$. To satisfy feasibility, we must have
$b_n(s) \geq d$ for all $n \in \left\{1,2,\ldots,N \right\}$ and
$s \in {\cal{S}}$. To see that $b_n(s_{\worst}) \leq d$, assume the channel condition at time $n$ is $s_{\worst}$, and consider two
control policies satisfying (\ref{Eq:str:y_star}), with the same
critical numbers $b_m(s)$, for all times $m<n$.
At time $n$, the first policy, $\boldsymbol{\pi^1}$,
transmits according to (\ref{Eq:str:y_star}), with critical number $b_n(s_{\worst})=d+ \epsilon$
($\epsilon > 0$), and the second, $\boldsymbol{\pi^2}$, transmits according to (\ref{Eq:str:y_star}), with critical number $b_n(s_{\worst})=d$. These two strategies
result in the same control action at time $n$ if $x_n \geq d+ \epsilon$, and we have already shown it is not
optimal to fill the buffer beyond $n \cdot d$, so we only need to consider the
case where $x_n < d + \epsilon$ and $\epsilon \leq (n-1) \cdot d$.
Let $Z_n^1,Z_{n-1}^1,\ldots,Z_1^1$ and  $Z_n^2,Z_{n-1}^2,\ldots,Z_1^2$
be random variables representing the number of packets transmitted at times $n,n-1,\ldots,1$ by
$\boldsymbol{\pi^1}$ and $\boldsymbol{\pi^2}$, respectively.
If $d \leq x_n \leq d + \epsilon$, then $Z_n^2=0$ and $Z_n^1 - Z_n^2=Z_n^1=\min\left\{\frac{P}{c_{\max}},d + \epsilon - x_n \right\}.$
If $x_n < d$, then $Z_n^2 = d-x_n$, $Z_n^1 = \min \left\{\frac{P}{c_{\max}},d + \epsilon - x_n \right\}$, and $Z_n^1-Z_n^2= \min \left\{\frac{P}{c_{\max}}-d+x_n,\epsilon \right\}$. Thus, for all $x_n < d + \epsilon$, we have $Z_n^1-Z_n^2 \geq 0$. If $Z_n^1-Z_n^2=0$, the two control policies result in the same actions for all remaining times, and therefore result in the same expected cost. So we only need to consider the case where $\lambda := Z_n^1-Z_n^2 >0$.
Because the critical numbers at times $n-1,n-2,\ldots,1$ are the same for both policies,
for any realization, $\omega$, of the channel condition over future times, we have $Z_m^1(\omega)
\leq Z_m^2(\omega),~\forall m \in \{n-1,\ldots,1\}$. Moreover, because the scheduler must satisfy the playout requirements for the last $n$ slots, we have $\sum_{m=1}^{n-1}
(Z_m^2(\omega)-Z_m^1(\omega))=\lambda$; i.e., over the remainder of the horizon,
an extra $\lambda$ packets are transmitted under the second
policy.
The total discounted holding costs from time $n$ until the end of the horizon are therefore lower for
$\boldsymbol{\pi^2}$ than $\boldsymbol{\pi^1}$,
because the number of packets remaining after transmission in each slot is never greater under policy $\boldsymbol{\pi^2}$. Furthermore, the
total discounted transmission costs of the extra $\lambda$ packets
are also lower for $\boldsymbol{\pi^2}$ as they are transmitted at the maximum cost $c_{\max}$
under $\boldsymbol{\pi^1}$, and transmitted later (and therefore discounted more heavily) under $\boldsymbol{\pi^2}$. Thus, the total discounted transmission plus
holding costs are lower for $\boldsymbol{\pi^2}$ under all
realizations, and the expected discounted cost of
$\boldsymbol{\pi^2}$ is lower than $\boldsymbol{\pi^1}$. We
conclude $b_n(s_{\worst})=d$.

To show $c_{s^1} \leq c_{s^2}$ implies $b_n(s^1) \geq b_n(s^2)$, we follow Kalymon's methodology for the proof of Theorem 1.3 in \cite{kalymon}. For all $y \in [d,\infty)$, we have:
\begin{eqnarray} \label{Eq:str_pr:chan1}
g_n\left(y,s^2\right) &=& c_{s^2} \cdot y + h(y-d) + \alpha \cdot \Expectation \left[V_{n-1}\left(y-d,S_{n-1}\right) \right] \nonumber \\
&=&
(c_{s^2}-c_{s^1}) \cdot y + c_{s^1} \cdot y + h(y-d) + \alpha \cdot \Expectation \left[V_{n-1}\left(y-d,S_{n-1}\right) \right] \nonumber \\
&=& (c_{s^2}-c_{s^1}) \cdot y + g_n\left(y,s^1\right)~.
\end{eqnarray}
\noindent Assume $b_n(s^1) < b_n(s^2)$ for some $n \in \left\{1,2,\ldots,N\right\}$ and $s^1,s^2 \in {\cal{S}}$, with $c_{s^1} \leq c_{s^2}$. Substituting first $y=b_n(s^1)$ and then $y=b_n(s^2)$ into (\ref{Eq:str_pr:chan1}) yields:
\begin{eqnarray}\label{Eq:str_pr:chan2}
\left(c_{s^2}-c_{s^1}\right) \cdot b_n\left(s^1\right) + g_n\left(b_n\left(s^1\right),s^1\right) &=& g_n\left(b_n\left(s^1\right),s^2\right) \nonumber \\
&\geq & g_n\left(b_n\left(s^2\right),s^2\right) \nonumber \\
&=& \left(c_{s^2}-c_{s^1}\right) \cdot b_n\left(s^2\right) + g_n\left(b_n\left(s^2\right),s^1\right)~.
\end{eqnarray}
\noindent Yet, $c_{s^1} \leq c_{s^2}$ and $b_n(s^1) < b_n(s^2)$ imply:
\begin{eqnarray} \label{Eq:str_pr:chan3}
\left(c_{s^2}-c_{s^1}\right) \cdot b_n\left(s^1\right)
< \left(c_{s^2}-c_{s^1}\right) \cdot b_n\left(s^2\right)~.
\end{eqnarray}
\noindent Equations (\ref{Eq:str_pr:chan2}) and (\ref{Eq:str_pr:chan3}) imply:
\begin{eqnarray*}
 g_n\left(b_n\left(s^1\right),s^1\right) > g_n\left(b_n\left(s^2\right),s^1\right)~,
\end{eqnarray*}
\noindent which clearly contradicts the fact that $b_n\left(s^1\right)$ is a global minimizer of $g_n\left(\cdot,s^1\right)$. We conclude that $c_{s^1} \leq c_{s^2}$ implies $b_n(s^1) \geq b_n(s^2)$, completing the proofs of (\ref{Eq:str:inc_c}) and Theorem \ref{Th:str:finite}. \qedsymbol 

%% file: appendix_pwl_finite_proof.tex
While the proof is similar in spirit to the proof of a finite generalized base-stock policy in \cite[pp.~324--334]{bensoussan_book}, some key differences include the introduction of (i) stochastic channel conditions (ordering costs); (ii) the underflow constraint $x+z \geq d$; and (iii) the power constraint $z \leq \tilde{z}_{\max}(s)$.

We show by induction on $n$ that the following two statements are true for every $n \in \{1,2,\ldots,N\}$ and $s \in \cal{S}$:
\begin{itemize}
\item[(i)] $V_n(x,s)$ is convex in $x$ on $\Real_+$.
    \item[(ii)] There exists a nonincreasing sequence of critical numbers $\bigl\{b_{n,k}(s)\bigr\}_{k \in \left\{-1,0,1,\ldots,K\right\}}$ such that the optimal control action with $n$ slots remaining is given by:
\begin{eqnarray}\label{Eq:str:gen_base}
z_n^*(x,s) := \left\{
\begin{array}{ll}
   \tilde{z}_{k-1}(s) , & \mbox{if } ~b_{n,k}(s)-\tilde{z}_{k-1}(s) \leq x < b_{n,k-1}(s)-\tilde{z}_{k-1}(s)~, \\ 
       &~~~~~~~~~~~~~~~~~~~~~~~~~~~~~~~~~~~~~~~~~~~k \in \{0,1,\ldots,K\} \\
   b_{n,k}(s)-x , & \mbox{if } ~b_{n,k}(s)-\tilde{z}_k(s) \leq x < b_{n,k}(s)-\tilde{z}_{k-1}(s)~, \\ 
       &~~~~~~~~~~~~~~~~~~~~~~~~~~~~~~~~~~~~~~~k \in \{0,1,\ldots,K-1\} \\
   b_{n,K}(s)-x , & \mbox{if } ~b_{n,K}(s)-\tilde{z}_{\max}(s) \leq x < b_{n,K}(s)-\tilde{z}_{K-1}(s) \\
   \tilde{z}_{\max}(s) , & \mbox{if } ~0 \leq x < b_{n,K}(s)-\tilde{z}_{\max}(s)
   \end{array} \right.
\end{eqnarray}
\end{itemize}

\noindent \underline{Base Case}: $n=1$ 
\begin{eqnarray}
V_1(x,s) &=& \min_{\max(0,d-x) \leq z \leq \tilde{z}_{\max}(s)}
\left\{{c}(z,s)+{h}({x}+z-{d})\right\} \label{Eq:str:v1} \\
&=& c\bigl(\max\left\{0,d-x\right\},s\bigr)+ {h}\bigl(\max\left\{0,x-d\right\}\bigr)~, \nonumber
\end{eqnarray}
which is convex because $c(\cdot,s)$ and $h(\cdot)$ are both convex and nondecreasing functions, and $\max\left\{0,d-x\right\}$ and $\max\left\{0,x-d\right\}$ are both convex functions (see, e.g., \cite[Section
3.2]{boyd} for the relevant results on convexity-preserving
operations). Further,
let $b_{1,-1}(s)=\infty$ and $b_{1,k}(s)=d$ for all $k \in \{0,1,\ldots,K\}$. Then (\ref{Eq:str:gen_base}) is equivalent to $z_1^*(x,s)=\max\{0,d-x\}$, which clearly achieves the minimum in (\ref{Eq:str:v1}).
\medskip

\noindent \underline{Induction Step}: We now assume (i)-(ii) are true for $n=m-1$ and all $s
\in {\cal{S}}$, and show they hold for $n=m$ and an arbitrary $s \in {\cal{S}}$. Let 
$\breve{x},\hat{x} \in \Real_+$ and $\theta \in [0,1]$ be arbitrary, and define $\bar{x}:=\theta \cdot \breve{x} + (1-\theta) \cdot \hat{x}$. We have:
\begin{eqnarray}
&&V_m(\theta \cdot \breve{x} + (1-\theta) \cdot \hat{x},{s}) \nonumber \\
&&= V_m(\bar{x},{s}) \nonumber \\
&&= \min_{\max(0,d-\bar{x}) \leq z \leq \tilde{z}_{\max}(s)}
\left\{ {c}(z,s)+{h}(\bar{x}+z-{d}) +\alpha \cdot \Expectation \bigl[V_{m-1}(\bar{x}+z-{d},{S}_{m-1})\mid S_m=s\bigr]
\right\} \nonumber \\
&&\leq \min_{\substack{\max\left\{0,d-\breve{x}\right\} \leq \breve{z} \leq \tilde{z}_{\max}(s)  \\ \max\left\{0,d-\hat{x}\right\} \leq \hat{z} \leq \tilde{z}_{\max}(s)}}
\left\{ \begin{array}{l}
{c}(\theta \cdot \breve{z} + (1-\theta) \cdot \hat{z},s)+{h}(\bar{x}+\theta \cdot \breve{z} + (1-\theta) \cdot \hat{z}-{d}) \\
+\alpha \cdot \Expectation \bigl[V_{m-1}(\bar{x}+\theta \cdot \breve{z} + (1-\theta) \cdot \hat{z}-{d},{S}_{m-1})\mid S_m=s\bigr]
\end{array}
\right\} \label{Eq:pwl:convexity_step1} \\
&&\leq \min_{\substack{\max\left\{0,d-\breve{x}\right\} \leq \breve{z} \leq \tilde{z}_{\max}(s)  \\ \max\left\{0,d-\hat{x}\right\} \leq \hat{z} \leq \tilde{z}_{\max}(s)}}
\left\{ \begin{array}{l}
\theta \cdot {c}(\breve{z},s) + (1-\theta) \cdot c(\hat{z},s)+ \\
\theta \cdot {h}(\breve{x} + \breve{z} - d) + (1-\theta) \cdot {h}(\hat{x} + \hat{z} - d) \\
+\alpha \cdot \theta \cdot \Expectation \bigl[V_{m-1}(\breve{x}+ \breve{z} -{d},{S}_{m-1})\mid S_m=s\bigr] \\
+\alpha \cdot (1-\theta) \cdot \Expectation \bigl[V_{m-1}(\hat{x}+ \hat{z} -{d},{S}_{m-1})\mid S_m=s\bigr]
\end{array}
\right\} \label{Eq:pwl:convexity_step}\\
&&=\theta \cdot \min_{\max\left\{0,d-\breve{x}\right\} \leq \breve{z} \leq \tilde{z}_{\max}(s)}
\left\{
\begin{array}{l}
{c}(\breve{z},s)+{h}(\breve{x}+\breve{z}-{d}) \\
+\alpha \cdot \Expectation \bigl[V_{m-1}(\breve{x}+\breve{z}-{d},{S}_{m-1})\mid S_m=s\bigr]
\end{array}
\right\}
\nonumber \\
&&~+(1-\theta) \cdot \min_{\max\left\{0,d-\hat{x}\right\} \leq \hat{z} \leq \tilde{z}_{\max}(s)}
\left\{
\begin{array}{l}
{c}(\hat{z},s)+{h}(\hat{x}+\hat{z}-{d}) \\
+\alpha \cdot \Expectation \bigl[V_{m-1}(\hat{x}+\hat{z}-{d},{S}_{m-1})\mid S_m=s\bigr]
\end{array}
\right\} \nonumber \\
&&=\theta \cdot V_m(\breve{x},{s}) + (1-\theta) \cdot V_m(\hat{x},{s})~, \nonumber
\end{eqnarray}
where (\ref{Eq:pwl:convexity_step}) follows from the convexity of $c(\cdot,s)$, $h(\cdot)$, and $\Expectation \bigl[V_{m-1}(\cdot, {S}_{m-1})\mid S_m=s\bigr]$, the last of which follows from the induction hypothesis.
Equation (\ref{Eq:pwl:convexity_step1}) follows from the fact that for every $\max\left\{0,d-\breve{x}\right\} \leq \breve{z} \leq \tilde{z}_{\max}(s)$ and $\max\left\{0,d-\hat{x}\right\} \leq \hat{z} \leq \tilde{z}_{\max}(s)$, there exists a $\max\left\{0,d-\bar{x}\right\} \leq \bar{z} \leq \tilde{z}_{\max}(s)$ (namely, $\bar{z}:=\theta \cdot \breve{z} + (1-\theta) \cdot \hat{z}$) such that:
\begin{eqnarray*}
&&{c}(\bar{z},s)+{h}(\bar{x}+\bar{z}-{d}) +\alpha \cdot \Expectation \bigl[V_{m-1}(\bar{x}+\bar{z}-{d},{S}_{m-1})\mid S_m=s\bigr] \\
&& =
{c}(\theta \cdot \breve{z} + (1-\theta) \cdot \hat{z},s)+{h}(\bar{x}+\theta \cdot \breve{z} + (1-\theta) \cdot \hat{z}-{d}) \\
&&~+\alpha \cdot \Expectation \bigl[V_{m-1}(\bar{x}+\theta \cdot \breve{z} + (1-\theta) \cdot \hat{z}-{d},{S}_{m-1})\mid S_m=s\bigr]~.
\end{eqnarray*}
This concludes the induction step for (i) and we now proceed to (ii).
%

Note first that $\tilde{g}_m(y,s) = h(y-d) + \alpha \cdot \Expectation \left[V_{m-1}\left(y-d,S_{m-1}\right) \middle| S_m = s\right]$ is convex in $y$, as $h(\cdot)$ is convex, and $V_{m-1}(x,s)$ is convex in $x$ for every $s \in {\cal S}$ by the induction hypothesis.
Let $b_{m,-1}(s) := \infty$ and
\begin{eqnarray*}
b_{m,k}(s) := \max\Bigl\{d, \inf\bigl\{b \bigm| \tilde{g}_m^{\prime+}(b,s) \geq -\tilde{c}_k(s) \bigr\} \Bigr\}~,~\forall k \in \{0,1,\ldots,K\}~,
\end{eqnarray*}
where $\tilde{g}_m^{\prime+}(b,s)$ represents the right derivative: 
\begin{eqnarray*}
\tilde{g}_m^{\prime+}(b,s) := \lim\limits_{y \downarrow b} \frac{\tilde{g}_m(y,s)-\tilde{g}_m(b,s)}{y-b}~,
\end{eqnarray*}
which is nondecreasing and continuous from the right, by the convexity of $\tilde{g}_m(\cdot,s)$ \cite[Section 24]{rockafellar}.
Note that $\bigl\{b_{m,k}(s)\bigr\}_{k \in \left\{-1,0,1,\ldots,K\right\}}$ is a nonincreasing sequence, because the sequence $\left\{\tilde{c}_k(s)\right\}_{k \in \{0,1,\ldots,K\}}$ is nondecreasing.
We show the optimal control action $z_m^*(x,s)$ is then given by (\ref{Eq:str:gen_base}), by considering the four exhaustive cases.

\noindent \underline{Case 1}: $b_{m,k}(s)-\tilde{z}_{k-1}(s) \leq x < b_{m,k-1}(s)-\tilde{z}_{k-1}(s)~,~k \in \{0,1,\ldots,K\}$ \\
\noindent
In order to show $z_m^*(x,s)$ is given by (\ref{Eq:str:gen_base}), it suffices to show:
\begin{eqnarray}
&&c^{\prime+}(z,s)+\tilde{g}_m^{\prime+}(x+z,s) < 0~,\hbox{ for }\max\{0,d-x\}\leq z <\tilde{z}_{k-1}(s)~,\hbox{ and} \label{Eq:str:wtsc2a} \\
&&c^{\prime+}(z,s)+\tilde{g}_m^{\prime+}(x+z,s) \geq 0~,\hbox{ for } \tilde{z}_{k-1}(s)\leq z \leq \tilde{z}_{\max}(s)~. \label{Eq:str:wtsc2b}
\end{eqnarray}
First, let $z \in \Bigl[\max\{0,d-x\},\tilde{z}_{k-1}(s)\Bigr)$ be arbitrary, and let $j\in\{0,1,\ldots,k-1\}$ be such that $z \in \Bigl[\tilde{z}_{j-1}(s),\tilde{z}_j(s)\Bigr)$.
If $b_{m,k-1}(s)=d$, then $b_{m,k}(s)=d$, as $d \leq b_{m,k}(s) \leq b_{m,k-1}(s)=d$. Yet, $b_{m,k}(s)=b_{m,k-1}(s)=d$ implies $d-\tilde{z}_{k-1}(s) \leq x < d-\tilde{z}_{k-1}(s)$, which is vacuous. Therefore, we need only consider $b_{m,k-1}(s)=\inf\bigl\{b \bigm| \tilde{g}_m^{\prime+}(b,s) \geq -\tilde{c}_{k-1}(s) \bigr\}$. By the construction of the piecewise-linear function $c(\cdot,s)$, $z < \tilde{z}_{k-1}(s)$ implies:
\begin{eqnarray}\label{Eq:str:aa0}
c^{\prime+}(z,s) \leq \tilde{c}_{k-1}(s)~.
\end{eqnarray}
We also have:
\begin{eqnarray*}
x+z < x + \tilde{z}_{k-1}(s) < b_{m,k-1}(s) = \inf\bigl\{b \bigm| \tilde{g}_m^{\prime+}(b,s) \geq -\tilde{c}_{k-1}(s) \bigr\}~,
\end{eqnarray*}
which implies:
\begin{eqnarray}\label{Eq:str:aa1}
\tilde{g}_m^{\prime+}\left(x+z,s\right) < -\tilde{c}_{k-1}(s)~.
\end{eqnarray}
Summing \eqref{Eq:str:aa0} and \eqref{Eq:str:aa1} yields \eqref{Eq:str:wtsc2a}.


Next, let $z \in \bigl[\tilde{z}_{k-1}(s),\tilde{z}_{\max}(s)\bigr]$ be arbitrary, so that by construction of $c(\cdot,s)$:
\begin{eqnarray}\label{Eq:str:ba}
c^{\prime+}(z,s) \geq \tilde{c}_k(s)~.
\end{eqnarray}
We also have:
\begin{eqnarray*}
x+z \geq x + \tilde{z}_{k-1}(s) \geq b_{m,k}(s) \geq \inf\bigl\{b \bigm| \tilde{g}_m^{\prime+}(b,s) \geq -\tilde{c}_k(s) \bigr\}~,
\end{eqnarray*}
which, in combination with the nondecreasing nature of $\tilde{g}_m^{\prime+}\left(\cdot,s\right)$, implies:
\begin{eqnarray}\label{Eq:str:bb}
\tilde{g}_m^{\prime+}\left(x+z,s\right) \geq 
\tilde{g}_m^{\prime+}\Bigl(\inf\bigl\{b \bigm| \tilde{g}_m^{\prime+}(b,s) \geq -\tilde{c}_k(s) \bigr\},s\Bigr)~. 
\end{eqnarray}
Because $\tilde{g}_m^{\prime+}\left(\cdot,s\right)$ is continuous from the right, 
\begin{eqnarray}\label{Eq:str:bbb}
\tilde{g}_m^{\prime+}\Bigl(\inf\bigl\{b \bigm| \tilde{g}_m^{\prime+}(b,s) \geq -\tilde{c}_k(s) \bigr\},s\Bigr) \geq -\tilde{c}_k(s)~.
\end{eqnarray}
Combining (\ref{Eq:str:bb}) and (\ref{Eq:str:bbb}), and summing with (\ref{Eq:str:ba}) yields (\ref{Eq:str:wtsc2b}).

\noindent \underline{Case 2}: $b_{m,k}(s)-\tilde{z}_k(s) \leq x < b_{m,k}(s)-\tilde{z}_{k-1}(s)~,~k \in \{0,1,\ldots,K-1\}$ \\
\noindent
In order to show $z_m^*(x,s)$ is given by (\ref{Eq:str:gen_base}), it suffices to show:
\begin{eqnarray}
&&c^{\prime+}(z,s)+\tilde{g}_m^{\prime+}(x+z,s) < 0~,\hbox{ for }\max\{0,d-x\} \leq z < b_{m,k}(s)-x~,\hbox{ and} \label{Eq:str:wtsc1a} \\
&&c^{\prime+}(z,s)+\tilde{g}_m^{\prime+}(x+z,s) \geq 0~,\hbox{ for } b_{m,k}(s)-x \leq z \leq \tilde{z}_{\max}(s)~. \label{Eq:str:wtsc1b}
\end{eqnarray}

First, let $z \in \Bigl[\max\{0,d-x\},b_{m,k}(s)-x\Bigr)$ be arbitrary. This case is vacuous if $b_{m,k}(s)=d$, so
%
$b_{m,k}(s)=\inf\bigl\{b \bigm| \tilde{g}_m^{\prime+}(b,s) \geq -\tilde{c}_k(s) \bigr\}$.
Thus, we have:
\begin{eqnarray*}
x+z<b_{m,k}(s)=\inf\bigl\{b \bigm| \tilde{g}_m^{\prime+}(b,s) \geq -\tilde{c}_k(s) \bigr\}~,
\end{eqnarray*}
which implies:
\begin{eqnarray} \label{Eq:str:e0}
\tilde{g}_m^{\prime+}(x+z,s) < -\tilde{c}_k(s)~.
\end{eqnarray}
Furthermore, from $z < b_{m,k}(s)-x \leq \tilde{z}_k(s)$ and the construction of the piecewise-linear function $c(\cdot,s)$, 
\begin{eqnarray} \label{Eq:str:e1}
c^{\prime+}(z,s) \leq \tilde{c}_k(s)~.
\end{eqnarray}
Summing \eqref{Eq:str:e0} and \eqref{Eq:str:e1} yields \eqref{Eq:str:wtsc1a}.

Next, let $z \in \bigl[b_{m,k}(s)-x,\tilde{z}_{\max}(s)\bigr]$ be arbitrary, so that $z \geq b_{m,k}(s)-x > \tilde{z}_{k-1}(s)$, which by the construction of the piecewise-linear function $c(\cdot,s)$ implies:
\begin{eqnarray}\label{Eq:str:ca}
c^{\prime+}(z,s) \geq \tilde{c}_k(s)~.
\end{eqnarray}
We also have
$x+z \geq b_{m,k}(s) \geq \inf\bigl\{b \bigm| \tilde{g}_m^{\prime+}(b,s) \geq -\tilde{c}_k(s) \bigr\}$. Therefore, because $\tilde{g}_m^{\prime+}(\cdot,s)$ is nondecreasing and continuous from the right, 
\begin{eqnarray}\label{Eq:str:cb}
\tilde{g}_m^{\prime+}\left(x+z,s\right) \geq \tilde{g}_m^{\prime+}\left(\inf\bigl\{b \bigm| \tilde{g}_m^{\prime+}(b,s) \geq -\tilde{c}_k(s) \bigr\},s\right) \geq -\tilde{c}_k(s)~.
\end{eqnarray}
Summing (\ref{Eq:str:ca}) and (\ref{Eq:str:cb}) yields (\ref{Eq:str:wtsc1b}).


\noindent \underline{Case 3}: $b_{m,K}(s)-\tilde{z}_{\max}(s) \leq  x < b_{m,K}(s)-\tilde{z}_{K-1}(s)$ \\
\noindent
This case is the same as Case 2, with $K$ in place of $k$, and $\tilde{z}_{\max}(s)$ in place of $\tilde{z}_k(s)$.

\noindent \underline{Case 4}: $0 \leq x < b_{m,K}(s)-\tilde{z}_{\max}(s)$ \\
\noindent
Let $z \in \bigl[\max\{0,d-x\},\tilde{z}_{\max}(s)\bigr)$ be arbitrary.
$\tilde{z}_{\max}(s) \geq d$ by assumption, so this case is vacuous if $b_{m,K}=d$. Thus, we have
$b_{m,K}(s)=\inf\bigl\{b \bigm| \tilde{g}_m^{\prime+}(b,s) \geq -\tilde{c}_K(s) \bigr\}$, which, in combination with $x+z < x+\tilde{z}_{\max} <  b_{m,K}(s)$, implies:
\begin{eqnarray} \label{Eq:str:f00}
\tilde{g}_m^{\prime+}\left(x+z,s\right) < -\tilde{c}_{K}(s)~.
\end{eqnarray}
Additionally, $z<\tilde{z}_{\max}(s)$ implies:
\begin{eqnarray}\label{Eq:str:f01}
c^{\prime+}(z,s) \leq \tilde{c}_{K}(s)~.
\end{eqnarray}
Summing \eqref{Eq:str:f00} and \eqref{Eq:str:f01} yields $c^{\prime+}(z,s)+\tilde{g}_m^{\prime+}\left(x+z,s\right)<0$ for all $z \in \bigl[\max\{0,d-x\},\tilde{z}_{\max}(s)\bigr)$, which implies $z_m^*(x,s)=\tilde{z}_{\max}(s)$.
\qedsymbol 

%% file: appendix_special_finite_proof.tex
We proceed in a manner similar to \cite{golabi85}, incorporating the per slot peak power constraints and the relaxing the linear ordering costs to piecewise-linear convex ordering costs.
Before proving Theorem \ref{Th:one:finite}, we state and prove two lemmas. Let
$\bar{{\boldsymbol{\pi}}}$ be a strategy that prescribes
transmitting according to \eqref{Eq:pwl:opt_structure}. 

\begin{lemma} \label{Le:app:gol_1}
If $\bar{\boldsymbol{{\boldsymbol{\pi}}}}$ is optimal for periods
$m-1,m-2,\ldots,1$, then
\begin{eqnarray} \label{Eq:app:lem_1}
\alpha \cdot \Expectation \left[V_{l-1}\bigl((r-1)\cdot d +
\eta,S\bigr)-V_{l-1}\bigl((r-1) \cdot d,S\bigr) \right] \geq -\eta
\cdot \bigl(\tilde{\gamma}_{l,r+1}+h\bigr),
\end{eqnarray}
\noindent for all $(l,r,\eta) \in {\cal{Z}}_1:=\left\{(l,r,\eta)\in \Nat \times \Nat \times [0,d]: 1 \leq l \leq m, 1 \leq r \leq l \right\}$.
\end{lemma}
\medskip

\begin{proof}
We proceed by induction on $l$.
\newpage
\noindent \underline{Base Case}: $l=1$ \\
$l=1$ implies $r=1$, so we have:
\begin{eqnarray*}
\begin{array}{l}
\alpha \cdot \Expectation \left[V_{l-1}\bigl((r-1)\cdot d + \eta,S\bigr)-V_{l-1}\bigl((r-1) \cdot d,S\bigr) \right] \\
=~ \alpha \cdot \Expectation \left[V_{0}\bigl(\eta,S\bigr)-V_{0}\bigl(0,S\bigr) \right] \\
=~ 0 \\
\geq ~ -\eta \cdot h \\
=~ -\eta \cdot \left(\tilde{\gamma}_{1,2} + h \right)~,
\end{array}
\end{eqnarray*}
\noindent and we conclude (\ref{Eq:app:lem_1}) holds for $l=1$.
\medskip

\noindent \underline{Induction Step}\\
Assume (\ref{Eq:app:lem_1}) is true for $l=2,3,\ldots,t$ and all
$r$ and $\eta$ such that $(l,r,\eta)\in {\cal{Z}}_1$. We show
(\ref{Eq:app:lem_1}) is true for $l=t+1$ by letting $r$ and $\eta$
be arbitrary such that $(t+1,r,\eta)\in {\cal{Z}}_1$. Note that
$(t+1,r,\eta)\in {\cal{Z}}_1$ implies $t \leq m-1$, so
$\bar{{\boldsymbol{\pi}}}$ is optimal at time $t$, and we have:
\begin{align*}%
&\alpha \cdot \Expectation \bigl[V_{t}\bigl((r-1)\cdot d +
\eta,S\bigr)-V_{t}\bigl((r-1) \cdot d,S\bigr) \bigr] \\
&= \sum\limits_{\bigl\{s:~b_{t,0}(s)\leq(r-1)\cdot d\bigr\}}\alpha
\cdot p(s) \cdot \Bigl[h \cdot \eta + \alpha \cdot \Expectation
\bigl[V_{t-1}\bigl((r-2)\cdot d + \eta,S
\bigr) - V_{t-1}\bigl((r-2)\cdot d,S \bigr) \bigr]\Bigr] \\
&+\sum_{k=0}^{K-1} \left\{
\begin{array}{l}
~~~~~~~~~~~~~~~~~~~~\smashoperator{\sum\limits_{\bigl\{s:\bigl(r-1+\tilde{L}_{k-1}(s)\bigr)\cdot d<
b_{t,k}(s)\leq\bigl(r-1+\tilde{L}_k(s)\bigr)\cdot d\bigr\}}}~~~~~-~\alpha \cdot p(s)
\cdot
\Bigl(\eta \cdot \tilde{c}_k(s)\Bigr) \\
+~~~~~~~~~~~~~~~~~\smashoperator{\sum\limits_{\left\{s:~b_{t,k+1}(s) \leq \bigl(r-1+\tilde{L}_{k}(s)\bigr) \cdot d < b_{t,k}(s) \right\}}}~~~~~
\alpha \cdot p(s) \cdot \left[h \cdot \eta + \alpha \cdot
\Expectation \Biggl[
\begin{array}{l}
V_{t-1}\Bigl(\bigl(r-2+\tilde{L}_k(s)\bigr)\cdot d + \eta,S \Bigr) \\
-{} V_{t-1}\Bigl(\bigl(r-2+\tilde{L}_k(s)\bigr)\cdot d,S \Bigr)
\end{array}
\Biggr]\right]
\end{array}
\right\} \\
&+{}~~~~~~~~~~~~~~~~~~~~~~~~
\smashoperator{\sum\limits_{\bigl\{s:\bigl(r-1+\tilde{L}_{K-1}(s)\bigr)\cdot d<
b_{t,K}(s)\leq\bigl(r-1+\tilde{L}_{\max}(s)\bigr)\cdot d\bigr\}}}~~~~~~~~-~\alpha \cdot p(s)
\cdot
\Bigl(\eta \cdot \tilde{c}_K(s)\Bigr) \\
&+{}~~~~~~~~~~~~~~~
\smashoperator{\sum\limits_{\bigl\{s:~b_{t,K}(s)>\bigl(r-1+\tilde{L}_{\max}(s)\bigr)\cdot d\bigr\}}}~~~~~~
\alpha \cdot p(s) \cdot \left[h \cdot \eta + \alpha \cdot
\Expectation \Biggl[
\begin{array}{l}
V_{t-1}\Bigl(\bigl(r-2+\tilde{L}_{\max}(s)\bigr)\cdot d + \eta,S \Bigr) \\
-{} V_{t-1}\Bigl(\bigl(r-2+\tilde{L}_{\max}(s)\bigr)\cdot d,S \Bigr)
\end{array}
\Biggr]\right] 
\end{align*}
\begin{align*}
&\geq \sum\limits_{\bigl\{s:~b_{t,0}(s)\leq(r-1)\cdot d\bigr\}}\alpha
\cdot p(s) \cdot \Bigl[-\eta \cdot \tilde{\gamma}_{t,r}\Bigr] \\
&~~~~~~+\sum_{k=0}^{K-1} \left\{
\begin{array}{l}
~~~~~~~~~~~~~~~~~~~~\smashoperator{\sum\limits_{\bigl\{s:\bigl(r-1+\tilde{L}_{k-1}(s)\bigr)\cdot d<
b_{t,k}(s)\leq\bigl(r-1+\tilde{L}_k(s)\bigr)\cdot d\bigr\}}}~~~~~-~\alpha \cdot p(s)
\cdot
\Bigl(\eta \cdot \tilde{c}_k(s)\Bigr) \\
+~~~~~~~~~~~~~~~~~\smashoperator{\sum\limits_{\left\{s:~b_{t,k+1}(s) \leq \bigl(r-1+\tilde{L}_{k}(s)\bigr) \cdot d < b_{t,k}(s) \right\}}}~~~~~
\alpha \cdot p(s) \cdot \left[-\eta \cdot \tilde{\gamma}_{t,r+\tilde{L}_k(s)}\right]
\end{array}
\right\} \\
&~~~~~~+{}~~~~~~~~~~~~~~~~~~~~~~~~
\smashoperator{\sum\limits_{\bigl\{s:\bigl(r-1+\tilde{L}_{K-1}(s)\bigr)\cdot d<
b_{t,K}(s)\leq\bigl(r-1+\tilde{L}_{\max}(s)\bigr)\cdot d\bigr\}}}~~~~~~~~-~\alpha \cdot p(s)
\cdot
\Bigl(\eta \cdot \tilde{c}_K(s)\Bigr) \\
&~~~~~~+{}~~~~~~~~~~~~~~~
\smashoperator{\sum\limits_{\bigl\{s:~b_{t,K}(s)>\bigl(r-1+\tilde{L}_{\max}(s)\bigr)\cdot d\bigr\}}}~~~~~~
\alpha \cdot p(s) \cdot \left[-\eta \cdot \tilde{\gamma}_{t,r+\tilde{L}_{\max}(s)}\right] \\
&~~~ \\
&=~~~~ -\alpha \cdot \eta \cdot
\left\{
\begin{array}{l}
    \sum\limits_{s:~\tilde{c}_0(s)\geq \tilde{\gamma}_{t,r}} p(s)\cdot\tilde{\gamma}_{t,r} \\
    ~~~~~+
    \sum\limits_{k=0}^{K-1}
    \left\{
    \begin{array}{l}
    \sum\limits_{s:~\tilde{\gamma}_{t,r+\tilde{L}_k(s)}\leq \tilde{c}_k(s)< \tilde{\gamma}_{t,r+\tilde{L}_{k-1}(s)}} p(s)\cdot \tilde{c}_k(s) \\
~~~~~+ \sum\limits_{s:~\tilde{c}_k(s) < \tilde{\gamma}_{t,r+\tilde{L}_k(s)} \leq \tilde{c}_{k+1}(s)}
p(s) \cdot \tilde{\gamma}_{t,r+\tilde{L}_k(s)}
\end{array}
\right\}
\\
~~~~~+ \sum\limits_{s:~\tilde{\gamma}_{t,r+\tilde{L}_{\max}(s)} \leq \tilde{c}_K(s) < \tilde{\gamma}_{t,r+\tilde{L}_{K-1}(s)}}
p(s) \cdot \tilde{c}_K(s) \\
~~~~~+ \sum\limits_{s:~\tilde{c}_K(s) < \tilde{\gamma}_{t,r+\tilde{L}_{\max}(s)}}
p(s) \cdot \tilde{\gamma}_{t,r+\tilde{L}_{\max}(s)} \\
\end{array}
\right\} \\
&~~~ \\
&=~~~~ -\eta \cdot \left(\tilde{\gamma}_{t+1,r+1} + h \right)~,
\end{align*}
\noindent where the inequality follows from the induction
hypothesis, and the penultimate equality follows from the definition $b_{n,k}:=j \cdot d$, if $\tilde{\gamma}_{n,j+1} \leq \tilde{c}_k(s) < \tilde{\gamma}_{n,j}$.
\noindent This concludes the induction step, and the proof of Lemma \ref{Le:app:gol_1}.
\end{proof}
\medskip

\medskip

\begin{lemma} \label{Le:app:gol_2}
If $\bar{{\boldsymbol{\pi}}}$ is optimal for periods
$m-1,m-2,\ldots,1$, then
\begin{eqnarray} \label{Eq:app:lem_2}
\alpha \cdot \Expectation \left[V_{l-1}\bigl((r-1)\cdot d -
\eta,S\bigr)-V_{l-1}\bigl((r-1) \cdot d,S\bigr) \right] \geq \eta
\cdot \bigl(\tilde{\gamma}_{l,r}+h\bigr),
\end{eqnarray}
\noindent for all $(l,r,\eta) \in {\cal{Z}}_2:=\left\{(l,r,\eta)\in \Nat \times
\Nat \times [0,d]: 2 \leq l \leq m, 2 \leq r \leq l \right\}$.
\end{lemma}
\medskip

\begin{proof}
We proceed by induction on $l$.
\medskip

\noindent \underline{Base Case}: $l=2$ \\
$l=2$ implies $r=2$, so we have:
\begin{align*}
&\alpha \cdot \Expectation \left[V_{l-1}\bigl((r-1)\cdot d - \eta,S\bigr)-V_{l-1}\bigl((r-1) \cdot d,S\bigr) \right] \\
&~~~~~~~~~~~~~~~~~~~~~~~~~~~~=~ \alpha \cdot \Expectation \left[V_{1}\bigl(d - \eta,S\bigr)-V_{1}\bigl(d,S\bigr) \right] \\
&~~~~~~~~~~~~~~~~~~~~~~~~~~~~=~ \alpha \cdot \Expectation \bigl[c(\eta,S) \bigr] \\
&~~~~~~~~~~~~~~~~~~~~~~~~~~~~=~ \eta \cdot
\left(\tilde{\gamma}_{2,2}+h \right),
\end{align*}
\noindent where the last equality follows from ${\gamma}_{2,2}=-h+
\alpha \cdot \Expectation [\tilde{c}_0(S)]$, and the fact that $\eta \leq \tilde{z}_0(s)$ for every $s \in {\cal S}$. So (\ref{Eq:app:lem_2}) holds with equality
for $l=2$.
\medskip

\noindent \underline{Induction Step} \\
Assume (\ref{Eq:app:lem_2}) is true for $l=2,3,\ldots,t$ and all
$r$ and $\eta$ such that $(l,r,\eta)\in {\cal{Z}}_2$. We show
(\ref{Eq:app:lem_2}) is true for $l=t+1$ by letting $r$ and $\eta$
be arbitrary such that $(t+1,r,\eta)\in {\cal{Z}}_2$. Note that
$(t+1,r,\eta)\in {\cal{Z}}_2$ implies $t \leq m-1$, so
$\bar{{\boldsymbol{\pi}}}$ is optimal at time $t$, and we have:
\begin{align*}
&\alpha \cdot \Expectation \bigl[V_{t}\bigl((r-1)\cdot d -
\eta,S\bigr)-V_{t}\bigl((r-1) \cdot d,S\bigr) \bigr] \\
&~~~ \\
&= \sum\limits_{\bigl\{s:~b_{t,0}(s)\leq(r-2)\cdot d\bigr\}}\alpha
\cdot p(s) \cdot \Bigl[-\eta \cdot h +
\alpha \cdot \Expectation \left[V_{t-1}\bigl((r-2)\cdot d -
\eta,S\bigr)-V_{t-1}\bigl((r-2) \cdot d,S\bigr) \right]
\Bigr] \\
&+\sum_{k=0}^{K-1} \left\{
\begin{array}{l}
~~~~~~~~~~~~~~~~~~~~\smashoperator{\sum\limits_{\bigl\{s:\bigl(r-2+\tilde{L}_{k-1}(s)\bigr)\cdot d<
b_{t,k}(s)\leq\bigl(r-2+\tilde{L}_k(s)\bigr)\cdot d\bigr\}}}~~~~~~\alpha \cdot p(s)
\cdot
\Bigl(\eta \cdot \tilde{c}_k(s)\Bigr) \\
+~~~~~~~~~~~~~~~~~\smashoperator{\sum\limits_{\left\{s:~b_{t,k+1}(s) \leq \bigl(r-2+\tilde{L}_{k}(s)\bigr) \cdot d < b_{t,k}(s) \right\}}}~~~~~
\alpha \cdot p(s) \cdot \left[-\eta \cdot h +
\alpha \cdot \Expectation \left[
\begin{array}{l}
V_{t-1}\Bigl(\bigl(r-2+\tilde{L}_k(s)\bigr)\cdot d -
\eta,S\Bigr) \\
-V_{t-1}\Bigl(\bigl(r-2+\tilde{L}_k(s)\bigr) \cdot d,S\Bigr)
\end{array}
\right]
\right]
\end{array}
\right\} \\
&+{}~~~~~~~~~~~~~~~~~~~~~~~~
\smashoperator{\sum\limits_{\bigl\{s:\bigl(r-2+\tilde{L}_{K-1}(s)\bigr)\cdot d<
b_{t,K}(s)\leq\bigl(r-2+\tilde{L}_{\max}(s)\bigr)\cdot d\bigr\}}}~~~~~~~~~\alpha \cdot p(s)
\cdot
\Bigl(\eta \cdot \tilde{c}_K(s)\Bigr) \\
&+{}~~~~~~~~~~~~~~~
\smashoperator{\sum\limits_{\bigl\{s:~b_{t,K}(s)>\bigl(r-2+\tilde{L}_{\max}(s)\bigr)\cdot d\bigr\}}}~~~~~~
\alpha \cdot p(s) \cdot \left[-\eta \cdot h +
\alpha \cdot \Expectation \left[
\begin{array}{l}
V_{t-1}\Bigl(\bigl(r-2+\tilde{L}_{\max}(s)\bigr)\cdot d -
\eta,S\Bigr) \\
-V_{t-1}\Bigl(\bigl(r-2+\tilde{L}_{\max}(s)\bigr) \cdot d,S\Bigr)
\end{array}
\right]\right] \\
&~~~ \\
&\geq \sum\limits_{\bigl\{s:~b_{t,0}(s)\leq(r-2)\cdot d\bigr\}}\alpha
\cdot p(s) \cdot \Bigl[\eta \cdot \tilde{\gamma}_{t,r-1}\Bigr] \\
&+\sum_{k=0}^{K-1} \left\{
\begin{array}{l}
~~~~~~~~~~~~~~~~~~~~\smashoperator{\sum\limits_{\bigl\{s:\bigl(r-2+\tilde{L}_{k-1}(s)\bigr)\cdot d<
b_{t,k}(s)\leq\bigl(r-2+\tilde{L}_k(s)\bigr)\cdot d\bigr\}}}~~~~~~\alpha \cdot p(s)
\cdot
\Bigl(\eta \cdot \tilde{c}_k(s)\Bigr) \\
+~~~~~~~~~~~~~~~~~\smashoperator{\sum\limits_{\left\{s:~b_{t,k+1}(s) \leq \bigl(r-2+\tilde{L}_{k}(s)\bigr) \cdot d < b_{t,k}(s) \right\}}}~~~~~
\alpha \cdot p(s) \cdot \left[\eta \cdot \tilde{\gamma}_{t,r-1+\tilde{L}_k(s)}\right]
\end{array}
\right\} \\
&+{}~~~~~~~~~~~~~~~~~~~~~~~~
\smashoperator{\sum\limits_{\bigl\{s:\bigl(r-2+\tilde{L}_{K-1}(s)\bigr)\cdot d<
b_{t,K}(s)\leq\bigl(r-2+\tilde{L}_{\max}(s)\bigr)\cdot d\bigr\}}}~~~~~~~~~\alpha \cdot p(s)
\cdot
\Bigl(\eta \cdot \tilde{c}_K(s)\Bigr) \\
&+{}~~~~~~~~~~~~~~~
\smashoperator{\sum\limits_{\bigl\{s:~b_{t,K}(s)>\bigl(r-2+\tilde{L}_{\max}(s)\bigr)\cdot d\bigr\}}}~~~~~~
\alpha \cdot p(s) \cdot \left[\eta \cdot \tilde{\gamma}_{t,r-1+\tilde{L}_{\max}(s)}\right] \\
\end{align*}
%
\begin{align*}
&=~~~~ \alpha \cdot \eta \cdot
\left\{
\begin{array}{l}
    \sum\limits_{s:~\tilde{c}_0(s)\geq \tilde{\gamma}_{t,r-1}} p(s)\cdot\tilde{\gamma}_{t,r-1} \\
    ~~~~~+
    \sum\limits_{k=0}^{K-1}
    \left\{
    \begin{array}{l}
    \sum\limits_{s:~\tilde{\gamma}_{t,r-1+\tilde{L}_k(s)}\leq \tilde{c}_k(s)< \tilde{\gamma}_{t,r-1+\tilde{L}_{k-1}(s)}} p(s)\cdot \tilde{c}_k(s) \\
~~~~~+ \sum\limits_{s:~\tilde{c}_k(s) < \tilde{\gamma}_{t,r-1+\tilde{L}_k(s)} \leq \tilde{c}_{k+1}(s)}
p(s) \cdot \tilde{\gamma}_{t,r-1+\tilde{L}_k(s)}
\end{array}
\right\}
\\
~~~~~+ \sum\limits_{s:~\tilde{\gamma}_{t,r-1+\tilde{L}_{\max}(s)} \leq \tilde{c}_K(s) < \tilde{\gamma}_{t,r-1+\tilde{L}_{K-1}(s)}}
p(s) \cdot \tilde{c}_K(s) \\
~~~~~+ \sum\limits_{s:~\tilde{c}_K(s) < \tilde{\gamma}_{t,r-1+\tilde{L}_{\max}(s)}}
p(s) \cdot \tilde{\gamma}_{t,r-1+\tilde{L}_{\max}(s)} \\
\end{array}
\right\} \\
&~~~ \\
&=~~~~ \eta \cdot \left(\tilde{\gamma}_{t+1,r} + h \right)~,
\end{align*}
\noindent where the inequality follows from the induction
hypothesis, and the penultimate equality again follows from the definition of $b_{n,k}(s)$.
\noindent This concludes the induction step, and the proof of Lemma \ref{Le:app:gol_2}.
\end{proof}
\medskip

We now return to the proof of Theorem \ref{Th:one:finite}. We first show
by induction that $V_n^{\bar{{\boldsymbol{\pi}}}}(x,s)=V_n(x,s),
\forall n\in \{1,2,\ldots,N\}$, $\forall s \in {\cal{S}},$ and $~\forall x\in \left\{0,d,2d,3d,\ldots \right\}$.   \\
\medskip

\noindent \underline{Base Case}: $n=1$ \\
With one slot remaining, we have:
\begin{eqnarray*}
V_1(x,s)
&=&
\min\limits_{\bigl\{\max(0,d-x) \leq z_1 \leq \tilde{z}_{\max}(s)\bigr\}}
\left\{
c(z_1,s) + h(x+z_1-d)
\right\} \nonumber \\
&=& c\Bigl(\max\{0,d-x\},s\Bigr) + h \Bigl(\max\{0,(x-d)\}
\Bigr)~,
\end{eqnarray*}
where the minimum is achieved by $z_1=\max\{0,d-x\}$.
$\tilde{\gamma}_{1,1}=\infty$ and $\tilde{\gamma}_{1,2}=0$, so $b_{1,k}(s) = d$ for all $s \in {\cal S}$ and $k \in \{0,1,\ldots,K\}$. Thus,
according to \eqref{Eq:pwl:opt_structure},
${\bar{z}}_1(x,s)$
is also equal to $\max\{0,d-x\}$, the optimal
amount.
\medskip

\noindent \underline{Induction Step} \\
Assume that for
$n=\{1,2,\ldots,m-1\},~V_n^{\bar{{\boldsymbol{\pi}}}}(x,s)=V_n(x,s),~\forall x\in \left\{0,d,2d,3d,\ldots \right\},~\forall s \in {\cal{S}}$. We show this is also true
for $n=m$ by considering first any strategy that transmits more
than $\bar{{\boldsymbol{\pi}}}$ at time $m$, and then any strategy
that transmits less than $\bar{{\boldsymbol{\pi}}}$ at time $m$.
Let $s \in {\cal{S}}$ be arbitrary.
 with $\tilde{\gamma}_{m,j_k+1}\leq \tilde{c}_k(s) <
\tilde{\gamma}_{m,j_k}$ so that $\bar{{\boldsymbol{\pi}}}$ prescribes
$b_{m,k}(s) = j_k \cdot d$ for $k \in \{0,1,\ldots,K\}$. Let ${\boldsymbol{\pi}}^q$ be a strategy
that at time $m$ transmits enough to satisfy the demands of slots
$m,m-1,m-2,\ldots,q+1,$ and $q$, and transmits optimally at times $m-1$,$m-2,\ldots,1$.
\medskip

\noindent \underline{Part I}: Do not transmit more than suggested by
$\bar{{\boldsymbol{\pi}}}$ at time $m$
\medskip

Let ${\boldsymbol{\pi}}^{\prime}(\epsilon)$ be a feasible strategy
with $z_m^{\prime}= {\bar{z}}_m+\epsilon,$ where $\epsilon>0$, and the optimal transmission policy at times $m-1$,$m-2,\ldots,1$.
We consider
four cases for the current 
buffer level $x$.
\medskip

\noindent \underline{Case (a)}: $j_k \cdot d-\tilde{z}_{k-1}(s) < x \leq j_{k-1} \cdot d-\tilde{z}_{k-1}(s),~k\in\{0,1,\ldots,K\}$ \\
In this case, ${\bar{z}}_m=\tilde{z}_{k-1}(s)$. Let $p$ be the integer such that $x+\tilde{z}_{k-1}(s)=p \cdot d$. Let $q,~\eta$ be such that $z_m^{\prime}=\tilde{z}_{k-1}(s)+\epsilon
=q\cdot d + \eta - x$ and $0\leq \eta < d~\Bigl($i.e., $q=\left\lfloor
\frac{z_m^{\prime}+x}{d} \right\rfloor$ and $\eta = z_m^{\prime}+x-q
\cdot d\Bigr)$. Thus, we have $q \geq p \geq j_k$.


Then we have:
\begin{eqnarray}
V_m^{{\boldsymbol{\pi}}^{\prime}(\epsilon)}(x,s) - V_m^{{\boldsymbol{\pi}}^q}(x,s) &=& c\Bigl(z_m^{\prime},s\Bigr) - c\Bigl(z_m^{\prime}-\eta,s\Bigr) \nonumber \\
&& +\eta \cdot h \nonumber \\
&& + \alpha \cdot \Expectation\Bigl[V_{m-1}\bigl((q-1)\cdot d + \eta,S\bigr)-V_{m-1}\bigl((q-1) \cdot d,S\bigr) \Bigr] \nonumber \\
&\geq& c\Bigl(z_m^{\prime},s\Bigr) - c\Bigl(z_m^{\prime}-\eta,s\Bigr)  \nonumber \\
&& - \eta \cdot \tilde{\gamma}_{m,q+1} \label{Eq:app:b_main3s}  \\
&\geq& c\Bigl(z_m^{\prime},s\Bigr) - c\Bigl(z_m^{\prime}-\eta,s\Bigr)  \nonumber \\
&& - \eta \cdot \tilde{\gamma}_{m,j_k+1} \label{Eq:app:b_main2s} \\
&\geq& \eta \cdot \bigl(\tilde{c}_k(s) - \tilde{\gamma}_{m,j_k+1} \bigr) \label{Eq:app:b_main1s} \\
&\geq&0. \label{Eq:app:b_mains}
\end{eqnarray}
\noindent 
Equation \eqref{Eq:app:b_main3s} follows from Lemma
\ref{Le:app:gol_1}, with $l=m$, $r=q$, and $\eta=\eta$.
Equation \eqref{Eq:app:b_main2s}
follows from $q+1 \geq j_k + 1$, which
implies $\tilde{\gamma}_{m,q+1}\leq \tilde{\gamma}_{m,j_k+1}$.
Equation \eqref{Eq:app:b_main1s} follows from $z_m^{\prime}-\eta \geq \tilde{z}_{k-1}(s)$ and the construction of $c(\cdot,s)$. Finally, \eqref{Eq:app:b_mains} follows from $\tilde{c}_k(s) \geq \tilde{\gamma}_{m,j_k+1}$, by construction of $j_k$, and we conclude:
\begin{eqnarray}\label{Eq:app:t_comp1b}
V_m^{{\boldsymbol{\pi}}^{\prime}(\epsilon)}(x,s) \geq V_m^{{\boldsymbol{\pi}}^q}(x,s)~.
\end{eqnarray}

Now let $t \in \left\{q+1, q+2, \ldots, m-p, m-p+1\right\}$ be
arbitrary. We have:
\begin{eqnarray}
V_m^{{\boldsymbol{\pi}}^{t-1}}(x,s) - V_m^{{\boldsymbol{\pi}}^t}(x,s)  &=& c\Bigl((m-t+2)\cdot d - x,s\Bigr) - c\Bigl((m-t+1)\cdot d - x,s\Bigr) \nonumber \\
&& +d \cdot h \nonumber \\
&& + \alpha \cdot \Expectation\Bigl[V_{m-1}\bigl((m-t+1)\cdot d,S\bigr)-V_{m-1}\bigl((m-t) \cdot d,S\bigr) \Bigr] \nonumber \\
&\geq&  c\Bigl((m-t+2)\cdot d - x,s\Bigr) - c\Bigl((m-t+1)\cdot d - x,s\Bigr)  \nonumber \\
&& - d \cdot \tilde{\gamma}_{m,m-t+2} \label{Eq:app:b_main3b}  \\
&\geq&  c\Bigl((m-t+2)\cdot d - x,s\Bigr) - c\Bigl((m-t+1)\cdot d - x,s\Bigr)  \nonumber \\
&& - d \cdot \tilde{\gamma}_{m,j_k+1} \label{Eq:app:b_main2b} \\
&\geq& d \cdot \bigl(\tilde{c}_k(s) - \tilde{\gamma}_{m,j_k+1} \bigr) \label{Eq:app:b_main1b} \\
&\geq&0. \label{Eq:app:b_mainb}
\end{eqnarray}
%
Equation \eqref{Eq:app:b_main3b} follows from Lemma \ref{Le:app:gol_1}, with
$l=m$, $r=m-t+1\leq m-q \leq m = l$, and $\eta=d$.
Equation \eqref{Eq:app:b_main2b} follows from:
\begin{eqnarray*}
t \leq m-p+1 ~~ \Leftrightarrow ~~ p+1 \leq m-t+2 ~~ \Rightarrow j_k + 1 \leq m-t+2 ~~ \Rightarrow
~~ \tilde{\gamma}_{m,j_k+1} \geq \tilde{\gamma}_{m,m-t+2}~.
\end{eqnarray*}
Equation \eqref{Eq:app:b_main1b} follows from the construction of $c(\cdot,s)$ and the fact that:
\begin{eqnarray*}
(m-t+1) \cdot d - x \geq \Bigl[m-(m-p+1)+1\Bigr] \cdot d - x = p \cdot d -x = \tilde{z}_{k-1}(s)~.
\end{eqnarray*}
Finally, \eqref{Eq:app:b_mainb} follows once again from $\tilde{c}_k(s) \geq \tilde{\gamma}_{m,j_k+1}$, by construction of $j_k$.
Rearranging (\ref{Eq:app:b_mainb}) yields:
\begin{eqnarray} \label{Eq:app:t_compb}
V_m^{{\boldsymbol{\pi}}^{t-1}}(x,s) \geq
V_m^{{\boldsymbol{\pi}}^t}(x,s),~\forall t\in
\left\{q+1,q+2,\ldots,m-p,m-p+1\right\}.
\end{eqnarray}
Noting that
$V_m^{\bar{{\boldsymbol{\pi}}}}(x,s)=V_m^{{\boldsymbol{\pi}}^{m-p+1}}(x,s)$,
(\ref{Eq:app:t_comp1b}) and repeated application of
(\ref{Eq:app:t_compb}) imply:
\begin{eqnarray*}
V_m^{\bar{{\boldsymbol{\pi}}}}(x,s)=V_m^{{\boldsymbol{\pi}}^{m-p+1}}(x,s)
\leq V_m^{{\boldsymbol{\pi}}^{m-p}}(x,s) \leq
\ldots \leq
V_m^{{\boldsymbol{\pi}}^{q+1}}(x,s) \leq
V_m^{{\boldsymbol{\pi}}^{q}}(x,s) \leq
V_m^{{\boldsymbol{\pi}}^{{\prime}}(\epsilon)}(x,s)~,
\end{eqnarray*}
and we conclude $\bar{{\boldsymbol{\pi}}}$ is at least as good as
${\boldsymbol{\pi}}^{\prime}(\epsilon)$.
\medskip


\noindent \underline{Case (b)}: $j_k \cdot d-\tilde{z}_{k}(s) < x \leq j_k \cdot d-\tilde{z}_{k-1}(s),~k\in\{0,1,\ldots,K-1\}$ \\
Let $q,\eta$ be such that $z_m^{\prime}=(m-q+1)\cdot d + \eta - x$ and
$0\leq \eta < d~\Bigl($i.e., $q=m+1-\left\lfloor
\frac{z_m^{\prime}+x}{d} \right\rfloor$ and $\eta =
z_m^{\prime}-(m-q+1) \cdot d - x\Bigr)$. Note that $m-q+1 \geq j_k$ by
the assumption that $z_m^{\prime} \geq {\bar{z}}_m = j_k \cdot d - x$.
Additionally, because $x \leq j_k \cdot d - \tilde{z}_{k-1}(s)$ and $m-q+1 \geq j_k$, we have:
\begin{eqnarray*}
(m-q+1) \cdot d - x \geq (m-q+1-j_k) \cdot d + \tilde{z}_{k-1}(s) \geq \tilde{z}_{k-1}(s)~,
\end{eqnarray*}
which implies:
\begin{eqnarray} \label{Eq:app:b_main5}
 c\Bigl((m-q+1)\cdot d + \eta - x,s\Bigr) - c\Bigl((m-q+1)\cdot d - x,s\Bigr) \geq \eta \cdot \tilde{c}_k(s)~.
\end{eqnarray}
Then we have:
\begin{eqnarray}
V_m^{{\boldsymbol{\pi}}^{\prime}(\epsilon)}(x,s) - V_m^{{\boldsymbol{\pi}}^q}(x,s) &=& c\Bigl((m-q+1)\cdot d + \eta - x,s\Bigr) - c\Bigl((m-q+1)\cdot d - x,s\Bigr) \nonumber \\
&& +\eta \cdot h \nonumber \\
&& + \alpha \cdot \Expectation\Bigl[V_{m-1}\bigl((m-q)\cdot d + \eta,S\bigr)-V_{m-1}\bigl((m-q) \cdot d,S\bigr) \Bigr] \nonumber \\
&\geq& c\Bigl((m-q+1)\cdot d + \eta - x,s\Bigr) - c\Bigl((m-q+1)\cdot d - x,s\Bigr) \nonumber \\
&& - \eta \cdot \tilde{\gamma}_{m,m-q+2} \label{Eq:app:b_main3}  \\
&\geq& c\Bigl((m-q+1)\cdot d + \eta - x,s\Bigr) - c\Bigl((m-q+1)\cdot d - x,s\Bigr) \nonumber \\
&& - \eta \cdot \tilde{\gamma}_{m,j_k+1} \label{Eq:app:b_main2} \\
&\geq& \eta \cdot \bigl(\tilde{c}_k(s) - \tilde{\gamma}_{m,j_k+1} \bigr) \label{Eq:app:b_main1} \\
&\geq&0. \label{Eq:app:b_main}
\end{eqnarray}
\noindent 
Equation \eqref{Eq:app:b_main3} follows from Lemma
\ref{Le:app:gol_1}, with $l=m$, $r=m-q\leq m-1$, and $\eta=\eta$.
Equation \eqref{Eq:app:b_main2}
follows from $m-q+2 \geq j_k + 1$, which
implies $\tilde{\gamma}_{m,m-q+2}\leq \tilde{\gamma}_{m,j_k+1}$.
Equation \eqref{Eq:app:b_main1} follows from \eqref{Eq:app:b_main5}. Finally, \eqref{Eq:app:b_main} follows from $\tilde{c}_k(s) \geq \tilde{\gamma}_{m,j_k+1}$, by construction of $j_k$, and we conclude:
\begin{eqnarray}\label{Eq:app:t_comp1}
V_m^{{\boldsymbol{\pi}}^{\prime}(\epsilon)}(x,s) \geq V_m^{{\boldsymbol{\pi}}^q}(x,s)~.
\end{eqnarray}

Now let $t \in \left\{q+1, q+2, \ldots, m-j_k, m-j_k+1\right\}$ be
arbitrary. We have:
\begin{eqnarray}
V_m^{{\boldsymbol{\pi}}^{t-1}}(x,s) - V_m^{{\boldsymbol{\pi}}^t}(x,s)  &=& c\Bigl((m-t+2)\cdot d - x,s\Bigr) - c\Bigl((m-t+1)\cdot d - x,s\Bigr) \nonumber \\
&& +d \cdot h \nonumber \\
&& + \alpha \cdot \Expectation\Bigl[V_{m-1}\bigl((m-t+1)\cdot d,S\bigr)-V_{m-1}\bigl((m-t) \cdot d,S\bigr) \Bigr] \nonumber \\
&\geq&  c\Bigl((m-t+2)\cdot d - x,s\Bigr) - c\Bigl((m-t+1)\cdot d - x,s\Bigr)  \nonumber \\
&& - d \cdot \tilde{\gamma}_{m,m-t+2} \label{Eq:app:b_main3a}  \\
&\geq&  c\Bigl((m-t+2)\cdot d - x,s\Bigr) - c\Bigl((m-t+1)\cdot d - x,s\Bigr)  \nonumber \\
&& - d \cdot \tilde{\gamma}_{m,j_k+1} \label{Eq:app:b_main2a} \\
&\geq& d \cdot \bigl(\tilde{c}_k(s) - \tilde{\gamma}_{m,j_k+1} \bigr) \label{Eq:app:b_main1a} \\
&\geq&0. \label{Eq:app:b_maina}
\end{eqnarray}
%
Equation \eqref{Eq:app:b_main3a} follows from Lemma \ref{Le:app:gol_1}, with
$l=m$, $r=m-t+1\leq m-q \leq m = l$, and $\eta=d$.
Equation \eqref{Eq:app:b_main2a} follows from:
\begin{eqnarray*}
t \leq m-j_k+1 ~~ \Leftrightarrow ~~ j_k+1 \leq m-t+2 ~~ \Rightarrow
~~ \tilde{\gamma}_{m,j_k+1} \geq \tilde{\gamma}_{m,m-t+2}~.
\end{eqnarray*}
Similarly to \eqref{Eq:app:b_main5}, equation \eqref{Eq:app:b_main1a} follows from the fact that:
\begin{eqnarray*}
(m-t+1) \cdot d - x \geq \Bigl[m-(m-j_k+1)+1\Bigr] \cdot d - x = j_k \cdot d -x \geq \tilde{z}_{k-1}(s)~.
\end{eqnarray*}
Finally, \eqref{Eq:app:b_maina} follows once again from $\tilde{c}_k(s) \geq \tilde{\gamma}_{m,j_k+1}$, by construction of $j_k$.
Rearranging (\ref{Eq:app:b_maina}) yields:
\begin{eqnarray} \label{Eq:app:t_comp}
V_m^{{\boldsymbol{\pi}}^{t-1}}(x,s) \geq
V_m^{{\boldsymbol{\pi}}^t}(x,s),~\forall t\in
\left\{q+1,q+2,\ldots,m-j_k,m-j_k+1\right\}.
\end{eqnarray}
Noting that
$V_m^{\bar{{\boldsymbol{\pi}}}}(x,s)=V_m^{{\boldsymbol{\pi}}^{m-j_k+1}}(x,s)$,
(\ref{Eq:app:t_comp1}) and repeated application of
(\ref{Eq:app:t_comp}) imply:
\begin{eqnarray*}
V_m^{\bar{{\boldsymbol{\pi}}}}(x,s)=V_m^{{\boldsymbol{\pi}}^{m-j_k+1}}(x,s)
\leq V_m^{{\boldsymbol{\pi}}^{m-j_k}}(x,s) \leq
\ldots \leq
V_m^{{\boldsymbol{\pi}}^{q+1}}(x,s) \leq
V_m^{{\boldsymbol{\pi}}^{q}}(x,s) \leq
V_m^{{\boldsymbol{\pi}}^{{\prime}}(\epsilon)}(x,s)~,
\end{eqnarray*}
and we conclude $\bar{{\boldsymbol{\pi}}}$ is at least as good as
${\boldsymbol{\pi}}^{\prime}(\epsilon)$.
\medskip

\noindent \underline{Case (c)}: $j_K \cdot d-\tilde{z}_{\max}(s) < x \leq j_K \cdot d-\tilde{z}_{K-1}(s)$ \\
Same as Case (b) with $K$ replacing $k$.

\noindent \underline{Case (d)}: $0 \leq x \leq j_K \cdot d-\tilde{z}_{\max}(s)$ \\ 
${\bar{z}}_m(x,s)=\tilde{z}_{\max}(s)$, the upper bound of the
action space, so it is not feasible to transmit more.
\medskip


\noindent \underline{Part II}: Do not transmit less than suggested by
$\bar{{\boldsymbol{\pi}}}$ at time $m$
\medskip

Let ${\boldsymbol{\pi}}^{\prime \prime}(\epsilon)$ be a feasible strategy
with $z_m^{\prime \prime}= {\bar{z}}_m-\epsilon,$ where $\epsilon>0$, and the optimal transmission policy at times $m-1$,$m-2,\ldots,1$.
To
satisfy feasibility, we require ${\bar{z}}_m-\epsilon \geq
\max(0,d-x)$. Define $\eta:=\epsilon - \left\lfloor \frac{\epsilon}{d} \right\rfloor \cdot d$, and note that $\eta \in [0,d)$. Let ${\boldsymbol{\pi}}_{\theta}^l$ be a strategy that at time $m$
satisfies the demands of periods $m,m-1,\ldots,l$, except for
$\theta$ units of the demand of period $l$, where $0\leq \theta \leq
d$, and behaves optimally in slots $m-1,m-2,\ldots,1$. We consider four exhaustive cases for the current buffer 
level
$x$.
\medskip

\noindent \underline{Case (a)}: $x> j_0 \cdot d$ \\
${\bar{z}}_m(x,s)=0$, the lower bound of the action space, so it
is not feasible to transmit less.
\medskip

\noindent \underline{Case (b)}: $j_{k} \cdot d - \tilde{z}_{k}(s) < x \leq j_{k} \cdot d - \tilde{z}_{k-1}(s),~k \in \{0,1,\ldots,K\}$, where we define 
$\tilde{z}_K(s):=\tilde{z}_{\max}(s)$ \\
Define $q:=m-j_k+1+\left\lfloor \frac{\epsilon}{d} \right\rfloor$.  By the feasibility of ${\boldsymbol{\pi}}^{\prime \prime}(\epsilon)$ and $\epsilon > 0$, we have \\
$q \in \left\{m-j_k+1, m-j_k+2, \ldots, m-2, m-1 \right\}$. Furthermore, we have:
\begin{eqnarray*}
[m-q+1] \cdot d - x = \left[m-\left(m-j_k+1+\left\lfloor \frac{\epsilon}{d} \right\rfloor\right)+1 \right] \cdot d - x \leq j_k \cdot d - x \leq \tilde{z}_k(s)~,
\end{eqnarray*}
which, by the construction of $c(\cdot,s)$, implies:
\begin{eqnarray} \label{Eq:case2b0}
c\Bigl((m-q+1)\cdot d - \eta -x,s\Bigr) - c\Bigl((m-q+1)\cdot d -x,s\Bigr) \geq -\eta \cdot \tilde{c}_k(s)~.
\end{eqnarray}
We now compare ${\boldsymbol{\pi}}_{\eta}^{q}$ and ${\boldsymbol{\pi}}_{0}^{q}$:
\begin{eqnarray}
V_m^{{\boldsymbol{\pi}}_{\eta}^{q}}(x,s) -  V_m^{{\boldsymbol{\pi}}_{0}^{q}}(x,s)
&=& c\Bigl((m-q+1)\cdot d - \eta -x,s\Bigr) - c\Bigl((m-q+1)\cdot d -x,s\Bigr) \nonumber \\
&& -h \cdot \eta \nonumber \\
&& + \alpha \cdot \Expectation\Bigl[V_{m-1}\bigl((m-q)\cdot d - \eta,S\bigr)-V_{m-1}\bigl((m-q) \cdot d,S\bigr) \Bigr] \nonumber \\
&\geq& c\Bigl((m-q+1)\cdot d - \eta -x,s\Bigr) - c\Bigl((m-q+1)\cdot d -x,s\Bigr) \nonumber \\
&&+\eta \cdot \tilde{\gamma}_{m,m-q+1} \label{Eq:case2b1} \\
&\geq& c\Bigl((m-q+1)\cdot d - \eta -x,s\Bigr) - c\Bigl((m-q+1)\cdot d -x,s\Bigr) \nonumber \\
&&+\eta \cdot \tilde{\gamma}_{m,j_k} \label{Eq:case2b2} \\
&\geq&\eta \cdot \Bigl[\tilde{\gamma}_{m,j_k}-\tilde{c}_k(s)\Bigr] \label{Eq:case2b3} \\
&\geq&0. \label{Eq:app:b2_main2}
\end{eqnarray}
Equation \eqref{Eq:case2b1} follows from Lemma \ref{Le:app:gol_2} with $r=m-q+1\leq m=l$ and $\eta =\eta$.
Equation \eqref{Eq:case2b2} follows from:
\begin{eqnarray*}
q\geq m-j_k+1 ~~\Leftrightarrow ~~ m-q+1 \leq j_k ~~ \Rightarrow ~~ \tilde{\gamma}_{m,j_k} \leq \tilde{\gamma}_{m,m-q+1}~.
\end{eqnarray*}
\noindent Equation \eqref{Eq:case2b3} follows from \eqref{Eq:case2b0}. Finally, \eqref{Eq:app:b2_main2} 
follows from $\tilde{c}_k(s) < \tilde{\gamma}_{m,j_k}$. Rearranging (\ref{Eq:app:b2_main2}) yields:
\begin{eqnarray} \label{Eq:app:b2_3}
V_m^{{\boldsymbol{\pi}}_{0}^{q}}(x,s) \leq V_m^{{\boldsymbol{\pi}}_{\eta}^{q}}(x,s)~.
\end{eqnarray}

Next, let $t \in \left\{m-j_k+1,m-j_k+2,\ldots,m-1 \right\}$ be arbitrary. We have:
\begin{eqnarray}
V_m^{{\boldsymbol{\pi}}_{0}^{t+1}}(x,s) - V_m^{{\boldsymbol{\pi}}_{0}^{t}}(x,s) &=& c\Bigl((m-t)\cdot d -x,s\Bigr) - c\Bigl((m-t+1)\cdot d -x,s\Bigr) \nonumber \\
&& - h \cdot d \nonumber \\
&& + \alpha \cdot \Expectation\Bigl[V_{m-1}\bigl((m-t-1)\cdot d,S\bigr)-V_{m-1}\bigl((m-t) \cdot d,S\bigr) \Bigr] \nonumber \\
&\geq& c\Bigl((m-t)\cdot d -x,s\Bigr) - c\Bigl((m-t+1)\cdot d -x,s\Bigr) \nonumber \\
&& + d \cdot \tilde{\gamma}_{m,m-t+1} \label{Eq:case2b4} \\
&\geq& c\Bigl((m-t)\cdot d -x,s\Bigr) - c\Bigl((m-t+1)\cdot d -x,s\Bigr) \nonumber \\
&& + d \cdot \tilde{\gamma}_{m,j_k} \label{Eq:case2b5} \\
&\geq& d \cdot \Bigl[\tilde{\gamma}_{m,j_k}-\tilde{c}_k(s)\Bigr] \label{Eq:case2b6} \\
&\geq&0. \label{Eq:app:b2_main}
\end{eqnarray}
Equation \eqref{Eq:case2b4} follows from Lemma \ref{Le:app:gol_2} with $r=m-t\leq m=l$ and $\eta = d$. 
Equation \eqref{Eq:case2b5}
follows from:
\begin{eqnarray*}
t\geq m-j_k+1 ~~\Leftrightarrow ~~ m-t+1 \leq j_k ~~ \Rightarrow ~~ \tilde{\gamma}_{m,j_k} \leq \tilde{\gamma}_{m-t+1}~.
\end{eqnarray*}
\noindent Equation \eqref{Eq:case2b6} follows from construction of $c(\cdot,s)$ and the fact that:
\begin{eqnarray*}
(m-t+1)\cdot d -x \leq \Bigl(m-(m-j_k+1)+1\Bigr) \cdot d - x = j_k \cdot d - x \leq \tilde{z}_k(s)~.
\end{eqnarray*}
\noindent Finally, \eqref{Eq:app:b2_main} follows from $\tilde{c}_k(s) <  \tilde{\gamma}_{m,j_k}$. Rearranging (\ref{Eq:app:b2_main}) yields:
\begin{eqnarray} \label{Eq:app:b2_2}
V_m^{{\boldsymbol{\pi}}_{0}^{t}}(x,s) \leq
V_m^{{\boldsymbol{\pi}}_{0}^{t+1}}(x,s)~~\forall t \in
\left\{m-j_k+1,m-j_k+2,\ldots,m-1 \right\}~.
\end{eqnarray}

Noting that
$\bar{\boldsymbol{\pi}} = {\boldsymbol{\pi}}_{0}^{m-j_k+1}$, (\ref{Eq:app:b2_3}) and repeated application of (\ref{Eq:app:b2_2}) imply:
\begin{eqnarray} \label{Eq:app:b2_7}
V_m^{\bar{\boldsymbol{\pi}}}(x,s) &=& V_m^{{\boldsymbol{\pi}}_{0}^{m-j_k+1}}(x,s) \nonumber \\
&\leq& V_m^{{\boldsymbol{\pi}}_{0}^{m-j_k+2}}(x,s) \leq \ldots \leq V_m^{{\boldsymbol{\pi}}_{0}^{q}}(x,s) \leq V_m^{{\boldsymbol{\pi}}_{\eta}^{q}}(x,s) = V_m^{{\boldsymbol{\pi}}^{\prime \prime}(\epsilon)}(x,s)~,
\end{eqnarray}
\noindent and we conclude $\bar{{\boldsymbol{\pi}}}$ is at least as good as
${\boldsymbol{\pi}}^{\prime \prime}(\epsilon)$.

\noindent \underline{Case (c)}: $j_{k} \cdot d - \tilde{z}_{k-1}(s) < x \leq j_{k-1} \cdot d - \tilde{z}_{k-1}(s),~k \in \{1,\ldots,K\}$ \\
%
In this case, $\bar{\boldsymbol{\pi}} = {\boldsymbol{\pi}}_{0}^{m+1-\frac{x+\tilde{z}_{k-1(s)}}{d}}$. Define $p:=m+1-\frac{x+\tilde{z}_{k-1(s)}}{d}$, and $q:=p+\left\lfloor \frac{\epsilon}{d} \right\rfloor$.

Again, we start by comparing ${\boldsymbol{\pi}}_{\eta}^{q}$ and ${\boldsymbol{\pi}}_{0}^{q}$:
\begin{eqnarray}
V_m^{{\boldsymbol{\pi}}_{\eta}^{q}}(x,s) -  V_m^{{\boldsymbol{\pi}}_{0}^{q}}(x,s)
&=& c\Bigl(z_m^{\prime \prime}- \eta,s\Bigr) - c\Bigl(z_m^{\prime \prime},s\Bigr) \nonumber \\
&& -h \cdot \eta \nonumber \\
&& + \alpha \cdot \Expectation\Bigl[V_{m-1}\bigl((m-q)\cdot d - \eta,S\bigr)-V_{m-1}\bigl((m-q) \cdot d,S\bigr) \Bigr] \nonumber \\
&\geq& c\Bigl(z_m^{\prime \prime}- \eta,s\Bigr) - c\Bigl(z_m^{\prime \prime},s\Bigr) \nonumber \\
&&+\eta \cdot \tilde{\gamma}_{m,m-q+1} \label{Eq:case2b1f} \\
&\geq& c\Bigl(z_m^{\prime \prime}- \eta,s\Bigr) - c\Bigl(z_m^{\prime \prime},s\Bigr) \nonumber \\
&&+\eta \cdot \tilde{\gamma}_{m,j_{k-1}} \label{Eq:case2b2f} \\
&\geq&\eta \cdot \Bigl[\tilde{\gamma}_{m,j_{k-1}}-\tilde{c}_{k-1}(s)\Bigr] \label{Eq:case2b3f} \\
&\geq&0. \label{Eq:app:b2_main2f}
\end{eqnarray}
Equation \eqref{Eq:case2b1f} follows from Lemma \ref{Le:app:gol_2} with $r=m-q+1\leq m=l$ and $\eta =\eta$.
Equation \eqref{Eq:case2b2f} follows from:
\begin{eqnarray*}
m-q+1 = m-\left(p+\left\lfloor \frac{\epsilon}{d} \right\rfloor\right)+1 = \frac{x+\tilde{z}_{k-1(s)}}{d}-\left\lfloor \frac{\epsilon}{d} \right\rfloor  \leq \frac{x+\tilde{z}_{k-1(s)}}{d} \leq j_{k-1}~, 
\end{eqnarray*}
which implies $\tilde{\gamma}_{m,j_{k-1}} \leq \tilde{\gamma}_{m,m-q+1}$.
Equation \eqref{Eq:case2b3f} follows from $z_m^{\prime \prime} < \tilde{z}_{k-1}(s)$ and the construction of $c(\cdot,s)$. Finally, \eqref{Eq:app:b2_main2f} 
follows from $\tilde{c}_{k-1}(s) < \tilde{\gamma}_{m,j_{k-1}}$. Rearranging (\ref{Eq:app:b2_main2f}) yields:
\begin{eqnarray} \label{Eq:app:b2_3f}
V_m^{{\boldsymbol{\pi}}_{0}^{q}}(x,s) \leq V_m^{{\boldsymbol{\pi}}_{\eta}^{q}}(x,s)~.
\end{eqnarray}

Next, let $\hat{t} \in \left\{p,p+1,\ldots,q-1\right\}$ be arbitrary. We have:
\begin{eqnarray}
V_m^{{\boldsymbol{\pi}}_{0}^{\hat{t}+1}}(x,s) - V_m^{{\boldsymbol{\pi}}_{0}^{\hat{t}}}(x,s) &=& c\Bigl((m-\hat{t})\cdot d -x,s\Bigr) - c\Bigl((m-\hat{t}+1)\cdot d -x,s\Bigr) \nonumber \\
&& - h \cdot d \nonumber \\
&& + \alpha \cdot \Expectation\Bigl[V_{m-1}\bigl((m-\hat{t}-1)\cdot d,S\bigr)-V_{m-1}\bigl((m-\hat{t}) \cdot d,S\bigr) \Bigr] \nonumber \\
&\geq& c\Bigl((m-\hat{t})\cdot d -x,s\Bigr) - c\Bigl((m-\hat{t}+1)\cdot d -x,s\Bigr) \nonumber \\
&& + d \cdot \tilde{\gamma}_{m,m-\hat{t}+1} \label{Eq:case2e4} \\
&\geq& c\Bigl((m-\hat{t})\cdot d -x,s\Bigr) - c\Bigl((m-\hat{t}+1)\cdot d -x,s\Bigr) \nonumber \\
&& + d \cdot \tilde{\gamma}_{m,j_{k-1}} \label{Eq:case2e5} \\
&\geq& d \cdot \Bigl[\tilde{\gamma}_{m,j_{k-1}}-\tilde{c}_{k-1}(s)\Bigr] \label{Eq:case2e6} \\
&\geq&0. \label{Eq:app:b2_maine}
\end{eqnarray}
Equation \eqref{Eq:case2e4} follows from Lemma \ref{Le:app:gol_2} with $r=m-\hat{t}\leq m=l$ and $\eta = d$. 
Equation \eqref{Eq:case2e5}
follows from:
\begin{eqnarray*}
\hat{t}\geq p ~~\Rightarrow ~~ m-\hat{t}+1 \leq m-p+1 = \frac{x+\tilde{z}_{k-1}(s)}{d} \leq j_{k-1} ~~ \Rightarrow ~~ \tilde{\gamma}_{m,j_{k-1}} \leq \tilde{\gamma}_{m-\hat{t}+1}~.
\end{eqnarray*}
\noindent Equation \eqref{Eq:case2e6} follows from construction of $c(\cdot,s)$ and the fact that:
\begin{eqnarray*}
(m-\hat{t}+1)\cdot d -x \leq (m-p+1)\cdot d -x =  \tilde{z}_{k-1}(s)~.
\end{eqnarray*}
\noindent Finally, \eqref{Eq:app:b2_maine} follows from $\tilde{c}_{k-1}(s) <  \tilde{\gamma}_{m,j_{k-1}}$.
Rearranging \eqref{Eq:app:b2_maine} yields:
\begin{eqnarray} \label{Eq:app:b2_8}
V_m^{{\boldsymbol{\pi}}_{0}^{\hat{t}}}(x,s) \leq
V_m^{{\boldsymbol{\pi}}_{0}^{\hat{t}+1}}(x,s)~~\forall \hat{t} \in
\left\{p,p+1,\ldots,q-1 \right\}~.
\end{eqnarray}
Then \eqref{Eq:app:b2_3f}
and repeated application of \eqref{Eq:app:b2_8} yield:
\begin{eqnarray*}
V_m^{\bar{\boldsymbol{\pi}}}(x,s) &=& V_m^{{\boldsymbol{\pi}}_{0}^{p}}(x,s) \\
&\leq& V_m^{{\boldsymbol{\pi}}_{0}^{p+1}}(x,s) \leq \ldots \leq
V_m^{{\boldsymbol{\pi}}_{0}^{q-1}}(x,s) \leq
V_m^{{\boldsymbol{\pi}}_{0}^{q}}(x,s) \leq V_m^{{\boldsymbol{\pi}}_{\eta}^{q}}(x,s) = V_m^{{\boldsymbol{\pi}}^{\prime \prime}(\epsilon)}(x,s)~,
\end{eqnarray*}
\noindent and we conclude $\bar{{\boldsymbol{\pi}}}$ is at least as good as
${\boldsymbol{\pi}}^{\prime \prime}(\epsilon)$.

\noindent \underline{Case (d)}: $0 \leq x \leq j_K \cdot d - \tilde{z}_{\max}$ \\
The same argument as Case (c) applies with $k$ replaced by $K+1$ and $\tilde{z}_K(s)=\tilde{z}_{\max}(s)$.
This completes Part II.

From Parts I and II, we conclude $\bar{{\boldsymbol{\pi}}}$ is optimal if the starting queue level is an integer multiple of the demand $d$. By assumption, the starting queue level $x$ at time $N$ is zero. Thus, $\bar{{\boldsymbol{\pi}}}$ is optimal at time $N$. $z_N^*(x,s)=\bar{z}_N(x,s)$ will also be an integer multiple of demand as  $b_{N,k}(s)$, and $\left\{\tilde{z}_k(s)\right\}_{k=0,1,\ldots,K}$
are all integer multiples of $d$. 
It follows that the queue level at the end of slot $N$ (equal to the queue level at the beginning of slot $N-1$), $z_N^*(x,s)-d$, will also be an integer multiple of $d$. Continuing this logic, if the strategy $\bar{{\boldsymbol{\pi}}}$ is used, the queue level at the beginning of each subsequent time slot will be an integer multiple of demand. Thus, 
$\bar{{\boldsymbol{\pi}}}$ is optimal. \qedsymbol 

%% file: appendix_two_user_properties_proof.tex
We prove statements (i)-(v) by joint induction on the time remaining, $n$.\\
\noindent \underline{Base Case}: $n=1$ \\
$V_0(\textbf{x},\textbf{s}_0)=0$, for all $\textbf{s}_0$, so (i) and (ii) hold trivially.
Let ${\textbf{s}}_1 \in {\cal{S}}$ be arbitrary.
$G_1(\textbf{y}_1,\textbf{s}_1) =
\textbf{c}_{\textbf{s}_1}^{\transpose} 
\textbf{y}_1
 +h(\textbf{y}_1-\textbf{d})$,
 which is convex and supermodular.
 Thus, (iii) and (iv) are
 true. Additionally, $G_1(\textbf{y}_1,\textbf{s}_1) =
\sum\limits_{m=1}^2 \left\{c_s^m \cdot y_1^m + h^m\left(y_1^m-d^m\right)\right\}$, so  $\inf \left\{\argmin\limits_{y_1^2 \in [d^2,\infty)} 
\biggl\{G_1\left(y_1^1,y_1^2,s_1^1,s_1^2\right)\biggr\}\right\}$ is independent of $y_1^1$, and vice versa. Thus, (v) is true for $n=1$, completing the base case.

\medskip

\noindent \underline{Induction Step}\\
Assume statements (i)-(v) are true for $n=2,3,\ldots,l-1$. We
want to show they are true for $n=l$. We let $\textbf{s} \in
\cal{S}$ be arbitrary, and proceed in order.
\begin{itemize}
 \item[(i)] Consider two arbitrary points,
$\bar{\textbf{x}}, \tilde{\textbf{x}} \in \Real_+^2$.
Let $\lambda \in [0,1]$ be
arbitrary, and define $\hat{\textbf{x}} := \lambda 
\bar{\textbf{x}} + (1-\lambda) 
\tilde{\textbf{x}}$. Let
${\textbf{y}}^*(\bar{\textbf{x}},\textbf{s})$,
${\textbf{y}}^*(\tilde{\textbf{x}},\textbf{s})$, and
${\textbf{y}}^*(\hat{\textbf{x}},\textbf{s})$ be optimal
buffer levels after transmission in slot $l-1$, for each of the respective
starting points. We have:
\begin{eqnarray}  \label{Eq:mp:llu_conv}
\lambda \cdot V_{l-1}(\bar{\textbf{x}},\textbf{s}) + (1-\lambda) \cdot
V_{l-1}(\tilde{\textbf{x}},\textbf{s}) &=&
-\textbf{c}_{\textbf{s}}^{\transpose} \hat{\textbf{x}}
+ \lambda \cdot
G_{l-1}\bigl(\textbf{y}^*(\bar{\textbf{x}},\textbf{s}),\textbf{s}\bigr) \nonumber \\
&& +~ (1-\lambda) \cdot
G_{l-1}\bigl(\textbf{y}^*(\tilde{\textbf{x}},\textbf{s}),\textbf{s}\bigr) \nonumber \\
& \geq & -\textbf{c}_{\textbf{s}}^{\transpose}
\hat{\textbf{x}} 
+G_{l-1}\bigl(\lambda 
\textbf{y}^*(\bar{\textbf{x}},\textbf{s})+(1-\lambda) 
\textbf{y}^*(\tilde{\textbf{x}},\textbf{s}) ,\textbf{s}\bigr) \nonumber \\
& \geq & -\textbf{c}_{\textbf{s}}^{\transpose}
\hat{\textbf{x}} 
+\min_{\textbf{y} \in \tilde{\cal{A}}^{\textbf{d}}
(\hat{\textbf{x}},\textbf{s})} \left\{
G_{l-1}(\textbf{y},\textbf{s})
\right\} \nonumber \\
&=& V_{l-1}(\hat{\textbf{x}},\textbf{s}) = V_{l-1}(\lambda 
\bar{\textbf{x}} + (1-\lambda) 
\tilde{\textbf{x}},\textbf{s})~,
\end{eqnarray}
where the first inequality follows from the convexity of
$G_{l-1}(\cdot,\textbf{s})$ 
from the induction hypothesis. The second inequality
follows from the following argument.
%
$\textbf{y}^*(\bar{\textbf{x}},\textbf{s}) \in \tilde{\cal{A}}^{\textbf{d}} (\bar{\textbf{x}},\textbf{s})$ implies:
\begin{eqnarray} \label{Eq:mp:bar_hyp}
\textbf{y}^*(\bar{\textbf{x}},\textbf{s})
\succeq \textbf{d} \vee \bar{\textbf{x}}
~\hbox{ and }~
 \textbf{c}_{\textbf{s}}^{\transpose}
\left[\textbf{y}^*(\bar{\textbf{x}},\textbf{s}) -\bar{\textbf{x}}
\right]\leq P~.
\end{eqnarray}
Similarly, $\textbf{y}^*(\tilde{\textbf{x}},\textbf{s}) \in \tilde{\cal{A}}^{\textbf{d}} (\tilde{\textbf{x}},\textbf{s})$ implies:
\begin{eqnarray} \label{Eq:mp:tilde_hyp}
\textbf{y}^*(\tilde{\textbf{x}},\textbf{s})
\succeq \textbf{d} \vee \tilde{\textbf{x}}
~ \hbox{ and }~
 \textbf{c}_{\textbf{s}}^{\transpose}
\left[\textbf{y}^*(\tilde{\textbf{x}},\textbf{s}) -\tilde{\textbf{x}}
\right]\leq P~.
\end{eqnarray}
Multiplying the equations in (\ref{Eq:mp:bar_hyp}) by $\lambda$ and
the equations in (\ref{Eq:mp:tilde_hyp}) by $1-\lambda$, and summing, we have:
\begin{eqnarray} \label{Eq:mp:hat_sides}
\lambda 
{\textbf{y}}^*(\bar{\textbf{x}},\textbf{s}) +
(1-\lambda) 
{\textbf{y}}^*(\tilde{\textbf{x}},\textbf{s}) 
\succeq  \lambda (\textbf{d} \vee \bar{\textbf{x}}) + 
(1-\lambda) (\textbf{d} \vee \tilde{\textbf{x}}) 
\succeq \textbf{d} \vee \hat{\textbf{x}}, 
\end{eqnarray}
and
\begin{align}\label{Eq:mp:hat_hyp}
&\textbf{c}_{\textbf{s}}^{\transpose} \left[\lambda
\textbf{y}^*(\bar{\textbf{x}},\textbf{s})+(1-\lambda) 
\textbf{y}^*(\tilde{\textbf{x}},\textbf{s})
-\hat{\textbf{x}} \right] \nonumber \\
&=\lambda \textbf{c}_{\textbf{s}}^{\transpose} \left[\textbf{y}^*(\bar{\textbf{x}},\textbf{s})-\bar{\textbf{x}}\right]
+(1-\lambda) \textbf{c}_{\textbf{s}}^{\transpose} \left[\textbf{y}^*(\tilde{\textbf{x}},\textbf{s})-\tilde{\textbf{x}}\right]
\leq P~.
\end{align}
From (\ref{Eq:mp:hat_sides}) and (\ref{Eq:mp:hat_hyp}), we
conclude $\lambda 
\textbf{y}^*(\bar{\textbf{x}},\textbf{s})
+ (1-\lambda) 
\textbf{y}^*(\tilde{\textbf{x}},\textbf{s})
\in \tilde{\cal{A}}^{\textbf{d}}(\hat{\textbf{x}},\textbf{s})$, as shown in
Figure \ref{Fig:multi_proof:convexity}. Thus, the value of
$G_{l-1}(\cdot,\textbf{s})$ at this point is greater than or equal to
the minimum of $G_l(\cdot,\textbf{s})$ over the region $\tilde{\cal{A}}^{\textbf{d}}(\hat{\textbf{x}},\textbf{s})$. From
(\ref{Eq:mp:llu_conv}), we conclude $V_{l-1}(\cdot, \textbf{s})$ is convex. This is a similar argument to the one used by Evans to show convexity in \cite{evans}.
\begin{figure}[htbp]
\centering \includegraphics[width=4.5in]{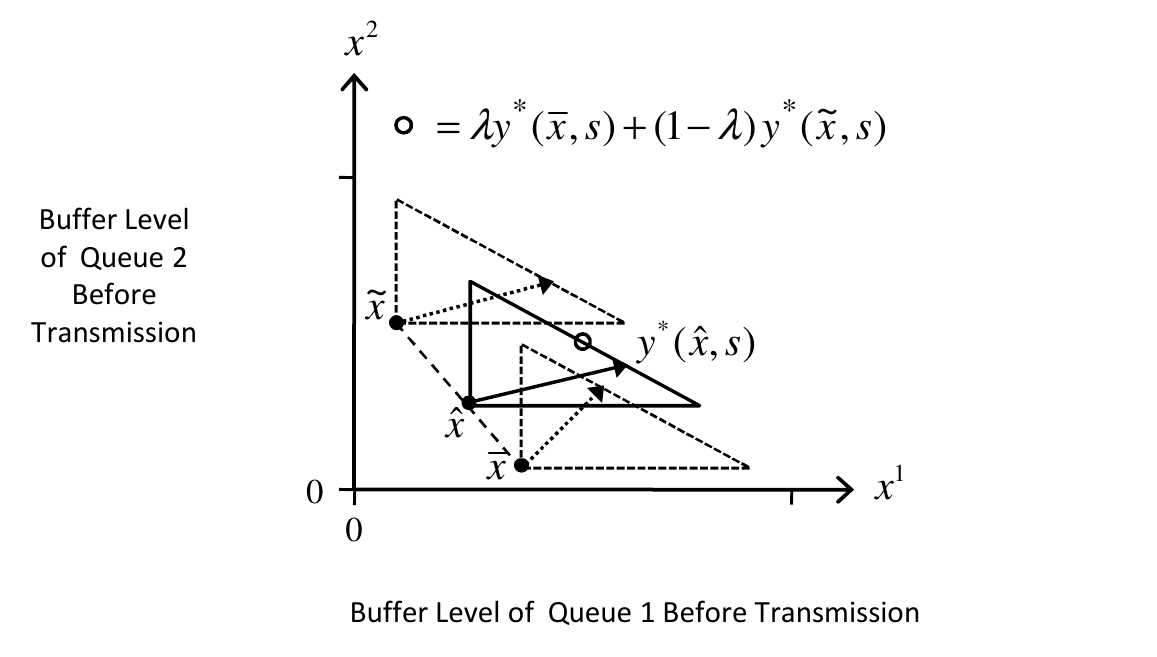} 
\caption{Diagram showing $\lambda 
\textbf{y}^*(\bar{\textbf{x}},\textbf{s})
+ (1-\lambda) 
\textbf{y}^*(\tilde{\textbf{x}},\textbf{s})
\in \tilde{\cal{A}}^{\textbf{d}}(\hat{\textbf{x}},\textbf{s})$ in the proof of the convexity of $V_{l-1}(\cdot,\textbf{s})$. }\label{Fig:multi_proof:convexity}
\end{figure}

\item[(ii)]
Recall that $V_{l-1}(\textbf{x},\textbf{s}) =
-\textbf{c}_{\textbf{s}}^{\transpose} \textbf{x}
+ \min_{\textbf{y} \in
\tilde{\cal{A}}^{\textbf{d}}(\textbf{x},\textbf{s})} \left\{G_{l-1}(\textbf{y},\textbf{s})\right\}$. The first term, $-\textbf{c}_{\textbf{s}}^{\transpose} \textbf{x}$, is clearly supermodular in $\textbf{x}$, so it suffices to show that the second term, $\min_{\textbf{y} \in
\tilde{\cal{A}}^{\textbf{d}}(\textbf{x},\textbf{s})} \left\{G_{l-1}(\textbf{y},\textbf{s})\right\}$, is also supermodular in $\textbf{x}$. Let $\bar{\textbf{x}}, \tilde{\textbf{x}} \in \Real^2$ be arbitrary. We want to show:
\begin{align}\label{Eq:mp:superwts}
&\min_{\textbf{y} \in
\tilde{\cal{A}}^{\textbf{d}}(\bar{\textbf{x}},\textbf{s})} \left\{G_{l-1}(\textbf{y},\textbf{s})\right\}
+\min_{\textbf{y} \in
\tilde{\cal{A}}^{\textbf{d}}(\tilde{\textbf{x}},\textbf{s})} \left\{G_{l-1}(\textbf{y},\textbf{s})\right\} \nonumber \\
&\leq~
\min_{\textbf{y} \in
\tilde{\cal{A}}^{\textbf{d}}(\bar{\textbf{x}} \wedge \tilde{\textbf{x}},\textbf{s})} \left\{G_{l-1}(\textbf{y},\textbf{s})\right\}
+\min_{\textbf{y} \in
\tilde{\cal{A}}^{\textbf{d}}(\bar{\textbf{x}} \vee \tilde{\textbf{x}},\textbf{s})} \left\{G_{l-1}(\textbf{y},\textbf{s})\right\}.
\end{align}
\noindent If $\bar{\textbf{x}}$ and $\tilde{\textbf{x}}$ are
comparable (i.e., $\tilde{x}^1 \geq \bar{x}^1$ and $\tilde{x}^2
\geq \bar{x}^2$ or $\tilde{x}^1 \leq \bar{x}^1$ and $\tilde{x}^2
\leq \bar{x}^2$), then (\ref{Eq:mp:superwts}) is
trivial. So we assume they are not comparable, and also assume
without loss of generality that $\bar{x}^1<\tilde{x}^1$ and
$\tilde{x}^2<\bar{x}^2$. We begin with a quick lemma.
\begin{lemma}
There exist optimal buffer levels after transmission in slot $l-1$, $\textbf{y}^*(\bar{\textbf{x}} \wedge \tilde{\textbf{x}},\textbf{s})$ and $\textbf{y}^*(\bar{\textbf{x}} \vee \tilde{\textbf{x}},\textbf{s})$, such that $\textbf{y}^*(\bar{\textbf{x}} \wedge \tilde{\textbf{x}},\textbf{s}) \nsucc \textbf{y}^*(\bar{\textbf{x}} \vee \tilde{\textbf{x}},\textbf{s})$; i.e., such that ${y}^{*^1}(\bar{\textbf{x}} \wedge \tilde{\textbf{x}},\textbf{s}) \leq {y}^{*^1}(\bar{\textbf{x}} \vee \tilde{\textbf{x}},\textbf{s})$ or ${y}^{*^2}(\bar{\textbf{x}} \wedge \tilde{\textbf{x}},\textbf{s}) \leq {y}^{*^2}(\bar{\textbf{x}} \vee \tilde{\textbf{x}},\textbf{s})$.
\end{lemma}
\begin{proof}
Fix a choice of $\textbf{y}^*(\bar{\textbf{x}} \vee \tilde{\textbf{x}},\textbf{s})$ such that $G_{l-1}\bigl(\textbf{y}^*\left(\bar{\textbf{x}} \vee \tilde{\textbf{x}},\textbf{s}\right),\textbf{s}\bigr) = \min\limits_{\textbf{y} \in
\tilde{\cal{A}}^{\textbf{d}}(\bar{\textbf{x}} \vee \tilde{\textbf{x}},\textbf{s})} \left\{G_{l-1}(\textbf{y},\textbf{s})\right\}$. Assume that for all optimal choices of $\textbf{y}^*(\bar{\textbf{x}} \wedge \tilde{\textbf{x}},\textbf{s})$, we have
$\textbf{y}^*(\bar{\textbf{x}} \wedge \tilde{\textbf{x}},\textbf{s}) \succ \textbf{y}^*(\bar{\textbf{x}} \vee \tilde{\textbf{x}},\textbf{s})$. Fix one such choice of $\textbf{y}^*(\bar{\textbf{x}} \wedge \tilde{\textbf{x}},\textbf{s})$, and we have:
\begin{eqnarray}\label{Eq:mp:l1_1}
\textbf{y}^*(\bar{\textbf{x}} \wedge \tilde{\textbf{x}},\textbf{s}) \succ \textbf{y}^*(\bar{\textbf{x}} \vee \tilde{\textbf{x}},\textbf{s}) \succeq \textbf{d} \vee (\bar{\textbf{x}} \vee \tilde{\textbf{x}})~. 
\end{eqnarray}
Further, $\textbf{y}^*(\bar{\textbf{x}} \wedge \tilde{\textbf{x}},\textbf{s}) \in \tilde{\cal{A}}^{\textbf{d}}(\bar{\textbf{x}} \wedge \tilde{\textbf{x}},\textbf{s})$ implies $\textbf{c}_{\textbf{s}}^{\transpose} 
\left[\textbf{y}^*(\bar{\textbf{x}} \wedge \tilde{\textbf{x}},\textbf{s})-\bar{\textbf{x}} \wedge \tilde{\textbf{x}}\right] \leq P$, and thus:
\begin{eqnarray}\label{Eq:mp:l1_2}
\textbf{c}_{\textbf{s}}^{\transpose} 
\left[\textbf{y}^*(\bar{\textbf{x}} \wedge \tilde{\textbf{x}},\textbf{s})-\bar{\textbf{x}} \vee \tilde{\textbf{x}} \right] \leq \textbf{c}_{\textbf{s}}^{\transpose} 
\left[\textbf{y}^*(\bar{\textbf{x}} \wedge \tilde{\textbf{x}},\textbf{s})-\bar{\textbf{x}} \wedge \tilde{\textbf{x}}\right] \leq P~.
\end{eqnarray}
Equations (\ref{Eq:mp:l1_1}) and (\ref{Eq:mp:l1_2}) imply $\textbf{y}^*(\bar{\textbf{x}} \wedge \tilde{\textbf{x}},\textbf{s}) \in \tilde{\cal{A}}^{\textbf{d}}(\bar{\textbf{x}} \vee \tilde{\textbf{x}},\textbf{s})$, and thus:
\begin{eqnarray}\label{Eq:mp:l1_3}
G_{l-1}\bigl(\textbf{y}^*\left(\bar{\textbf{x}} \vee \tilde{\textbf{x}},\textbf{s}\right),\textbf{s}\bigr) = \min_{\textbf{y} \in
\tilde{\cal{A}}^{\textbf{d}}(\bar{\textbf{x}} \vee \tilde{\textbf{x}},\textbf{s})} \left\{G_{l-1}(\textbf{y},\textbf{s})\right\} \leq G_{l-1}\bigr(\textbf{y}^*\left(\bar{\textbf{x}} \wedge \tilde{\textbf{x}},\textbf{s}\right),\textbf{s}\bigr)~.
\end{eqnarray}
However, we also have:
\begin{eqnarray}\label{Eq:mp:l1_4}
\textbf{y}^*(\bar{\textbf{x}} \vee \tilde{\textbf{x}},\textbf{s}) \succeq
\textbf{d} \vee (\bar{\textbf{x}} \vee \tilde{\textbf{x}})
\succeq \textbf{d} \vee (\bar{\textbf{x}} \wedge \tilde{\textbf{x}})~,
\end{eqnarray}
and
\begin{eqnarray}\label{Eq:mp:l1_5}
\textbf{c}_{\textbf{s}}^{\transpose} 
\left[\textbf{y}^*(\bar{\textbf{x}} \vee \tilde{\textbf{x}},\textbf{s})-\bar{\textbf{x}} \wedge \tilde{\textbf{x}}\right] \leq \textbf{c}_{\textbf{s}}^{\transpose} 
\left[\textbf{y}^*(\bar{\textbf{x}} \wedge \tilde{\textbf{x}},\textbf{s})-\bar{\textbf{x}} \wedge \tilde{\textbf{x}} \right] \leq P~.
\end{eqnarray}
Equations (\ref{Eq:mp:l1_4}) and (\ref{Eq:mp:l1_5}) imply $\textbf{y}^*(\bar{\textbf{x}} \vee \tilde{\textbf{x}},\textbf{s}) \in \tilde{\cal{A}}^{\textbf{d}}(\bar{\textbf{x}} \wedge \tilde{\textbf{x}},\textbf{s})$, which, in combination with (\ref{Eq:mp:l1_3}), implies it is optimal to move from $\bar{\textbf{x}} \wedge \tilde{\textbf{x}}$ to $\textbf{y}^*(\bar{\textbf{x}} \vee \tilde{\textbf{x}},\textbf{s})$, contradicting the assumption that $\textbf{y}^*(\bar{\textbf{x}} \wedge \tilde{\textbf{x}},\textbf{s}) \succ \textbf{y}^*(\bar{\textbf{x}} \vee \tilde{\textbf{x}},\textbf{s})$ for all possible choices of $\textbf{y}^*(\bar{\textbf{x}} \wedge \tilde{\textbf{x}},\textbf{s})$.
\end{proof}
\medskip

Now let $\textbf{y}^*(\bar{\textbf{x}} \wedge \tilde{\textbf{x}},\textbf{s})$ and $\textbf{y}^*(\bar{\textbf{x}} \vee \tilde{\textbf{x}},\textbf{s})$ be arbitrary optimal actions such that $\textbf{y}^*(\bar{\textbf{x}} \wedge \tilde{\textbf{x}},\textbf{s}) \nsucc \textbf{y}^*(\bar{\textbf{x}} \vee \tilde{\textbf{x}},\textbf{s})$. We show (\ref{Eq:mp:superwts}) by considering two
exhaustive cases.
\medskip

\noindent \underline{Case 1}: $\textbf{y}^*(\bar{\textbf{x}} \vee \tilde{\textbf{x}},\textbf{s}) \succeq \textbf{y}^*(\bar{\textbf{x}} \wedge \tilde{\textbf{x}},\textbf{s})$ \\
We start with another lemma.
\begin{lemma} \label{Le:mp:parallelogram}
Let $f:[d^1,\infty) \times [d^2,\infty) 
\rightarrow \Real$ be convex and supermodular, let $\sigma, \beta \in [0,1]$ be arbitrary, and let $\textbf{z} = (z_1,z_2) \preceq (\hat{z}_1,\hat{z}_2) = \hat{\textbf{z}}$. Define $\textbf{z}^{\lambda_1,\lambda_2}:=\Bigl(\lambda_1 \hat{z}_1 + (1-\lambda_1) z_1, \lambda_2 \hat{z}_2 + (1-\lambda_2) z_2\Bigr)$. Then
\begin{eqnarray}\label{Eq:sup_lem}
f(\textbf{z}) + f(\hat{\textbf{z}}) \geq  f(\textbf{z}^{\sigma,\beta}) + f(\textbf{z}^{1-\sigma,1-\beta})~.
\end{eqnarray}
\end{lemma}
\begin{proof}
~ \\
\noindent \underline{Step 1}: Assume $\sigma,\beta \leq \frac{1}{2}$. Assume without loss of generality that $\sigma \leq \beta$.
By the convexity of $f(\cdot)$, we have:
\begin{eqnarray} \label{Eq:mp:par_1}
f(\textbf{z}) + f(\hat{\textbf{z}}) \geq  f(\textbf{z}^{\sigma,\sigma}) + f(\textbf{z}^{1-\sigma,1-\sigma})~,
\end{eqnarray}
and
\begin{eqnarray} \label{Eq:mp:par_2}
f(\textbf{z}^{1-\sigma,1-\sigma}) + f({\textbf{z}}^{1-\sigma,\sigma}) \geq  f(\textbf{z}^{1-\sigma,\beta}) + f(\textbf{z}^{1-\sigma,1-\beta})~.
\end{eqnarray}
By the supermodularity of $f(\cdot)$, we have:
\begin{eqnarray} \label{Eq:mp:par_3}
f(\textbf{z}^{1-\sigma,\beta}) + f({\textbf{z}}^{\sigma,\sigma}) \geq  f(\textbf{z}^{\sigma,\beta}) + f(\textbf{z}^{1-\sigma,\sigma})~.
\end{eqnarray}
Figure \ref{Fig:multi_proof:parallelogram} shows these relationships. Combining (\ref{Eq:mp:par_1})-(\ref{Eq:mp:par_3}), we have:
\begin{eqnarray*}
f(\textbf{z}) + f(\hat{\textbf{z}}) &\geq& f(\textbf{z}^{\sigma,\sigma}) + f(\textbf{z}^{1-\sigma,1-\sigma}) \\
& \geq & f(\textbf{z}^{\sigma,\sigma}) + f(\textbf{z}^{1-\sigma,\beta}) - f(\textbf{z}^{1-\sigma,\sigma}) + f(\textbf{z}^{1-\sigma,1-\beta}) \\
& \geq & f(\textbf{z}^{\sigma,\beta}) + f(\textbf{z}^{1-\sigma,1-\beta})~.
\end{eqnarray*}
\begin{figure}[htbp]
\centering \includegraphics[width=4.5in]{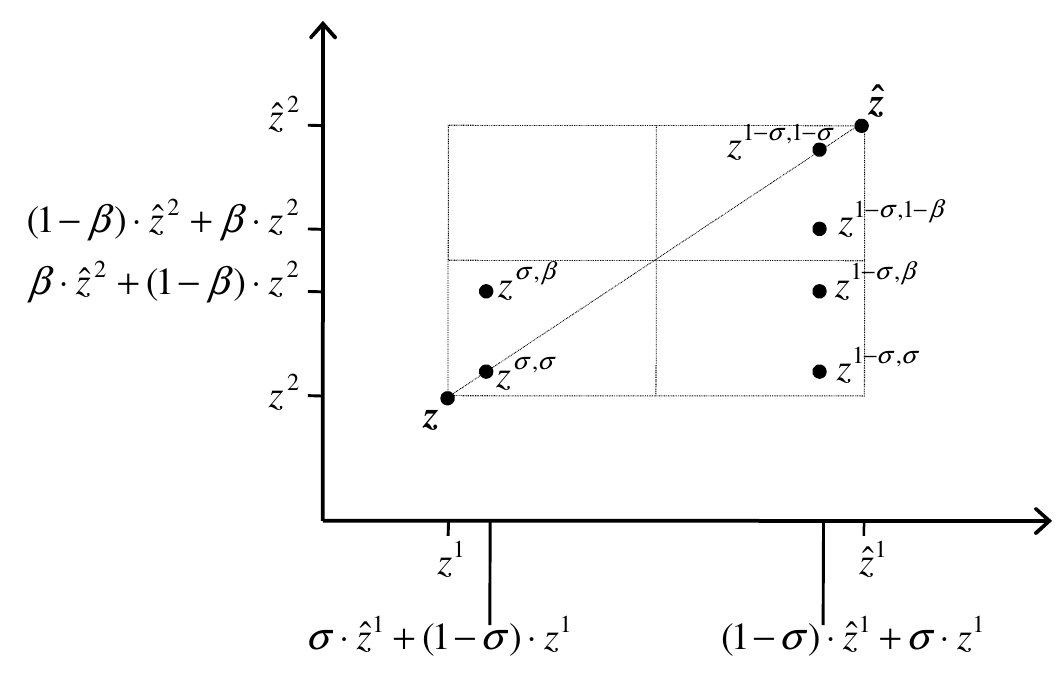} 
\caption{Diagram of the points referred to in Step 1 of the proof of Lemma \ref{Le:mp:parallelogram}}\label{Fig:multi_proof:parallelogram}
\end{figure}

\noindent \underline{Step 2}: Now let $\sigma,\beta \in [0,1]$, and define $\hat{\sigma}:=\min\left\{\sigma,1-\sigma\right\}$ and $\hat{\beta}:=\min\left\{\beta,1-\beta\right\}$. Then $\hat{\sigma},\hat{\beta} \leq \frac{1}{2}$, so by Step 1, we have:
\begin{eqnarray} \label{Eq:mp:par_4}
f(\textbf{z}) + f(\hat{\textbf{z}}) \geq  f(\textbf{z}^{\hat{\sigma},\hat{\beta}}) + f(\textbf{z}^{1-\hat{\sigma},1-\hat{\beta}})~.
\end{eqnarray}
Note that $\textbf{z}^{\sigma,\beta} \wedge \textbf{z}^{1-\sigma,1-\beta} = \textbf{z}^{\hat{\sigma},\hat{\beta}}$, and $\textbf{z}^{\sigma,\beta} \vee \textbf{z}^{1-\sigma,1-\beta} = \textbf{z}^{1-\hat{\sigma},1-\hat{\beta}}$, so by the supermodularity of $f(\cdot)$, we have:
\begin{eqnarray} \label{Eq:mp:par_5}
f(\textbf{z}^{\hat{\sigma},\hat{\beta}}) + f(\textbf{z}^{1-\hat{\sigma},1-\hat{\beta}}) \geq f(\textbf{z}^{\sigma,\beta}) + f(\textbf{z}^{1-\sigma,1-\beta})~.
\end{eqnarray}
Combining (\ref{Eq:mp:par_4}) and (\ref{Eq:mp:par_5}) yields the desire result, (\ref{Eq:sup_lem}).
\end{proof}

Next, define the following points, shown in Figure \ref{Fig:multi_proof:case1}:
\begin{eqnarray*}
\begin{array}{l}
\bar{\textbf{y}} := \left(
\begin{array}{l}
\bar{x}^1+\max\left\{y^{*^1}\left(\bar{\textbf{x}} \wedge \tilde{\textbf{x}},\textbf{s} \right)-\bar{x}^1, y^{*^1}\left(\bar{\textbf{x}} \vee \tilde{\textbf{x}},\textbf{s} \right)-\tilde{x}^1 \right\}, \\
\bar{x}^2+\min\left\{y^{*^2}\left(\bar{\textbf{x}} \wedge \tilde{\textbf{x}},\textbf{s} \right)-\tilde{x}^2, y^{*^2}\left(\bar{\textbf{x}} \vee \tilde{\textbf{x}},\textbf{s} \right)-\bar{x}^2 \right\}
\end{array}
\right),
\hbox{ and } \\
\\
\tilde{\textbf{y}} := \left(
\begin{array}{l}
\tilde{x}^1+\min\left\{y^{*^1}\left(\bar{\textbf{x}} \wedge \tilde{\textbf{x}},\textbf{s} \right)-\bar{x}^1, y^{*^1}\left(\bar{\textbf{x}} \vee \tilde{\textbf{x}},\textbf{s} \right)-\tilde{x}^1 \right\}, \\
\tilde{x}^2+\max\left\{y^{*^2}\left(\bar{\textbf{x}} \wedge \tilde{\textbf{x}},\textbf{s} \right)-\tilde{x}^2, y^{*^2}\left(\bar{\textbf{x}} \vee \tilde{\textbf{x}},\textbf{s} \right)-\bar{x}^2 \right\}
\end{array}
\right).
\end{array}
\end{eqnarray*}
Note that $\bar{\textbf{y}} \succeq \textbf{d} \vee \bar{\textbf{x}}$
and $\tilde{\textbf{y}} \succeq 
\textbf{d} \vee \tilde{\textbf{x}}$.
Furthermore, we have:
\begin{eqnarray*}
\textbf{c}_{\textbf{s}}^{\transpose} 
\left(\bar{\textbf{y}}-\bar{\textbf{x}}\right) &=& \textbf{c}_{\textbf{s}}^{\transpose} 
\left(
\begin{array}{l}
\max\left\{y^{*^1}\left(\bar{\textbf{x}} \wedge \tilde{\textbf{x}},\textbf{s} \right)-\bar{x}^1, y^{*^1}\left(\bar{\textbf{x}} \vee \tilde{\textbf{x}},\textbf{s} \right)-\tilde{x}^1 \right\}, \\
\min\left\{y^{*^2}\left(\bar{\textbf{x}} \wedge \tilde{\textbf{x}},\textbf{s} \right)-\tilde{x}^2, y^{*^2}\left(\bar{\textbf{x}} \vee \tilde{\textbf{x}},\textbf{s} \right)-\bar{x}^2 \right\}
\end{array}
\right) \\
&\leq& \max \left\{
\begin{array}{l}
\textbf{c}_{\textbf{s}}^{\transpose} 
\left(y^{*^1}\left(\bar{\textbf{x}} \wedge \tilde{\textbf{x}},\textbf{s} \right)-\bar{x}^1,y^{*^2}\left(\bar{\textbf{x}} \wedge \tilde{\textbf{x}},\textbf{s} \right)-\tilde{x}^2 \right), \\
\textbf{c}_{\textbf{s}}^{\transpose} 
\left(y^{*^1}\left(\bar{\textbf{x}} \vee \tilde{\textbf{x}},\textbf{s} \right)-\tilde{x}^1, y^{*^2}\left(\bar{\textbf{x}} \vee \tilde{\textbf{x}},\textbf{s} \right)-\bar{x}^2 \right)
\end{array}
\right\} \\
&=& \max \Bigl\{\textbf{c}_{\textbf{s}}^{\transpose} 
\Bigl(\textbf{y}^*\left(\bar{\textbf{x}} \wedge \tilde{\textbf{x}},\textbf{s}\right)-\left(\bar{\textbf{x}} \wedge \tilde{\textbf{x}}\right) \Bigr),\textbf{c}_{\textbf{s}}^{\transpose} 
\Bigl(\textbf{y}^*\left(\bar{\textbf{x}} \vee \tilde{\textbf{x}},\textbf{s}\right)-\left(\bar{\textbf{x}} \vee \tilde{\textbf{x}}\right) \Bigr) \Bigr\} \leq P.
\end{eqnarray*}
By a similar argument, $\textbf{c}_{\textbf{s}}^{\transpose} 
\left(\tilde{\textbf{y}}-\tilde{\textbf{x}}\right) \leq P$, and thus $\bar{\textbf{y}} \in \tilde{\cal{A}}^{\textbf{d}}\left(\bar{\textbf{x}},\textbf{s}\right)$, and $\tilde{\textbf{y}} \in \tilde{\cal{A}}^{\textbf{d}}\left(\tilde{\textbf{x}},\textbf{s}\right)$. So we have:
\begin{eqnarray} \label{Eq:mp:case1_1}
\min_{\textbf{y} \in
\tilde{\cal{A}}^{\textbf{d}}(\bar{\textbf{x}},\textbf{s})} \left\{G_{l-1}(\textbf{y},\textbf{s})\right\}
+\min_{\textbf{y} \in
\tilde{\cal{A}}^{\textbf{d}}(\tilde{\textbf{x}},\textbf{s})} \left\{G_{l-1}(\textbf{y},\textbf{s})\right\}
\leq G_{l-1}(\bar{\textbf{y}},\textbf{s}) + G_{l-1}(\tilde{\textbf{y}},\textbf{s})~.
\end{eqnarray}
\begin{figure}[htbp]
\centering \includegraphics[width=4.5in]{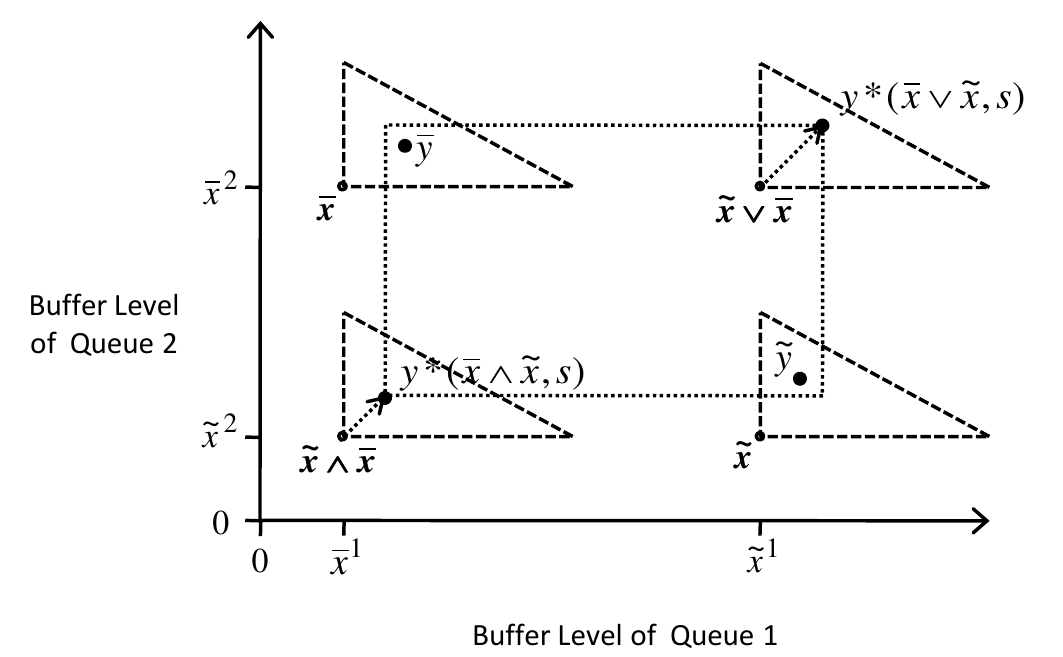} 
\caption{Construction of feasible points $\bar{\textbf{y}}$ and $\tilde{\textbf{y}}$ in Case 1 of the proof of supermodularity of $V_{l-1}(\cdot,\textbf{s})$.}\label{Fig:multi_proof:case1}
\end{figure}

Now define\footnote{If $y^{*^1}(\bar{\textbf{x}} \vee \tilde{\textbf{x}}, \textbf{s})-y^{*^1}(\bar{\textbf{x}} \wedge \tilde{\textbf{x}}, \textbf{s})=0$, let $\sigma$ be arbitrary in $[0,1]$. Similarly, if $y^{*^2}(\bar{\textbf{x}} \vee \tilde{\textbf{x}}, \textbf{s})-y^{*^2}(\bar{\textbf{x}} \wedge \tilde{\textbf{x}}, \textbf{s})=0$, let $\beta$ be arbitrary in $[0,1]$.}:
\begin{eqnarray*}
\begin{array}{l}
\sigma:= \frac{y^{*^1}(\bar{\textbf{x}} \vee \tilde{\textbf{x}}, \textbf{s})-\tilde{y}^1}{y^{*^1}(\bar{\textbf{x}} \vee \tilde{\textbf{x}}, \textbf{s})-y^{*^1}(\bar{\textbf{x}} \wedge \tilde{\textbf{x}}, \textbf{s})}~, \hbox{ and }\\
\\
\beta:=\frac{y^{*^2}(\bar{\textbf{x}} \vee \tilde{\textbf{x}}, \textbf{s})-\tilde{y}^2}{y^{*^2}(\bar{\textbf{x}} \vee \tilde{\textbf{x}}, \textbf{s})-y^{*^2}(\bar{\textbf{x}} \wedge \tilde{\textbf{x}}, \textbf{s})}~.
\end{array}
\end{eqnarray*}
Rearranging the definitions of $\sigma$ and $\beta$ yields:
\begin{eqnarray*}
\tilde{\textbf{y}} = \Bigl(
(1-\sigma) \cdot y^{*^1}(\bar{\textbf{x}} \vee \tilde{\textbf{x}}, \textbf{s}) + \sigma \cdot y^{*^1}(\bar{\textbf{x}} \wedge \tilde{\textbf{x}}, \textbf{s}),
(1-\beta) \cdot y^{*^2}(\bar{\textbf{x}} \vee \tilde{\textbf{x}}, \textbf{s}) + \beta \cdot y^{*^2}(\bar{\textbf{x}} \wedge \tilde{\textbf{x}}, \textbf{s})
\Bigr)~.
\end{eqnarray*}
It is also straightforward to check that:
\begin{eqnarray*}
\bar{\textbf{y}} = \Bigl(
\sigma \cdot y^{*^1}(\bar{\textbf{x}} \vee \tilde{\textbf{x}}, \textbf{s}) + (1-\sigma) \cdot y^{*^1}(\bar{\textbf{x}} \wedge \tilde{\textbf{x}}, \textbf{s}),
\beta \cdot y^{*^2}(\bar{\textbf{x}} \vee \tilde{\textbf{x}}, \textbf{s}) + (1-\beta) \cdot y^{*^2}(\bar{\textbf{x}} \wedge \tilde{\textbf{x}}, \textbf{s})
\Bigr)~.
\end{eqnarray*}
Note also that
\begin{eqnarray*}
y^{*^1}(\bar{\textbf{x}} \wedge \tilde{\textbf{x}}, \textbf{s}) &=& \min\left\{y^{*^1}(\bar{\textbf{x}} \wedge \tilde{\textbf{x}}, \textbf{s}), y^{*^1}(\bar{\textbf{x}} \wedge \tilde{\textbf{x}}, \textbf{s}) + (\tilde{x}^1-\tilde{x}^2) \right\} \\
&\leq & \min\left\{y^{*^1}(\bar{\textbf{x}} \vee \tilde{\textbf{x}}, \textbf{s}), y^{*^1}(\bar{\textbf{x}} \wedge \tilde{\textbf{x}}, \textbf{s}) + (\tilde{x}^1-\tilde{x}^2) \right\} \\
&=& \tilde{y}^1 \\
& \leq & y^{*^1}(\bar{\textbf{x}} \vee \tilde{\textbf{x}}, \textbf{s})~,
\end{eqnarray*}
and thus, $\sigma \in [0,1]$. Similarly, $y^{*^2}(\bar{\textbf{x}} \wedge \tilde{\textbf{x}}, \textbf{s}) \leq \tilde{y}^2 \leq y^{*^2}(\bar{\textbf{x}} \vee \tilde{\textbf{x}}, \textbf{s})$, and thus, $\beta \in [0,1]$.

Since $G_{l-1}(\cdot,\textbf{s})$ is convex and supermodular, we can now apply Lemma \ref{Le:mp:parallelogram}, with $\textbf{y}^*(\bar{\textbf{x}} \wedge \tilde{\textbf{x}}, \textbf{s})$ playing the role of $\textbf{z}$; $\textbf{y}^*(\bar{\textbf{x}} \vee \tilde{\textbf{x}}, \textbf{s})$ the role of $\hat{\textbf{z}}$; $\bar{\textbf{y}}$ the role of $\textbf{z}^{\sigma,\beta}$; and
$\tilde{\textbf{y}}$ the role of $\textbf{z}^{1-\sigma,1-\beta}$, to get:
\begin{eqnarray}\label{Eq:mp:case1_2}
 G_{l-1}(\bar{\textbf{y}},\textbf{s}) + G_{l-1}(\tilde{\textbf{y}},\textbf{s}) & \leq &
 G_{l-1}\Bigl(\textbf{y}^*(\bar{\textbf{x}} \wedge \tilde{\textbf{x}},\textbf{s}),\textbf{s}\Bigr) + G_{l-1}\Bigl(\textbf{y}^*(\bar{\textbf{x}} \vee \tilde{\textbf{x}},\textbf{s}),\textbf{s}\Bigr) \nonumber \\
&=&
\min_{\textbf{y} \in
\tilde{\cal{A}}^{\textbf{d}}(\bar{\textbf{x}} \wedge \tilde{\textbf{x}},\textbf{s})} \left\{G_{l-1}(\textbf{y},\textbf{s})\right\}
+\min_{\textbf{y} \in
\tilde{\cal{A}}^{\textbf{d}}(\bar{\textbf{x}} \vee \tilde{\textbf{x}},\textbf{s})} \left\{G_{l-1}(\textbf{y},\textbf{s})\right\}.
\end{eqnarray}
Combining equations (\ref{Eq:mp:case1_1}) and (\ref{Eq:mp:case1_2}) yields the desired result, (\ref{Eq:mp:superwts}).

\noindent \underline{Case 2}: $\textbf{y}^*(\bar{\textbf{x}} \vee \tilde{\textbf{x}},\textbf{s}) \nsucceq \textbf{y}^*(\bar{\textbf{x}} \wedge \tilde{\textbf{x}},\textbf{s}) \nsucc \textbf{y}^*(\bar{\textbf{x}} \vee \tilde{\textbf{x}},\textbf{s})$ \\
There are two possibilities for this case. The first possibility is that ${y}^{*^1}(\bar{\textbf{x}} \wedge \tilde{\textbf{x}},\textbf{s}) > {y}^{*^1}(\bar{\textbf{x}} \vee \tilde{\textbf{x}},\textbf{s})$ and ${y}^{*^2}(\bar{\textbf{x}} \wedge \tilde{\textbf{x}},\textbf{s}) \leq {y}^{*^2}(\bar{\textbf{x}} \vee \tilde{\textbf{x}},\textbf{s})$. The second possibility is that ${y}^{*^1}(\bar{\textbf{x}} \wedge \tilde{\textbf{x}},\textbf{s}) \leq {y}^{*^1}(\bar{\textbf{x}} \vee \tilde{\textbf{x}},\textbf{s})$ and ${y}^{*^2}(\bar{\textbf{x}} \wedge \tilde{\textbf{x}},\textbf{s}) > {y}^{*^2}(\bar{\textbf{x}} \vee \tilde{\textbf{x}},\textbf{s})$. We show (\ref{Eq:mp:superwts}) under the first possibility, and a symmetric argument can be used to show (\ref{Eq:mp:superwts}) under the second possibility. We have:
\begin{eqnarray} \label{Eq:mp:case2_1a}
{y}^{*^1}(\bar{\textbf{x}} \wedge \tilde{\textbf{x}},\textbf{s}) > {y}^{*^1}(\bar{\textbf{x}} \vee \tilde{\textbf{x}},\textbf{s}) \geq \max\left\{(\bar{\textbf{x}} \vee \tilde{\textbf{x}})^1,d^1\right\} = \max \left\{\tilde{x}^1,d^1\right\}~,
\end{eqnarray}
\begin{eqnarray} \label{Eq:mp:case2_2a}
{y}^{*^2}(\bar{\textbf{x}} \wedge \tilde{\textbf{x}},\textbf{s}) \geq \max\left\{(\bar{\textbf{x}} \wedge \tilde{\textbf{x}})^2,d^2\right\} = \max\left\{\tilde{x}^2,d^2\right\}~,
\end{eqnarray}
and
\begin{eqnarray} \label{Eq:mp:case2_3a}
\textbf{c}_{\textbf{s}}^{\transpose} 
\Bigl[\textbf{y}^*(\bar{\textbf{x}} \wedge \tilde{\textbf{x}},\textbf{s}) - \tilde{\textbf{x}} \Bigr] \leq \textbf{c}_{\textbf{s}}^{\transpose} 
\Bigl[\textbf{y}^*(\bar{\textbf{x}} \wedge \tilde{\textbf{x}},\textbf{s}) - (\bar{\textbf{x}} \wedge \tilde{\textbf{x}}) \Bigr] \leq P~.
\end{eqnarray}
Equations (\ref{Eq:mp:case2_1a}), (\ref{Eq:mp:case2_2a}), and (\ref{Eq:mp:case2_3a}) imply $\textbf{y}^*(\bar{\textbf{x}} \wedge \tilde{\textbf{x}},\textbf{s}) \in \tilde{\cal{A}}^{\textbf{d}}\left(\tilde{\textbf{x}},\textbf{s}\right)$. If it also happens that $\textbf{y}^*(\bar{\textbf{x}} \vee \tilde{\textbf{x}},\textbf{s}) \in \tilde{\cal{A}}^{\textbf{d}}\left(\bar{\textbf{x}},\textbf{s}\right)$, then we have:
\begin{align*}
& \min_{\textbf{y} \in
\tilde{\cal{A}}^{\textbf{d}}(\bar{\textbf{x}},\textbf{s})} \left\{G_{l-1}(\textbf{y},\textbf{s})\right\}
+\min_{\textbf{y} \in
\tilde{\cal{A}}^{\textbf{d}}(\tilde{\textbf{x}},\textbf{s})} \left\{G_{l-1}(\textbf{y},\textbf{s})\right\} \\
& \leq
G_{l-1}\bigl(\textbf{y}^*(\bar{\textbf{x}} \wedge \tilde{\textbf{x}},\textbf{s}), \textbf{s} \bigr) + G_{l-1}\bigl(\textbf{y}^*(\bar{\textbf{x}} \vee \tilde{\textbf{x}},\textbf{s}), \textbf{s} \bigr) \\
&=
\min_{\textbf{y} \in
\tilde{\cal{A}}^{\textbf{d}}(\bar{\textbf{x}} \wedge \tilde{\textbf{x}},\textbf{s})} \left\{G_{l-1}(\textbf{y},\textbf{s})\right\}
+\min_{\textbf{y} \in
\tilde{\cal{A}}^{\textbf{d}}(\bar{\textbf{x}} \vee \tilde{\textbf{x}},\textbf{s})} \left\{G_{l-1}(\textbf{y},\textbf{s})\right\}.
\end{align*}

Otherwise, define:
\begin{eqnarray*}
\gamma := \frac{\textbf{c}_{\textbf{s}}^{\transpose} 
\bigl[\textbf{y}^*(\bar{\textbf{x}} \vee \tilde{\textbf{x}},\textbf{s}) - \bar{\textbf{x}} \bigr]-P}{\textbf{c}_{\textbf{s}}^{\transpose} 
\bigl[\textbf{y}^*(\bar{\textbf{x}} \vee \tilde{\textbf{x}},\textbf{s})-\textbf{y}^*(\bar{\textbf{x}} \wedge \tilde{\textbf{x}},\textbf{s}) \bigr]}~.
\end{eqnarray*}
From $\textbf{y}^*(\bar{\textbf{x}} \vee \tilde{\textbf{x}},\textbf{s}) \notin \tilde{\cal{A}}^{\textbf{d}}\left(\bar{\textbf{x}},\textbf{s}\right)$ and $\textbf{y}^*(\bar{\textbf{x}} \wedge \tilde{\textbf{x}},\textbf{s}) \in \tilde{\cal{A}}^{\textbf{d}}\left(\bar{\textbf{x}} \wedge \tilde{\textbf{x}},\textbf{s}\right)$, we know:
\begin{eqnarray} \label{Eq:mp:case2_1}
\textbf{c}_{\textbf{s}}^{\transpose} 
\textbf{y}^*(\bar{\textbf{x}} \vee \tilde{\textbf{x}},\textbf{s}) > 
\textbf{c}_{\textbf{s}}^{\transpose} 
\bar{\textbf{x}}+P \geq
\textbf{c}_{\textbf{s}}^{\transpose} 
(\bar{\textbf{x}} \wedge \tilde{\textbf{x}})+P \geq
\textbf{c}_{\textbf{s}}^{\transpose} 
\textbf{y}^*(\bar{\textbf{x}} \wedge \tilde{\textbf{x}},\textbf{s})~.
\end{eqnarray}
It is clear from (\ref{Eq:mp:case2_1}) that the numerator 
and
denominator of $\gamma$ are positive,
and $\gamma \in [0,1]$.
Now define:
\begin{eqnarray*}
\begin{array}{l}
\bar{\textbf{y}}:= \gamma 
\textbf{y}^*(\bar{\textbf{x}} \wedge \tilde{\textbf{x}},\textbf{s}) + (1-\gamma) 
\textbf{y}^*(\bar{\textbf{x}} \vee \tilde{\textbf{x}},\textbf{s})~, \hbox{ and } \\
\tilde{\textbf{y}}:= (1-\gamma) 
\textbf{y}^*(\bar{\textbf{x}} \wedge \tilde{\textbf{x}},\textbf{s}) + \gamma 
\textbf{y}^*(\bar{\textbf{x}} \vee \tilde{\textbf{x}},\textbf{s})~.
\end{array}
\end{eqnarray*}
It is somewhat tedious but straightforward to show that $\bar{\textbf{y}} \in \tilde{\cal{A}}^{\textbf{d}}\left(\bar{\textbf{x}},\textbf{s}\right)$, and $\tilde{\textbf{y}} \in \tilde{\cal{A}}^{\textbf{d}}\left(\tilde{\textbf{x}},\textbf{s}\right)$. Thus, we have:
\begin{eqnarray} \label{Eq:mp:case2_2}
\min_{\textbf{y} \in
\tilde{\cal{A}}^{\textbf{d}}(\bar{\textbf{x}},\textbf{s})} \left\{G_{l-1}(\textbf{y},\textbf{s})\right\}
+\min_{\textbf{y} \in
\tilde{\cal{A}}^{\textbf{d}}(\tilde{\textbf{x}},\textbf{s})} \left\{G_{l-1}(\textbf{y},\textbf{s})\right\}
\leq
G_{l-1}(\bar{\textbf{y}}, \textbf{s}) + G_{l-1}(\tilde{\textbf{y}}, \textbf{s})~.
\end{eqnarray}
In Figure \ref{Fig:multi_proof:case2}, $\bar{\textbf{y}}$ is the point where the line segment connecting $\textbf{y}^*(\bar{\textbf{x}} \wedge \tilde{\textbf{x}},\textbf{s})$ and $\textbf{y}^*(\bar{\textbf{x}} \vee \tilde{\textbf{x}},\textbf{s})$ intersects the budget constraint (hypotenuse) of $\tilde{\cal{A}}^{\textbf{d}}\left(\bar{\textbf{x}},\textbf{s}\right)$, and $\tilde{\textbf{y}}$ is a point along this line segment the same distance from $\textbf{y}^*(\bar{\textbf{x}} \wedge \tilde{\textbf{x}},\textbf{s})$ as $\bar{\textbf{y}}$ is from $\textbf{y}^*(\bar{\textbf{x}} \vee \tilde{\textbf{x}},\textbf{s})$. By the convexity of $G_{l-1}(\cdot, \textbf{s})$ along this line segment, we have:
 \begin{eqnarray}\label{Eq:mp:case2_3}
 G_{l-1}(\bar{\textbf{y}}, \textbf{s}) + G_{l-1}(\tilde{\textbf{y}}, \textbf{s}) &\leq & G_{l-1}\Bigl(\textbf{y}^*(\bar{\textbf{x}} \wedge \tilde{\textbf{x}},\textbf{s}), \textbf{s}\Bigr)+G_{l-1}\Bigl(\textbf{y}^*(\bar{\textbf{x}} \vee \tilde{\textbf{x}},\textbf{s}), \textbf{s}\Bigr) \nonumber \\
 &=&
 \min_{\textbf{y} \in
\tilde{\cal{A}}^{\textbf{d}}(\bar{\textbf{x}} \wedge \tilde{\textbf{x}},\textbf{s})} \left\{G_{l-1}(\textbf{y},\textbf{s})\right\}
+\min_{\textbf{y} \in
\tilde{\cal{A}}^{\textbf{d}}(\bar{\textbf{x}} \vee \tilde{\textbf{x}},\textbf{s})} \left\{G_{l-1}(\textbf{y},\textbf{s})\right\}.
 \end{eqnarray}
Combining (\ref{Eq:mp:case2_2}) and (\ref{Eq:mp:case2_3}) yields the desired result, (\ref{Eq:mp:superwts}).
\begin{figure}[htbp]
\centering \includegraphics[width=3.5in]{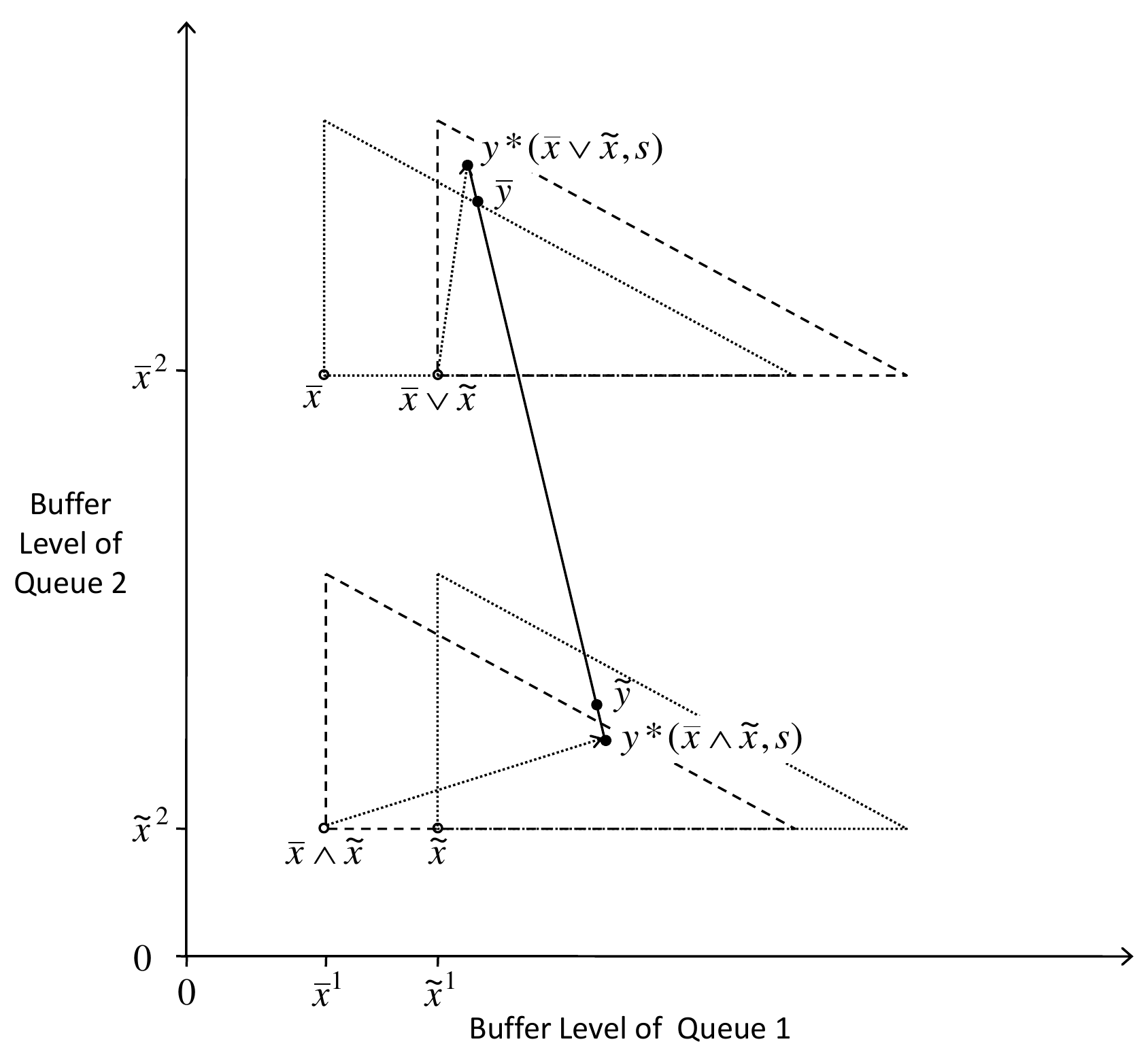} 
\caption{Construction of feasible points $\bar{\textbf{y}}$ and $\tilde{\textbf{y}}$ in Case 2 of the proof of supermodularity of $V_{l-1}(\cdot,\textbf{s})$.}\label{Fig:multi_proof:case2}
\end{figure}

\item[(iii)] $G_l(\textbf{y},\textbf{s})=
\textbf{c}_{\textbf{s}}^{\transpose}
\textbf{y}
+h(\textbf{y}-\textbf{d})
 +\alpha \cdot \Expectation
 \bigl[V_{l-1}(\textbf{y}-\textbf{d},\textbf{S})\bigr].$ By (i),
 for all $\textbf{s}$,
 $V_{l-1}(\textbf{x},\textbf{s})$ is convex in $\textbf{x}$; thus,
$V_{l-1}(\textbf{y}-\textbf{d},\textbf{s})$ is convex in
$\textbf{y}$ as it is the composition of a convex function with an
affine function. $\Expectation
 \bigl[V_{l-1}(\textbf{y}-\textbf{d},\textbf{S})\bigr]$ is also convex
 as it is the nonnegative weighted sum/integral of convex functions. It follows that $G_l(\textbf{y},\textbf{s})$, the sum
 of convex functions, is convex in $\textbf{y}$.

\item[(iv)] Supermodularity of $G_l(\textbf{y},\textbf{s})$ follows from the same series of arguments as (iii), because, like convexity, supermodularity is preserved under addition and scalar multiplication (Smith and McCardle refer to these as \emph{closed convex cone properties} \cite{mccardle}).

 \item[(v)] This step basically follows from Topkis' Theorem 2.8.1 \cite[pg. 76]{topkis}, but, for the reader's benefit, we reproduce the proof here with our notation. Let $y^2, \hat{y}^2\in [d^2,\infty)$ 
     be arbitrary
 with $y^2<\hat{y}^2$. Let $\bar{y}^1 \in \argmin_{y^1\in 
 [d^1,\infty)}
 \bigl\{G_l\left(y^1,y^2,\textbf{s}\right)\bigr\}$
and $\tilde{y}^1 \in \argmin_{y^1 \in 
[d^1,\infty)}
\bigl\{G_l\left(y^1,\hat{y}^2,\textbf{s}\right)\bigr\}$
be arbitrary. We want to show:
\begin{eqnarray*}
\bar{y}^1 \wedge \tilde{y}^1 
\in \argmin_{y^1 \in 
[d^1,\infty)}
\bigl\{G_l\left(y^1,\hat{y}^2,\textbf{s}\right)\bigr\}~.
\end{eqnarray*}
\noindent If $\tilde{y}^1 \leq \bar{y}^1$, this is trivial, so we
check that it is true for $\tilde{y}^1 > \bar{y}^1$. Since
$\bar{y}^1$ is a minimizer of
$G_l\left(\cdot,y^2,\textbf{s}\right)$, we have:
\begin{eqnarray} \label{Eq:multi_proof:super_mono_1}
G_l\left(\bar{y}^1,y^2,\textbf{s}\right)\leq
G_l\left(\tilde{y}^1,y^2,\textbf{s}\right)~,
\end{eqnarray}
and since $\tilde{y}^1$ is a minimizer of
$G_l\left(\cdot,\hat{y}^2,\textbf{s}\right)$, we have:
\begin{eqnarray} \label{Eq:multi_proof:super_mono_2}
G_l\left(\tilde{y}^1,\hat{y}^2,\textbf{s}\right)\leq
G_l\left(\bar{y}^1,\hat{y}^2,\textbf{s}\right)~.
\end{eqnarray}
By the supermodularity of $G_l(\cdot,\textbf{s})$, we have:
\begin{eqnarray*}
G_l\left(\tilde{y}^1,y^2,\textbf{s}\right) +
G_l\left(\bar{y}^1,\hat{y}^2,\textbf{s}\right) &\leq&
G_l\left(\tilde{y}^1 \wedge \bar{y}^1,y^2 \wedge
\hat{y}^2,\textbf{s}\right) + G_l\left(\tilde{y}^1 \vee
\bar{y}^1,y^2 \vee \hat{y}^2,\textbf{s}\right) \\
&=& G_l\left(\bar{y}^1,y^2,\textbf{s}\right) +
G_l\left(\tilde{y}^1,\hat{y}^2,\textbf{s}\right)~,
\end{eqnarray*}
or, rearranging terms:
\begin{eqnarray} \label{Eq:multi_proof:super_mono_3}
G_l\left(\tilde{y}^1,y^2,\textbf{s}\right) -
G_l\left(\bar{y}^1,y^2,\textbf{s}\right) \leq
G_l\left(\tilde{y}^1,\hat{y}^2,\textbf{s}\right) -
G_l\left(\bar{y}^1,\hat{y}^2,\textbf{s}\right)~.
\end{eqnarray}
Combining (\ref{Eq:multi_proof:super_mono_1}),
(\ref{Eq:multi_proof:super_mono_2}), and
(\ref{Eq:multi_proof:super_mono_3}) yields:
\begin{eqnarray} \label{Eq:multi_proof:super_mono_4}
0 \leq G_l\left(\tilde{y}^1,y^2,\textbf{s}\right) -
G_l\left(\bar{y}^1,y^2,\textbf{s}\right) \leq
G_l\left(\tilde{y}^1,\hat{y}^2,\textbf{s}\right) -
G_l\left(\bar{y}^1,\hat{y}^2,\textbf{s}\right) \leq 0~.
\end{eqnarray}
So (\ref{Eq:multi_proof:super_mono_4}) holds with equality
throughout, implying
$G_l\left(\tilde{y}^1,\hat{y}^2,\textbf{s}\right) =
G_l\left(\bar{y}^1,\hat{y}^2,\textbf{s}\right)$, and we conclude:
\begin{eqnarray*}
\tilde{y}^1 \wedge \bar{y}^1 = \bar{y}^1 \in \argmin_{y^1 \in 
[d^1,\infty)}
\bigl\{G_l\left(y^1,\hat{y}^2,\textbf{s}\right)\bigr\}.
\end{eqnarray*}
\noindent Since $\bar{y}^1$ and $\tilde{y}^1$ were chosen
arbitrarily, we have:
\begin{eqnarray*}
\inf \left\{\argmin_{y_n^1 \in 
[d^1,\infty)}
\biggl\{G_n\left(y_n^1,y_n^2,s^1,s^2\right)\biggr\}
\right\} \geq \inf \left\{\argmin_{y_n^1 \in 
[d^1,\infty)}
\biggl\{G_n\left(y_n^1,{\hat{y}_n}^2,s^1,s^2\right)\biggr\}
\right\}.
\end{eqnarray*}
\noindent The first implication in (v) follows from a symmetric
argument.
\qedsymbol
\end{itemize}

%% file: appendix_two_user_structure_proof.tex
Let $n \in \{1,2,\ldots,N\}$ and $\textbf{s} \in {\cal S}$ be arbitrary. We start by proving \eqref{Eq:two:y_star}. First, let $\textbf{x} \in  {\mathcal R}_{I}(n,\textbf{s})$ and $\hat{\textbf{y}} \in \tilde{\cal{A}}^{\textbf{d}}({\textbf{x}},\textbf{s})$ be arbitrary. We know from Theorem \ref{Th:two_users:properties} that $G_n(\cdot,\textbf{s})$ is convex on $[d^1,\infty) \times [d^2,\infty)$, which implies that $G_n(\cdot,\textbf{s})$ is also convex on any line segment in $[d^1,\infty) \times [d^2,\infty)$ (see, e.g., \cite[Theorem 4.1]{rockafellar}). Specifically, by the convexity of $G_n(\cdot,\textbf{s})$ along the line $y^1=\hat{y}^1$ and the fact that $\hat{y}^2 \geq x^2 \geq f_n^2(\hat{y}^1,\textbf{s})$, we have:
\begin{eqnarray} \label{Eq:twos:vert}
G_n(\hat{\textbf{y}},\textbf{s}) \geq  G_n\bigl((\hat{y}^1,x^2),\textbf{s}\bigr) \geq G_n\Bigl(\bigl(\hat{y}^1,f_n^2(\hat{y}^1,\textbf{s})\bigr),\textbf{s}\Bigr)~.
\end{eqnarray}
Similarly, by the convexity of $G_n(\cdot,\textbf{s})$ along the line $y^2=x^2$ and the fact that $\hat{y}^1 \geq x^1 \geq f_n^1(x^2,\textbf{s})$, we have:
\begin{eqnarray} \label{Eq:twos:hor}
G_n\bigl((\hat{y}^1,x^2),\textbf{s}\bigr) \geq G_n(\textbf{x},\textbf{s}) \geq G_n\Bigl(\bigl(f_n^1(x^2,\textbf{s}),x^2\bigr),\textbf{s}\Bigr)~.
\end{eqnarray}
Combining \eqref{Eq:twos:vert} and \eqref{Eq:twos:hor} yields:
\begin{eqnarray*}
G_n(\hat{\textbf{y}},\textbf{s}) \geq G_n\Bigl(\bigl(\hat{y}^1,x^2\bigr),\textbf{s}\Bigr) \geq G_n(\textbf{x},\textbf{s})~,
\end{eqnarray*}
and we conclude $G_n(\textbf{x},\textbf{s}) = \min_{\textbf{y} \in \tilde{\cal{A}}^{\textbf{d}}({\textbf{x}},\textbf{s})} \left\{G_n(\textbf{y},\textbf{s}) \right\}$.

Second, let $\textbf{x} \in  {\mathcal R}_{II}(n,\textbf{s})$ be arbitrary. Then $\textbf{b}_n(\textbf{s}) \in \tilde{\cal{A}}^{\textbf{d}}({\textbf{x}},\textbf{s})$ and $\textbf{b}_n(\textbf{s})$ is a global minimizer of $G_n(\cdot,\textbf{s})$, so it is clearly optimal to transmit to bring the receivers' buffer levels up to $\textbf{b}_n(\textbf{s})$.

Next, let $\textbf{x} \in {\mathcal R}_{III-A}(n,\textbf{s})$ and $\tilde{\textbf{y}} \in \tilde{\cal{A}}^{\textbf{d}}({\textbf{x}},\textbf{s})$ be arbitrary. By definition of $f_n^1(\cdot,\textbf{s})$, we have:
\begin{eqnarray} \label{Eq:twos:iiia1}
G_n(\tilde{\textbf{y}},\textbf{s}) \geq  G_n\Bigl(\bigl(f_n^1(\tilde{y}^2,\textbf{s}),\tilde{y}^2 \bigr),\textbf{s}\Bigr)~.
\end{eqnarray}
Furthermore, the function $\min_{y^1 \in [d^1,\infty)}\left\{G_n\Bigl((y^1,y^2),\textbf{s}\Bigr)\right\}$ is convex in $y^2$ since $[d^1,\infty)$ is a convex set (see, e.g., \cite[pp. 101-102]{boyd}).
Thus, $\tilde{y}^2 \geq x^2 \geq b_n^2(\textbf{s})$ implies:
\begin{align} \label{Eq:twos:iiia2}
&G_n\Bigl(\bigl(f_n^1(\tilde{y}^2,\textbf{s}),\tilde{y}^2\bigr),\textbf{s}\Bigr) \nonumber \\
&\geq
G_n\Bigl(\bigl(f_n^1(x^2,\textbf{s}),x^2\bigr),\textbf{s}\Bigr)  \\
&\geq
G_n\Bigl(\bigl(f_n^1(b_n^2(\textbf{s}),\textbf{s}),b_n^2(\textbf{s})\bigr),\textbf{s}\Bigr) \nonumber \\
&=G_n\Bigl(\textbf{b}_n(\textbf{s}),\textbf{s}\Bigr)~. \nonumber
\end{align}
Combining \eqref{Eq:twos:iiia1} and \eqref{Eq:twos:iiia2} yields:
\begin{eqnarray*}
G_n(\tilde{\textbf{y}},\textbf{s}) \geq G_n\Bigl(\bigl(f_n^1(x^2,\textbf{s}),x^2\bigr),\textbf{s}\Bigr)~,
\end{eqnarray*}
and $\textbf{x} \in {\mathcal R}_{III-A}(n,\textbf{s})$ implies $\Bigl(f_n^1(x^2,\textbf{s}),x^2\Bigr) \in \tilde{\cal{A}}^{\textbf{d}}({\textbf{x}},\textbf{s})$. Since $\tilde{\textbf{y}} \in \tilde{\cal{A}}^{\textbf{d}}({\textbf{x}},\textbf{s})$ was arbitrary, we conclude $\textbf{y}_n^*(\textbf{x},\textbf{s}) = \Bigl(f_n^1(x^2,\textbf{s}),x^2\Bigr)$ is optimal.

The optimality of $\textbf{y}_n^*(\textbf{x},\textbf{s})=\Bigl(x^1 ,f_n^2(x^1,\textbf{s}) \Bigr)$ for $\textbf{x} \in  {\mathcal R}_{III-B}(n,\textbf{s})$ follows from a symmetric argument, using the convexity of $G_n(\cdot,\textbf{s})$ along the curve $\Bigl(x^1,f_n^2(x^1,\textbf{s})\Bigr)$.

Finally, we prove \eqref{Eq:two:full}. Define:
\begin{eqnarray*}
{\cal{H}}^{\textbf{d}}({\textbf{x}},\textbf{s}):= \Bigl\{\textbf{y} \in [d^1,\infty) \times [d^2,\infty)~:~\textbf{y} \succeq \textbf{x} \hbox{ and } \textbf{c}_{\textbf{s}}^{\transpose}\left[{\textbf{y}} - \textbf{x}\right] = P\Bigr\} \subset \tilde{\cal{A}}^{\textbf{d}}({\textbf{x}},\textbf{s})~.
\end{eqnarray*}
First, let $\textbf{x} \in {\mathcal R}_{IV-B}(n,\textbf{s})$ and $\hat{\textbf{y}} \in \tilde{\cal{A}}^{\textbf{d}}({\textbf{x}},\textbf{s})$ be arbitrary such that $\textbf{c}_{\textbf{s}}^{\transpose}\left[\hat{\textbf{y}} - \textbf{x} \right] < P$.
Define
\begin{eqnarray*}
\lambda_0 := \frac{\textbf{c}_{\textbf{s}}^{\transpose}\textbf{b}_n(\textbf{s})-\textbf{c}_{\textbf{s}}^{\transpose}\textbf{x}-P}{\textbf{c}_{\textbf{s}}^{\transpose}\textbf{b}_n(\textbf{s})-\textbf{c}_{\textbf{s}}^{\transpose}\hat{\textbf{y}}}~.
\end{eqnarray*}
Note that $\textbf{c}_{\textbf{s}}^{\transpose}\left[\hat{\textbf{y}} - \textbf{x} \right] < P$ and $\textbf{c}_{\textbf{s}}^{\transpose}\left[\textbf{b}_n(\textbf{s}) - \textbf{x} \right] > P$ imply $\lambda_0 \in (0,1)$. Then define:
\begin{eqnarray*}
\tilde{\textbf{y}} := \lambda_0\hat{\textbf{y}} + (1-\lambda_0)\textbf{b}_n(\textbf{s})~.
\end{eqnarray*}
By the convexity of $G_n(\cdot,\textbf{s})$ along the line segment from $\hat{\textbf{y}}$ to $\textbf{b}_n(\textbf{s})$, we have:
\begin{eqnarray*}
G_n(\hat{\textbf{y}},\textbf{s}) \geq G_n(\tilde{\textbf{y}},\textbf{s}) \geq G_n\Bigl(\textbf{b}_n(\textbf{s}),\textbf{s}\Bigr)~.
\end{eqnarray*}
Since $\hat{\textbf{y}} \in \tilde{\cal{A}}^{\textbf{d}}({\textbf{x}},\textbf{s})$ was arbitrary, we conclude:
\begin{eqnarray*}
\min_{\textbf{y} \in \tilde{\cal{A}}^{\textbf{d}}({\textbf{x}},\textbf{s})} \Bigl\{G_n(\textbf{y},\textbf{s})\Bigr\} = \min_{\textbf{y} \in {\cal{H}}^{\textbf{d}}({\textbf{x}},\textbf{s})}\Bigl\{G_n(\textbf{y},\textbf{s})\Bigr\}~.
\end{eqnarray*}
Next, let $\textbf{x} \in {\mathcal R}_{IV-C}(n,\textbf{s})$ and $\hat{\textbf{y}} \in \tilde{\cal{A}}^{\textbf{d}}({\textbf{x}},\textbf{s})$ be arbitrary such that $\textbf{c}_{\textbf{s}}^{\transpose}\left[\hat{\textbf{y}} - \textbf{x} \right] < P$. We consider two exhaustive cases, and
for each case, we construct a $\tilde{\textbf{y}} \in {\cal{H}}^{\textbf{d}}({\textbf{x}},\textbf{s})$ such that $G_n\left(\tilde{\textbf{y}},\textbf{s}\right) \leq G_n\left(\hat{\textbf{y}},\textbf{s}\right)$.
\medskip

\noindent \underline{Case 1}: $\hat{{y}}^2 < {f}_n^2\left(\hat{y}^1,\textbf{s}\right)$ and $\bar{\textbf{y}}:=\Bigl(\hat{y}^1, {f}_n^2\left(\hat{y}^1,\textbf{s}\right)\Bigr) \notin \tilde{\cal{A}}^{\textbf{d}}({\textbf{x}},\textbf{s})$  \\
Let $\tilde{\textbf{y}}:= \Bigl(\hat{y}^1, x^2 + \frac{P-c_{s^1}\cdot[\hat{y}^1-x^1]}{c_{s^2}} \Bigr)$. Then, by the convexity of $G_n(\cdot,\textbf{s})$ along $y^1=\hat{y}^1$, the definition of ${f}_n^2\left(\hat{y}^1,\textbf{s}\right)$, and $\hat{y}^2 \leq \tilde{y}^2 \leq {f}_n^2\left(\hat{y}^1,\textbf{s}\right)$, we have:
\begin{eqnarray*}
G_n\left(\bar{\textbf{y}},\textbf{s}\right)= G_n\Bigl(\bigl(\hat{y}^1, {f}_n^2(\hat{y}^1,\textbf{s})\bigr),\textbf{s}\Bigr) \leq  G_n\left(\tilde{\textbf{y}},\textbf{s}\right) \leq G_n\left(\hat{\textbf{y}},\textbf{s}\right)~.
 \end{eqnarray*}
 It is also straightforward to check that $\tilde{\textbf{y}} \in {\cal{H}}^{\textbf{d}}({\textbf{x}},\textbf{s})$, as desired.

 \noindent \underline{Case 2}: All other $\hat{\textbf{y}} \in \tilde{\cal{A}}^{\textbf{d}}({\textbf{x}},\textbf{s})$ such that $\textbf{c}_{\textbf{s}}^{\transpose}\left[\hat{\textbf{y}} - \textbf{x} \right] < P$ \\
By the definition of $f_n^2\left(\hat{y}^1,\textbf{s}\right)$,
 we have:
 \begin{eqnarray} \label{Eq:twos:lower}
 G_n\left(\hat{\textbf{y}},\textbf{s}\right) \geq G_n\Bigl(\bigl(\hat{y}^1,f_n^2(\hat{y}^1,\textbf{s})\bigr),\textbf{s}\Bigr)~.
 \end{eqnarray}
Define:
 \begin{align*}
& \tilde{y}^1 := \sup\left\{
  y^1 \in \left[x^1,\hat{y}^1\right):
\textbf{c}_{\textbf{s}}^{\transpose}\Bigl({y}^1, f_n^2\left({y}^1,\textbf{s}\right)\Bigr) \geq \textbf{c}_{\textbf{s}}^{\transpose}\textbf{x} + P
 \right\}, \hbox{ and } \\
 & \tilde{y}^2 := \frac{P-c_{s^1}\cdot\left[\tilde{y}^1-x^1\right]}{c_{s^2}}~.
 \end{align*}
 By the convexity of $G_n(\cdot,\textbf{s})$ along $\Bigl(y^1, f_n^2\left({y}^1,\textbf{s}\right)\Bigr)$, 
 we have:
  \begin{eqnarray} \label{Eq:twos:mid}
G_n\Bigl(\bigl(\hat{y}^1,f_n^2(\hat{y}^1,\textbf{s})\bigr),\textbf{s}\Bigr) \geq G_n\Bigl(\bigl(\tilde{y}^1,f_n^2(\tilde{y}^1,\textbf{s})\bigr),\textbf{s}\Bigr) ~.
 \end{eqnarray}
 Furthermore, we have:
 \begin{eqnarray} \label{Eq:twos:equal}
G_n\Bigl(\bigl(\tilde{y}^1,f_n^2(\tilde{y}^1,\textbf{s})\bigr),\textbf{s}\Bigr) =
G_n\Bigl(\bigl(\tilde{y}^1,\tilde{y}^2\bigr),\textbf{s}\Bigr)= G_n\left(\tilde{\textbf{y}},\textbf{s}\right)~.
 \end{eqnarray}
 If $\tilde{y}^2 = f_n^2\left(\tilde{y}^1,\textbf{s}\right)$, \eqref{Eq:twos:equal} is trivial. Otherwise, there is a discontinuity in $f_n^2(\cdot,\textbf{s})$ at $\tilde{y}^1$, and we have:
 \begin{eqnarray} \label{Eq:twos:stara}
 \lim\limits_{y^1 \nearrow \tilde{y}^1}f_n^2\left({y}^1,\textbf{s}\right) \geq \tilde{y}^2 \geq \lim\limits_{y^1 \searrow \tilde{y}^1}f_n^2\left({y}^1,\textbf{s}\right)~,
 \end{eqnarray}
 with at least one of the inequalities being strict. Nonetheless,  $G_n\Bigl(\bigl(y^1,f_n^2(y^1,\textbf{s})\bigr),\textbf{s}\Bigr)$ is a continuous function of $y^1$, and therefore: 
\begin{eqnarray}  \label{Eq:twos:starb}
G_n\Bigl(\bigl(\tilde{y}^1, \lim\limits_{y^1 \nearrow \tilde{y}^1}f_n^2\left({y}^1,\textbf{s}\right)\bigr) ,\textbf{s}\Bigr) 
= G_n\Bigl(\bigl(\tilde{y}^1, \lim\limits_{y^1 \searrow \tilde{y}^1}f_n^2\left({y}^1,\textbf{s}\right) \bigr) ,\textbf{s}\Bigr) 
= G_n\Bigl(\bigl(\tilde{y}^1,f_n^2\left(\tilde{y}^1,\textbf{s}\right) \bigr) ,\textbf{s}\Bigr).
\end{eqnarray}
The convexity of $G_n(\cdot,\textbf{s})$ along the line $y^1=\tilde{y}^1$ and \eqref{Eq:twos:starb} imply:
\begin{eqnarray*}
G_n\Bigl(\bigl(\tilde{y}^1,y^2 \bigr) ,\textbf{s}\Bigr) = G_n\Bigl(\bigl(\tilde{y}^1,f_n^2\left(\tilde{y}^1,\textbf{s}\right) \bigr) ,\textbf{s}\Bigr)~,~\forall y^2 \in \left[ \lim\limits_{y^1 \searrow \tilde{y}^1}f_n^2\left({y}^1,\textbf{s}\right), \lim\limits_{y^1 \nearrow \tilde{y}^1}f_n^2\left({y}^1,\textbf{s}\right)\right]~,
\end{eqnarray*}
which in combination with \eqref{Eq:twos:stara} implies \eqref{Eq:twos:equal}.
 Combining \eqref{Eq:twos:lower}-\eqref{Eq:twos:equal} yields the desired result: $G_n\left(\tilde{\textbf{y}},\textbf{s}\right) \leq G_n\left(\hat{\textbf{y}},\textbf{s}\right)$ for a $\tilde{\textbf{y}} \in {\cal{H}}^{\textbf{d}}({\textbf{x}},\textbf{s})$.

The validity of \eqref{Eq:two:full} for $\textbf{x} \in {\mathcal R}_{IV-A}(n,\textbf{s})$ follows from a symmetric argument, completing the proof of \eqref{Eq:two:full} and Theorem \ref{Th:two_users:structure}.
\qedsymbol

%% file: appendix_structure_infinite_proof.tex
Our line of analysis is similar in spirit to \cite{fedzipII}, \cite{iglehart}, and \cite[Chapter 8]{heyman_sobel}. 
Let $x \in \Real_+$ and $s \in {\cal{S}}$ be
arbitrary. First, we show inductively that $V_1(x,s) \leq V_2(x,s) \leq
\ldots \leq V_n(x,s) \leq V_{n+1}(x,s) \leq \ldots$.

\noindent \underline{Base Case}: $n=1$
\begin{eqnarray*}
V_1(x,s) &=&
\min\limits_{\bigl\{\max(0,d-x) \leq z \leq \tilde{z}_{\max}(s)\bigr\}}
\left\{
c(z,s) + h(x+z-d)
\right\} \\
&\leq&
\min\limits_{\bigl\{\max(0,d-x) \leq z \leq \tilde{z}_{\max}(s)\bigr\}}
\left\{
\begin{array}{l}
c(z,s) + h(x+z-d) \\
+ \alpha \cdot \Expectation \bigl[V_{1}(x+z-d,S_{1}) \bigm| S_2 =s\bigr]
\end{array}
\right\} \\
&=&
V_2(x,s)~,
\end{eqnarray*}
\noindent where the inequality follows from $V_1(x,s) \geq 0,~\forall x, \forall s$.
\medskip

\noindent \underline{Induction Step}: 
 Assume $V_n(x,s) \leq V_{n+1}(x,s)$ for $n=1,2,\ldots,m-1$. We show it is true for $n=m$:
\begin{eqnarray*}
V_m(x,s)
&=&
\min\limits_{\bigl\{\max(0,d-x) \leq z \leq \tilde{z}_{\max}(s)\bigr\}}
\left\{
\begin{array}{l}
c(z,s) + h(x+z-d) \\
+ \alpha \cdot \Expectation \bigl[V_{m-1}(x+z-d,S_{m-1}) \bigm| S_m =s\bigr]
\end{array}
\right\} \\
&\leq&
\min\limits_{\bigl\{\max(0,d-x) \leq z \leq \tilde{z}_{\max}(s)\bigr\}}
\left\{
\begin{array}{l}
c(z,s) + h(x+z-d) \\
+ \alpha \cdot \Expectation \bigl[V_{m}(x+z-d,S_{m}) \bigm| S_{m+1} =s\bigr]
\end{array}
\right\} \\
&=&
V_{m+1}(x,s)~,
\end{eqnarray*}
\noindent where the inequality follows from the induction hypothesis and the homogeneity of the Markov process representing the channel condition. So, for every $x \in \Real_+$ and $s \in {\cal{S}}$, $\left\{V_n(x,s)\right\}_{n=1,2,\ldots}$ is a nondecreasing sequence.

Next, consider a policy ${\boldsymbol{\pi}}^d$ transmitting $d$ packets in every slot, regardless of channel condition. Define:
\begin{eqnarray}\label{Eq:ctildemaxdef}
\tilde{c}_{\max}:=\sup\limits_{\substack{s \in {\cal S} \\ k \in \{0,1,\ldots,K\}}}\left\{\tilde{c}_k(s) \right\}<\infty~.
\end{eqnarray}
Then we have:
\begin{eqnarray*}
V_n(x,s) \leq
V_n^{{\boldsymbol{\pi}}^d}(x,s) \leq \Bigl(\tilde{c}_{\max} \cdot d + h(x) \Bigr) \frac{1-\alpha^n}{1-\alpha} \leq \Bigl(\tilde{c}_{\max} \cdot d + h(x) \Bigr) \frac{1}{1-\alpha} 
< \infty~,
\end{eqnarray*}
\noindent so $\left\{V_n(x,s)\right\}_{n=1,2,\ldots}$ is a bounded nondecreasing
sequence, implying 
$\lim_{n \rightarrow \infty}V_n(x,s)$ exists
and is finite, $\forall x \in \Real_+, \forall s \in {\cal{S}}$.

We now move on to part (b). 
Recall from Section \ref{Se:fin_proof} that $V_n(x,s)$ is convex in $x$,
for all $n$ and all $s$. Define $V_{\infty}(x,s):=\lim_{n \rightarrow \infty} V_n(x,s)$. Let $s \in
{\cal{S}}$ be arbitrary, but fixed.
$V_{\infty}(x,s)=\sup_{n \in \Nat} V_n(x,s)$, so $V_{\infty}(x,s)$ is convex in $x$
as it is the pointwise supremum of the convex functions
$\left\{V_n(x,s)\right\}_{n=1,2,\ldots}$.

Define $\tilde{g}_{\infty}: [d,\infty) \times {\cal{S}} \rightarrow \Real_+$ by
\begin{eqnarray} \label{Eq:app:mct}
\tilde{g}_{\infty}(y,s) &:=& 
h(y-d) + \alpha \cdot
\Expectation \left[V_{\infty}(y-d,S^{\prime}) \mid S = s \right] \nonumber \\
&=& 
h(y-d) + \alpha \cdot
\Expectation \left[\lim\limits_{n \rightarrow \infty}V_{n}(y-d,S^{\prime}) \mid S = s \right] \nonumber \\
&=& 
h(y-d) + \alpha \cdot
\lim\limits_{n \rightarrow \infty}\Expectation \left[V_{n}(y-d,S^{\prime}) \mid S = s \right] \\
&=& \lim\limits_{n \rightarrow \infty} \tilde{g}_n(y,s)~, \nonumber
\end{eqnarray}
\noindent where (\ref{Eq:app:mct}) follows from the homogeneity of the Markov process representing the channel condition and the Monotone Convergence Theorem. 
Furthermore, for each $s \in {\cal{S}}$, $\tilde{g}_{\infty}(y,s)$ is convex in $y$ 
and $\lim\limits_{y
\rightarrow \infty}\tilde{g}_{\infty}(y,s) \geq \lim\limits_{y \rightarrow
\infty} h(y-d) 
= \infty$. Thus, for every $s$, at least
one finite number achieves the global minimum of
$\tilde{g}_{\infty}(y,s)$.

Next, we proceed to part (d), and let $s \in {\cal S}$ be arbitrary.
Define $b_{\infty,-1}(s) := \infty$ and
\begin{eqnarray*}
b_{\infty,k}(s) :=
\max\Bigl\{d, \inf\bigl\{b \bigm| \tilde{g}_{\infty}^{\prime+}(b,s) \geq -\tilde{c}_k(s) \bigr\} \Bigr\}~,~\forall k \in \{0,1,\ldots,K\}~.
\end{eqnarray*}
Clearly, $b_{\infty,-1}(s) = \lim\limits_{n\rightarrow\infty} b_{n,-1}(s)$, as $b_{n,-1}(s):=\infty$ for every $n$. Let $k \in \{0,1,\ldots,K\}$ be arbitrary. We want to show:
\begin{eqnarray*}
\lim\limits_{n\rightarrow\infty}
b_{n,k}(s)&=&\lim\limits_{n\rightarrow\infty} \max\Bigl\{d, \inf\bigl\{b \bigm| \tilde{g}_{n}^{\prime+}(b,s) \geq -\tilde{c}_k(s) \bigr\} \Bigr\} \nonumber \\
&=& \max\Bigl\{d, \inf\bigl\{b \bigm| \tilde{g}_{\infty}^{\prime+}(b,s) \geq -\tilde{c}_k(s) \bigr\} \Bigr\} :=
b_{\infty,k}(s)~.
\end{eqnarray*}
By the continuity of $\max\{d,\cdot\}$, it suffices to show:
\begin{eqnarray}\label{Eq:inf_proof:crits}
\lim\limits_{n\rightarrow\infty} \Bigl\{\inf\bigl\{b \bigm| \tilde{g}_{n}^{\prime+}(b,s) \geq -\tilde{c}_k(s) \bigr\}\Bigr\}
= \inf\bigl\{b \bigm| \tilde{g}_{\infty}^{\prime+}(b,s) \geq -\tilde{c}_k(s) \bigr\}~.
\end{eqnarray}
Before proceeding to show \eqref{Eq:inf_proof:crits}, 
we present
a lemma due to Sobel \cite[Lemma 3,~pg.~732]{sobel_71},
which is also presented in \cite[Lemma 8-5,~pg.~425]{heyman_sobel}.
\medskip


\begin{lemma}[Sobel, 1971] \label{Le:inf:sobel}
Let $g,g_1,g_2,\ldots$ be convex functions on an open convex subset $X$ of $\Real$ such that  $g_n(x) \rightarrow g(x)$ as $n \rightarrow \infty$ and  $g_n(x)\leq g_{n+1}(x)$ for all $n$ and $x$. Let $g_n^{\prime-}(x)$ and $g^{\prime-}(x)$ denote derivatives from the left and $g_n^{\prime+}(x)$ and $g^{\prime+}(x)$ denote derivatives from the right. Then for all $x \in X$:
\begin{eqnarray} \label{Eq:sobel}
g^{\prime-}(x) \leq \liminf\limits_{n \rightarrow \infty}g_n^{\prime-}(x) \leq \limsup\limits_{n \rightarrow \infty}g_n^{\prime+}(x) \leq g^{\prime+}(x)~.
\end{eqnarray}
\end{lemma}
~\\
We now prove \eqref{Eq:inf_proof:crits} by contradiction. Define:
\begin{eqnarray*}
\hat{b}_{n,k}(s)&:=&\inf\bigl\{b \bigm| \tilde{g}_{n}^{\prime+}(b,s) \geq -\tilde{c}_k(s) \bigr\}, \hbox{ and} \\
\hat{b}_{\infty,k}(s)&:=&\inf\bigl\{b \bigm| \tilde{g}_{\infty}^{\prime+}(b,s) \geq -\tilde{c}_k(s) \bigr\}~. 
\end{eqnarray*}
First, assume $\liminf\limits_{n \rightarrow \infty} \hat{b}_{n,k}(s) < \hat{b}_{\infty,k}(s)$, so there exists an $x_0 \in \Real_+$ such that $d < x_0 < \hat{b}_{\infty,k}(s)$, and a sequence $\left\{n_i\right\}_{i=1,2,\ldots}$ such that
$\lim\limits_{i \rightarrow \infty} \hat{b}_{n_i,k}(s) = x_0$.
Then we have:
\begin{eqnarray}
-\tilde{c}_k(s) &\leq& \lim_{i \rightarrow \infty} \tilde{g}_{n_i}^{\prime+}(x_0,s) \label{Eq:bconv:2a}\\
&\leq& \limsup_{n \rightarrow \infty} \tilde{g}_{n}^{\prime+}(x_0,s) \nonumber \\  
&\leq& \tilde{g}_{\infty}^{\prime+}(x_0,s)~. \label{Eq:bconv:2c}
\end{eqnarray}
Here, \eqref{Eq:bconv:2a} follows from $\lim\limits_{i \rightarrow \infty} \hat{b}_{n_i,k}(s) = x_0$, and the fact that $\tilde{g}_{n}^{\prime+}(\cdot,s)$ is continuous from the right. Equation \eqref{Eq:bconv:2c} follows from Lemma \ref{Le:inf:sobel}.
Yet, $\tilde{g}_{\infty}^{\prime+}(x_0,s) \geq -\tilde{c}_k(s)$ implies $\hat{b}_{\infty,k}(s) \leq x_0$, which is a contradiction. We conclude:
\begin{eqnarray} \label{Eq:bconv:conc1}
\liminf\limits_{n \rightarrow \infty} \hat{b}_{n,k}(s) \geq \hat{b}_{\infty,k}(s).
\end{eqnarray}

Next, assume $\limsup\limits_{n \rightarrow \infty} \hat{b}_{n,k}(s) > \hat{b}_{\infty,k}(s) \geq d$, 
and define: 
\begin{eqnarray*}
x_1:=\frac{\limsup\limits_{n \rightarrow \infty} \hat{b}_{n,k}(s) + \hat{b}_{\infty,k}(s)}{2}~.
\end{eqnarray*}
Then we have:
\begin{eqnarray}
-\tilde{c}_k(s)&\leq&\tilde{g}_{\infty}^{\prime+}\bigl(\hat{b}_{\infty,k}(s),s \bigr) \label{Eq:bconv:1a} \\
&\leq& \tilde{g}_{\infty}^{\prime-}\bigl(x_1,s \bigr) \label{Eq:bconv:1b} \\
&\leq& \liminf_{n \rightarrow \infty}\tilde{g}_{n}^{\prime-}\bigl(x_1,s \bigr)  \label{Eq:bconv:1c} \\
&\leq& \liminf_{n \rightarrow \infty}\tilde{g}_{n}^{\prime+}\bigl(x_1,s \bigr) \label{Eq:bconv:1d}
\end{eqnarray}
Here, \eqref{Eq:bconv:1a} follows from the fact that $\tilde{g}_{\infty}^{\prime+}(\cdot,s)$ is continuous from the right; \eqref{Eq:bconv:1c} follows from Lemma \ref{Le:inf:sobel}\footnote{One hypothesis of Lemma \ref{Le:inf:sobel} is that all functions are defined on an \emph{open} convex subset of $\Real$. While our functions $\tilde{g}_{\infty}(\cdot,s)$ and $\left\{\tilde{g}_{n}(\cdot,s)\right\}_{n \in \Nat}$ are defined on $[d,\infty)$, we only apply Lemma \ref{Le:inf:sobel} at the points $x_0,x_1 \in (d,\infty)$. Thus, equations \eqref{Eq:bconv:2c} and \eqref{Eq:bconv:1c} follow from the application of Lemma \ref{Le:inf:sobel} to the restrictions of the functions $\tilde{g}_{\infty}(\cdot,s)$ and $\left\{\tilde{g}_{n}(\cdot,s)\right\}_{n \in \Nat}$ to the domain of $(d,\infty)$.}; and \eqref{Eq:bconv:1b} and \eqref{Eq:bconv:1d} follow from the fact (see, e.g., \cite[pg.~228]{rockafellar}) that for a proper convex function $f$ on $\Real$, $z_1 < x < z_2$ implies:
\begin{eqnarray*}
f^{\prime+}(z_1) \leq f^{\prime-}(x) \leq f^{\prime+}(x) \leq f^{\prime-}(z_2)~.
\end{eqnarray*}
$\liminf\limits_{n \rightarrow \infty}\tilde{g}_{n}^{\prime+}\bigl(x_1,s \bigr) \geq -\tilde{c}_k(s)$ implies that for every sequence $\left\{n_j\right\}_{j=1,2,\ldots}$, we have:
\begin{eqnarray*}
\lim_{j \rightarrow \infty}\tilde{g}_{n_j}^{\prime+}\bigl(x_1,s \bigr) \geq -\tilde{c}_k(s)~,
\end{eqnarray*}
and, in turn:
\begin{eqnarray*}
\lim_{j \rightarrow \infty}\hat{b}_{n_j,k}(s) \leq x_1~.
\end{eqnarray*}
Therefore, $\limsup\limits_{n \rightarrow \infty}\hat{b}_{n,k}(s) \leq x_1$, which is a contradiction. We conclude:
\begin{eqnarray}\label{Eq:bconv:conc2}
\limsup\limits_{n \rightarrow \infty} \hat{b}_{n,k}(s) \leq \hat{b}_{\infty,k}(s)~.
\end{eqnarray}
Equations \eqref{Eq:bconv:conc1} and \eqref{Eq:bconv:conc2} imply \eqref{Eq:inf_proof:crits}.

We are now ready to prove parts (e) and (f) of Theorem \ref{Th:one:infinite}. Define
\begin{eqnarray*}
z_{\infty}^*(x,s) := \left\{
\begin{array}{ll}
   \tilde{z}_{k-1}(s) , & \mbox{if } ~b_{\infty,k}(s)-\tilde{z}_{k-1}(s) < x \leq b_{\infty,k-1}(s)-\tilde{z}_{k-1}(s)~, \\
   &~~~~~~~~~~~~~~~~~~~~~~~~~~~~~~~~~~~~~~~~~~~k \in \{0,1,\ldots,K\} \\
   b_{\infty,k}(s)-x , & \mbox{if } ~b_{\infty,k}(s)-\tilde{z}_k(s) < x \leq b_{\infty,k}(s)-\tilde{z}_{k-1}(s)~,\\
   &~~~~~~~~~~~~~~~~~~~~~~~~~~~~~~~~~~~~~~k \in \{0,1,\ldots,K-1\} \\
   b_{\infty,K}(s)-x , & \mbox{if } ~b_{\infty,K}(s)-\tilde{z}_{\max}(s) < x \leq b_{\infty,K}(s)-\tilde{z}_{K-1}(s) \\
   \tilde{z}_{\max}(s) , & \mbox{if } ~0 \leq x \leq b_{\infty,K}(s)-\tilde{z}_{\max}(s)
   \end{array} \right.
\end{eqnarray*}
%
Clearly, $\lim\limits_{n \rightarrow \infty} b_{n,k}(s) = b_{\infty,k}(s)$
implies $\lim\limits_{n \rightarrow \infty} z_n^*(x,s) =
z_{\infty}^*(x,s),~\forall x \in \Real+,~\forall s \in {\cal{S}}$.
Furthermore, $\tilde{g}_n(y,s) \rightarrow \tilde{g}_{\infty}(y,s)$ and
$z_n^*(x,s) \rightarrow z_{\infty}^*(x,s)$ as $n \rightarrow \infty$ imply:
\begin{eqnarray} \label{Eq:convergence}
\lim\limits_{n \rightarrow
\infty}\tilde{g}_n\bigl(x+z_n^*(x,s)\bigr)=\tilde{g}_{\infty}\bigl(x+z_{\infty}^*(x,s)\bigr),~\forall
x \in \Real+,~\forall s \in {\cal{S}}~.
\end{eqnarray}
So for all $x \in \Real+$
and $s \in {\cal{S}}$, we have:
\begin{eqnarray}
V_{\infty}(x,s) &=& \lim_{n \rightarrow \infty} V_n(x,s) \nonumber \\
&=& \lim_{n \rightarrow \infty}
\min\limits_{\bigl\{\max(0,d-x) \leq z \leq \tilde{z}_{\max}(s)\bigr\}}
\Bigl\{c(z,s) + \tilde{g}_n(x+z,s) \Bigr\} \nonumber \\
&=& \lim_{n \rightarrow \infty}
\Bigl\{c\bigl(z_n^*(x,s),s\bigr) + \tilde{g}_n\bigl(x+z_n^*(x,s),s\bigr) \Bigr\} \label{Eq:finite_used} \\
&=&
c\bigl(z_{\infty}^*(x,s),s\bigr) + \tilde{g}_{\infty}\bigl(x+z_{\infty}^*(x,s),s\bigr)  \label{Eq:convergence_used} \\
&=& \min\limits_{\bigl\{\max(0,d-x) \leq z \leq \tilde{z}_{\max}(s)\bigr\}}
\Bigl\{
c(z,s) + \tilde{g}_{\infty}(x+z,s)
\Bigr\} \label{Eq:back_used} \\
&=&
\min\limits_{\bigl\{\max(0,d-x) \leq z \leq \tilde{z}_{\max}(s)\bigr\}}
\left\{
\begin{array}{l}
c(z,s) + h(x+z-d) \\
+ \alpha \cdot \Expectation \bigl[V_{\infty}(x+z-d,S^{\prime}) \bigm| S =s\bigr]
\end{array}
\right\}~. \nonumber
\end{eqnarray}
Equation \eqref{Eq:finite_used} follows from Theorem
\ref{Th:str:finite}, and \eqref{Eq:convergence_used} follows from \eqref{Eq:convergence} and the continuity of $c(\cdot,s)$. Equation \eqref{Eq:back_used} follows from the same line of analysis as part (ii) of the induction step in the proof of Theorem \ref{Th:pwl:fin}, with $\tilde{g}_{\infty}(\cdot,s)$, $b_{\infty,k}(s)$, and $z_{\infty}^*(\cdot,s)$ replacing $\tilde{g}_{m}(\cdot,s)$, $b_{m,k}(s)$, and $z_{m}^*(\cdot,s)$, respectively.
Thus, $V_{\infty}(\cdot,\cdot)$, the limit of the finite horizon value functions, satisfies the $\alpha$-DCOE
(\ref{Eq:pwl:inf_functional}) and is also equal to the infinite horizon discounted expected cost-to-go resulting from the stationary policy
\noindent $\boldsymbol{{\pi}_{\infty}^*}:=\left(z_{\infty}^*,z_{\infty}^*,\ldots\right)$. We conclude $\boldsymbol{{\pi}_{\infty}^*}$, the natural extension of the finite horizon optimal policy, is optimal for the infinite horizon problem (see, for example, \cite[Propositions 9.12 and 9.16]{shreve}). \qedsymbol

%% file: appendix_two_infinite_discounted_proof.tex
We follow the same line of analysis as the proof of
Theorem
\ref{Th:one:infinite}. 
Let $\textbf{x} \in \Real_+^2$ and $\textbf{s} \in {\cal{S}}$ be
arbitrary. First, we show inductively that $V_1(\textbf{x},\textbf{s}) \leq V_2(\textbf{x},\textbf{s}) \leq
\ldots \leq V_n(\textbf{x},\textbf{s}) \leq V_{n+1}(\textbf{x},\textbf{s}) \leq \ldots$.

\noindent \underline{Base Case}: $n=1$
\begin{eqnarray*}
V_1(\textbf{x},\textbf{s}) &=&
\min_{\textbf{z} \in {\cal A}^{\textbf{d}}(\textbf{x},\textbf{s})}
\left\{
\textbf{c}_{\textbf{s}}^{\transpose} \textbf{z} +{h}(\textbf{x}+\textbf{z}-\textbf{d})
\right\} \\
&\leq&
\min_{\textbf{z} \in {\cal A}^{\textbf{d}}(\textbf{x},\textbf{s})}
\left\{
\begin{array}{l}
\textbf{c}_{\textbf{s}}^{\transpose} \textbf{z} +{h}(\textbf{x}+\textbf{z}-\textbf{d}) \\
 +\alpha \cdot \Expectation \bigl[V_{1}(\textbf{x}+\textbf{z}-\textbf{d},\textbf{S}_{1})\bigm| \textbf{S}_2=\textbf{s}\bigr]
\end{array}
\right\} \\
&=&
V_2(\textbf{x},\textbf{s})~,
\end{eqnarray*}
\noindent where the inequality follows from $V_1(\textbf{x},\textbf{s}) \geq 0,~\forall \textbf{x}, \forall \textbf{s}$.
\medskip

\noindent \underline{Induction Step}: 
 Assume $V_n(\textbf{x},\textbf{s}) \leq V_{n+1}(\textbf{x},\textbf{s})$ for $n=1,2,\ldots,m-1$. We show it is true for $n=m$:
\begin{eqnarray*}
V_m(\textbf{x},\textbf{s})
&=&
\min_{\textbf{z} \in {\cal A}^{\textbf{d}}(\textbf{x},\textbf{s})}
\left\{
\begin{array}{l}
\textbf{c}_{\textbf{s}}^{\transpose} \textbf{z} +{h}(\textbf{x}+\textbf{z}-\textbf{d}) \\
 +\alpha \cdot \Expectation \bigl[V_{m-1}(\textbf{x}+\textbf{z}-\textbf{d},\textbf{S}_{m-1})\bigm| \textbf{S}_m=\textbf{s}\bigr]
\end{array}
\right\} \\
&\leq&
\min_{\textbf{z} \in {\cal A}^{\textbf{d}}(\textbf{x},\textbf{s})}
\left\{
\begin{array}{l}
\textbf{c}_{\textbf{s}}^{\transpose} \textbf{z} +{h}(\textbf{x}+\textbf{z}-\textbf{d}) \\
 +\alpha \cdot \Expectation \bigl[V_{m}(\textbf{x}+\textbf{z}-\textbf{d},\textbf{S}_{m})\bigm| \textbf{S}_{m+1}=\textbf{s}\bigr]
\end{array}
\right\} \\
&=&
V_{m+1}(\textbf{x},\textbf{s})~,
\end{eqnarray*}
\noindent where the inequality follows from the induction hypothesis and the homogeneity of the Markov process representing the channel condition. So, for every $\textbf{x} \in \Real_+^2$ and $\textbf{s} \in {\cal{S}}$, $\left\{V_n(\textbf{x},\textbf{s})\right\}_{n=1,2,\ldots}$ is a nondecreasing sequence.

Next, consider a policy ${\boldsymbol{\pi}}^{\textbf{d}}$ transmitting $d^1$ packets to user 1 and $d^2$ packets to user 2 in every slot, regardless of channel condition. Define:
\begin{eqnarray} \label{Eq:ctmax}
\textbf{c}_{\max}^{\transpose}:=\left(c_{\max}^1,c_{\max}^2\right)^{\transpose},\hbox{ where }c_{\max}^i:=\sup\limits_{s^i \in {\cal S}^i }\left\{{c}_{s^i} \right\}<\infty~.
\end{eqnarray}
Then we have:
\begin{eqnarray*}
V_n(\textbf{x},\textbf{s}) \leq
V_n^{{\boldsymbol{\pi}}^{\textbf{d}}}(\textbf{x},\textbf{s}) \leq \Bigl(\textbf{c}_{\max}^{\transpose}\textbf{d} + h(\textbf{x}) \Bigr) \frac{1-\alpha^n}{1-\alpha} \leq \Bigl(\textbf{c}_{\max}^{\transpose}\textbf{d} + h(\textbf{x}) \Bigr) \frac{1}{1-\alpha} 
< \infty~,
\end{eqnarray*}
\noindent so $\left\{V_n(\textbf{x},\textbf{s})\right\}_{n=1,2,\ldots}$ is a bounded nondecreasing
sequence, implying 
$\lim_{n \rightarrow \infty}V_n(\textbf{x},\textbf{s})$ exists
and is finite, $\forall \textbf{x} \in \Real_+^2, \forall \textbf{s} \in {\cal{S}}$.


Next, recall from Theorem \ref{Th:two_users:properties} that 
$V_n(\textbf{x},\textbf{s})$ is convex and supermodular in $\textbf{x}$,
for all $n$ and all $\textbf{s}$. Define $V_{\infty}(\textbf{x},\textbf{s}):=\lim_{n \rightarrow \infty} V_n(\textbf{x},\textbf{s})$. Let $\textbf{s} \in
{\cal{S}}$ be arbitrary, but fixed.
$V_{\infty}(\textbf{x},\textbf{s})=\sup_{n \in \Nat} V_n(\textbf{x},\textbf{s})$, so $V_{\infty}(\textbf{x},\textbf{s})$ is convex in $\textbf{x}$
as it is the pointwise supremum of the convex functions
$\left\{V_n(\textbf{x},\textbf{s})\right\}_{n=1,2,\ldots}$. Furthermore, the pointwise limit of supermodular functions is supermodular (see, e.g., \cite[Lemma 2.6.1]{topkis}), so $V_{\infty}(\textbf{x},\textbf{s})$ is also supermodular in $\textbf{x}$.

Define $G_{\infty}: [d^1,\infty) \times [d^2,\infty) \times {\cal{S}} \rightarrow \Real_+$ by
\begin{eqnarray} \label{Eq:app2:mct}
G_{\infty}(\textbf{y},\textbf{s})&:=& \textbf{c}_{\textbf{s}}^{\transpose} \textbf{y} +{h}(\textbf{y}-\textbf{d})+\alpha \cdot \Expectation \bigl[V_{\infty}(\textbf{y}-\textbf{d},\textbf{S}^{\prime})\bigm| \textbf{S}=\textbf{s}\bigr] \nonumber \\
&=& \textbf{c}_{\textbf{s}}^{\transpose} \textbf{y} +{h}(\textbf{y}-\textbf{d})+\alpha \cdot \Expectation \bigl[\lim_{n \rightarrow \infty}V_{n}(\textbf{y}-\textbf{d},\textbf{S}^{\prime})\bigm| \textbf{S}=\textbf{s}\bigr] \nonumber \\
&=& \textbf{c}_{\textbf{s}}^{\transpose} \textbf{y} +{h}(\textbf{y}-\textbf{d})+\alpha \cdot \lim_{n \rightarrow \infty} \Expectation \bigl[V_{n}(\textbf{y}-\textbf{d},\textbf{S}^{\prime})\bigm| \textbf{S}=\textbf{s}\bigr] \\
&=&\lim_{n \rightarrow \infty} G_n(\textbf{y},\textbf{s})~,
\end{eqnarray}
\noindent where (\ref{Eq:app2:mct}) follows from the homogeneity of the Markov process representing the channel condition and the Monotone Convergence Theorem. 
Furthermore, for each $\textbf{s} \in {\cal{S}}$, $G_{\infty}(\textbf{y},\textbf{s})$ is convex and supermodular in $\textbf{y}$ as it is the sum of
an affine function of $\textbf{y}$,
a convex separable function of $\textbf{y}-\textbf{d}$ and a weighted sum of 
the convex supermodular functions $V_{\infty}\left(\textbf{y}-\textbf{d},\textbf{s}^{\prime}\right)$. Additionally, $\lim\limits_{||\textbf{y}||
\rightarrow \infty}G_{\infty}(y,s) 
\geq \lim\limits_{||\textbf{y}||
\rightarrow \infty}\textbf{c}_{\textbf{s}}^{\transpose} \textbf{y} = \infty$. Thus, for every $\textbf{s}$, at least
one finite vector achieves the global minimum of
$G_{\infty}(\textbf{y},\textbf{s})$; ${\mathcal B}_{\infty}(\textbf{s})$ is a nonempty closed convex set; and $\textbf{b}_{\infty}(\textbf{s})$, $f_{\infty}^1(\cdot,\textbf{s})$, and $f_{\infty}^2(\cdot,\textbf{s})$ are well-defined. The structure of the optimal policy outlined in (b) then follows from the same line of analysis used to prove the the structure of the optimal policy in the induction step of Theorem \ref{Th:two_users:structure}.

Moreover, since for a fixed $\textbf{s} \in {\cal S}$ and $x^2 \in [d^2,\infty)$,
\begin{align*}
f_n^1(x^2,\textbf{s}):= \inf \left\{\argmin_{y^1 \in
[d^1,\infty)}
\biggl\{G_n\left(y^1,x^2,s^1,s^2\right)\biggr\}
\right\}
=
\inf\bigl\{b^1 \bigm| G_{n}^{\prime+}\left(b^1,x^2,s^1,s^2\right) \geq 0 \bigr\}~,
\end{align*}
the convergence of $f_n^1(x^2,\textbf{s})$ to $f_{\infty}^1(x^2,\textbf{s})$ follows from the same argument used to show
\eqref{Eq:inf_proof:crits}. The convergence of $f_n^2(x^1,\textbf{s})$ to $f_{\infty}^2(x^1,\textbf{s})$ follows from a symmetric argument.

For all $\textbf{s} \in {\cal S}$ and $x^1 \in [d^1,\infty)$, define:
\begin{align*}
&\Psi_n(x^1,\textbf{s}):=\min\limits_{x^2 \in [d^2,\infty)}\left\{G_n(x^1,x^2,s^1,s^2)\right\}=G_n\bigl(x^1,f_n^2(x^1,\textbf{s}),s^1,s^2\bigr)~,~\forall n\in \Nat, \\ 
&~~\hbox{    and} \\
&\Psi_{\infty}(x^1,\textbf{s}):=\min\limits_{x^2 \in [d^2,\infty)}\left\{G_{\infty}(x^1,x^2,s^1,s^2)\right\}=G_{\infty}\bigl(x^1,f_{\infty}^2(x^1,\textbf{s}),s^1,s^2\bigr)~. 
\end{align*}
For fixed but arbitrary $x^1$ and $\textbf{s}$, $f_n^2(x^1,\textbf{s})$ converges to $f_{\infty}^2(x^1,\textbf{s})$, and, by Dini's Theorem, $G_n(x^1,\cdot,\textbf{s})$ converges to $G_{\infty}(x^1,\cdot,\textbf{s})$ uniformly on a compact interval containing $f_{\infty}^2(x^1,\textbf{s})$. Thus, $\Psi_n(x^1,\textbf{s})$ converges pointwise to $\Psi_{\infty}(x^1,\textbf{s})$. Moreover, for every $\textbf{s}$, $\{\Psi_n(x^1,\textbf{s})\}_{n \in \Nat}$ and $\Psi_{\infty}(x^1,\textbf{s})$ are all convex in $x^1$ with the limit as $x^1$ approaches infinity equal to infinity. Therefore, by the same argument used to show \eqref{Eq:inf_proof:crits}, $b_n^1(\textbf{s})$ converges pointwise to $b_{\infty}^1(\textbf{s})$.

For all $\textbf{s} \in {\cal S}$ and $x^2 \in [d^2,\infty)$, define:
\begin{align*}
&\tilde{\Psi}_n(x^2\textbf{s}):=G_n\bigl(b_n^1(\textbf{s}),x^2,s^1,s^2\bigr)~,~\forall n\in \Nat, \\ 
&~~\hbox{    and} \\
&\tilde{\Psi}_n(x^2\textbf{s}):=G_{\infty}\bigl(b_{\infty}^1(\textbf{s}),x^2,s^1,s^2\bigr)~. 
\end{align*}
For fixed but arbitrary $x^2$ and $\textbf{s}$, $b_n^1(\textbf{s})$ converges to $b_{\infty}^1(\textbf{s})$, and, by Dini's Theorem, $G_n(\cdot,x^2,\textbf{s})$ converges to $G_{\infty}(\cdot,x^2,\textbf{s})$ uniformly on a compact interval around $b_{\infty}^1(\textbf{s})$. Thus, $\tilde{\Psi}_n(x^2,\textbf{s})$ converges pointwise to $\tilde{\Psi}_{\infty}(x^2,\textbf{s})$. Moreover, for every $\textbf{s}$, $\{\tilde{\Psi}_n(x^2,\textbf{s})\}_{n \in \Nat}$ and $\tilde{\Psi}_{\infty}(x^2,\textbf{s})$ are all convex in $x^2$ with the limit as $x^2$ approaches infinity equal to infinity. Therefore, by the same argument used to show \eqref{Eq:inf_proof:crits}, $b_n^2(\textbf{s})$ converges pointwise to $b_{\infty}^2(\textbf{s})$, and we conclude $\textbf{b}_{\infty}(\textbf{s})=\lim\limits_{n \rightarrow \infty}\textbf{b}_n(\textbf{s})$.
\qedsymbol

%% file: appendix_two_inf_average_proof.tex
Substituting \eqref{Eq:mdef} and \eqref{Eq:wdef} into the $\alpha$-DCOE \eqref{Eq:adcoe} and rearranging yields:
\begin{align} \label{Eq:adcoere}
&(1-\alpha) \cdot m_{\infty,\alpha} + w_{\infty,\alpha}(\textbf{x},\textbf{s}) \nonumber \\ 
&=\min_{\textbf{y} \in \tilde{\cal A}^{\textbf{d}}(\textbf{x},\textbf{s})}
\Bigl\{
\textbf{c}_{\textbf{s}}^{\transpose}[\textbf{y}-\textbf{x}]
+\textbf{h}\left(\textbf{y}-\textbf{d}\right) 
 +\alpha \cdot \Expectation \bigl[w_{\infty,\alpha}(\textbf{y}-\textbf{d},\textbf{S}^{\prime})\bigm| \textbf{S}=\textbf{s}\bigr]
\Bigr\},~\forall \textbf{x} \in \Real_+^2,\forall
\textbf{s} \in {\cal{S}}~.
\end{align}
The main idea of the vanishing discount approach is to take the limit as $\alpha$ goes to 1, and show that \eqref{Eq:adcoere} converges to the ACOE \eqref{Eq:inf_acoe}.

We start by presenting five conditions from the literature on the vanishing discount approach.
\begin{conditiong}
$\rho:=\inf\limits_{\boldsymbol{\pi} \in \boldsymbol{\Pi}}\inf\limits_{\substack{\textbf{x} \in {\Real_+^2}\\{\textbf{s} \in {\cal S}}}} \left\{ \limsup\limits_{N \rightarrow \infty} \frac{1}{N} V^{\boldsymbol{\pi}}_{N,1}(\textbf{x},\textbf{s}) \right\}<\infty.$
\end{conditiong}

\begin{conditionw}
\begin{itemize}
\item[(i)] The state space ${\Real_+^2} \times {\cal S}$ is a locally compact space with countable base.
\item[(ii)] The action space $\tilde{\cal A}^{\textbf{d}}(\textbf{x},\textbf{s})$ is a nonempty compact subset of the state space ${\Real_+^2} \times {\cal S}$, and the multifunction $\phi:(\textbf{x},\textbf{s}) \mapsto \tilde{\cal A}^{\textbf{d}}(\textbf{x},\textbf{s})$ is upper semicontinuous; that is, $\phi^{-1}(F)$ is closed in ${\Real_+^2} \times {\cal S}$ for every closed set $F \subset {\Real_+^2}$.
    \item[(iii)] The transition law is weakly continuous (see, e.g., \cite[Appendix C]{lerma2}).
    \item[(iv)] The one-stage cost $c(\textbf{z},\textbf{s})+h(\textbf{x}+\textbf{z}-\textbf{d})$ is lower semicontinuous and nonnegative.
\end{itemize}
\end{conditionw}

\begin{conditionb}
$\sup\limits_{\alpha < 1} w_{\infty,\alpha}(\textbf{x},\textbf{s}) < \infty$ for all $\textbf{x} \in {\Real_+^2}$ and $\textbf{s} \in {\cal S}$.
\end{conditionb}

\begin{conditionb2} There is a measurable function $\bar{\kappa}: {\Real_+^2} \times {\cal S} \rightarrow \Real_+$ such that $\bar{\kappa} \geq w_{\infty,\alpha}$ for all $\alpha \in [0,1)$, and:
\begin{eqnarray} \label{Eq:b2cond}
\Expectation\bigl[\bar{\kappa}(\textbf{y}-\textbf{d},\textbf{S}^{\prime}) \bigm| \textbf{S}=\textbf{s}\bigr] < \infty,~\forall (\textbf{x},\textbf{s}) \in {\Real_+^2} \times {\cal S},~\forall \textbf{y} \in \tilde{\cal A}^{\textbf{d}}(\textbf{x},\textbf{s})~.
\end{eqnarray}
\end{conditionb2}

\begin{conditione}
For every
increasing sequence of discount factors $\{\alpha(l)\}_{l=1,2,\ldots}$ approaching 1, the sequence $\left\{ w_{\infty,\alpha(l)}\right\}_{l=1,2,\ldots}$ is equicontinuous.
\end{conditione}
We show below that our model satisfies these five conditions, but first we show how they lead to Theorem \ref{Th:two:infa}. Parts (b), (c), and (e) of Theorem \ref{Th:two:infa} follow directly from the following theorem due to Sch\"{a}l \cite[Theorem 3.8]{schal2} and adapted to our notation.
\begin{theorem}[Sch\"{a}l, 1993] \label{Th:sprop}
Suppose conditions (G), (W), and (B) hold. Then the minimum average cost $\rho^*= \inf\limits_{\boldsymbol{\pi} \in \boldsymbol{\Pi}}\inf\limits_{\substack{\textbf{x} \in {\Real_+^2}\\{\textbf{s} \in {\cal S}}}} \left\{ \limsup\limits_{N \rightarrow \infty} \frac{1}{N} V^{\boldsymbol{\pi}}_{N,1}(\textbf{x},\textbf{s}) \right\}=\lim\limits_{\alpha \nearrow 1} (1-\alpha) \cdot m_{\infty,\alpha}$.
Moreover, there exists an optimal selector $\textbf{y}_{\infty,1}^*(\cdot,\cdot)$ such that:
\begin{align} \label{Eq:pwl:inf_acoe_half}
\rho^* + w_{\infty,1}(\textbf{x},\textbf{s})
&\geq \min_{\textbf{y} \in \tilde{\cal A}^{\textbf{d}}(\textbf{x},\textbf{s})}
\left\{
\begin{array}{l}
\textbf{c}_{\textbf{s}}^{\transpose}[\textbf{y}-\textbf{x}]
+\textbf{h}\left(\textbf{y}-\textbf{d}\right) \\
 +\Expectation \bigl[w_{\infty,1}(\textbf{y}-\textbf{d},\textbf{S}^{\prime})\bigm| \textbf{S}=\textbf{s}\bigr]
 \end{array}
\right\} \\
&= \textbf{c}_{\textbf{s}}^{\transpose}\Bigl[\textbf{y}_{\infty,1}^*(\textbf{x},\textbf{s})-\textbf{x}\Bigr] + \textbf{h}\Bigl(\textbf{y}_{\infty,1}^*(\textbf{x},\textbf{s})-\textbf{d}\Bigr) \nonumber \\
&~~+\Expectation \left[w_{\infty,1}\Bigl(\textbf{y}_{\infty,1}^*(\textbf{x},\textbf{s})-\textbf{d},\textbf{S}^{\prime}\Bigr) \middle| \textbf{S} =\textbf{s}\right],~\forall \textbf{x} \in \Real_+^2,~\forall \textbf{s} \in {\cal S}, \nonumber
\end{align}
where for every $(\textbf{x},\textbf{s}) \in \Real_+^2 \times {\cal S}$ and any increasing sequence of discount factors $\{\alpha(l)\}_{l=1,2,\ldots}$ approaching 1,
 \begin{eqnarray} \label{Eq:wall}
 w_{\infty,1}(\textbf{x},\textbf{s}):= \liminf\limits_{l \rightarrow \infty}w_{\infty,\alpha(l)}(\textbf{x},\textbf{s})~.
 \end{eqnarray}
Furthermore, for every $(\textbf{x},\textbf{s}) \in \Real_+^2 \times {\cal S}$ and any increasing sequence of discount factors $\{\alpha(l)\}_{l=1,2,\ldots}$ approaching 1, there exists a subsequence $\{\alpha(l_i)\}_{i=1,2,\ldots}$ approaching 1 and a sequence $\{\textbf{x}(i)\}_{i=1,2,\ldots}$ approaching $\textbf{x}$ such that:
\begin{eqnarray*}
    \textbf{y}_{\infty,1}^*(\textbf{x},\textbf{s})&=\lim\limits_{i \rightarrow \infty} \textbf{y}_{\infty,\alpha(l_i)}^*(\textbf{x}(i),\textbf{s})~.
\end{eqnarray*}
\end{theorem}

To get the opposite inequality from \eqref{Eq:pwl:inf_acoe_half}, we use a method from \cite{fernandez} and \cite[Theorem 4.1]{montes} (which is presented in \cite[Section 5.5]{lerma2}). Namely, for every $\textbf{x} \in \Real_+^2$, $\textbf{s}\in {\cal S}$, $\textbf{y} \in \tilde{\cal A}^{\textbf{d}}(\textbf{x},\textbf{s})$, and $\alpha(l)$ from \eqref{Eq:wall}, \eqref{Eq:adcoere} implies:
\begin{align} \label{Eq:adcoereind}
&(1-\alpha(l)) \cdot m_{\infty,\alpha(l)} + w_{\infty,\alpha(l)}(\textbf{x},\textbf{s}) \nonumber \\
&\leq
\textbf{c}_{\textbf{s}}^{\transpose}[\textbf{y}-\textbf{x}]
+\textbf{h}\left(\textbf{y}-\textbf{d}\right) 
 +\alpha(l) \cdot \Expectation \bigl[w_{\infty,\alpha(l)}(\textbf{y}-\textbf{d},\textbf{S}^{\prime})\bigm| \textbf{S}=\textbf{s}\bigr]~.
\end{align}
Furthermore, in combination with Conditions (B) and (E), the Arzel\'{a}-Ascoli Theorem implies there exists a subsequence $\{\alpha(l_i)\}_{i=1,2,\ldots}$ of $\{\alpha(l)\}_{l=1,2,\ldots}$ such that:
\begin{eqnarray} \label{Eq:wsubascoli}
w_{\infty,1}(\textbf{x},\textbf{s})= \lim\limits_{i \rightarrow \infty}w_{\infty,\alpha(l_i)}(\textbf{x},\textbf{s})~,
\forall \textbf{x} \in \Real_+^2,~\forall \textbf{s} \in {\cal S}~.
\end{eqnarray}
Then, taking the limit of \eqref{Eq:adcoereind} as $\alpha$ goes to 1 along the sequence $\{\alpha(l_i)\}_{i=1,2,\ldots}$, \eqref{Eq:rho_conv}, \eqref{Eq:wsubascoli}, Condition (B2), and the Lebesgue Dominated Convergence Theorem imply:
\begin{align*}
\rho^* + w_{\infty,1}(\textbf{x},\textbf{s})
\leq
&\textbf{c}_{\textbf{s}}^{\transpose}[\textbf{y}-\textbf{x}]
+\textbf{h}\left(\textbf{y}-\textbf{d}\right) 
 +\Expectation \bigl[w_{\infty,1}(\textbf{y}-\textbf{d},\textbf{S}^{\prime})\bigm| \textbf{S}=\textbf{s}\bigr]~, \\
 &~~~~~~~~~~~~~~~\forall \textbf{x} \in \Real_+^2,~\forall \textbf{s} \in {\cal S},~\forall \textbf{y} \in \tilde{\cal A}^{\textbf{d}}(\textbf{x},\textbf{s})~,
\end{align*}
which implies:
\begin{align} \label{Eq:pwl:inf_acoe_other_half}
\rho^* + w_{\infty,1}(\textbf{x},\textbf{s})
&\leq \min_{\textbf{y} \in \tilde{\cal A}^{\textbf{d}}(\textbf{x},\textbf{s})}
\left\{
\begin{array}{l}
\textbf{c}_{\textbf{s}}^{\transpose}[\textbf{y}-\textbf{x}]
+\textbf{h}\left(\textbf{y}-\textbf{d}\right) \\
 +\Expectation \bigl[w_{\infty,1}(\textbf{y}-\textbf{d},\textbf{S}^{\prime})\bigm| \textbf{S}=\textbf{s}\bigr]
 \end{array}
\right\}~,~\forall \textbf{x} \in \Real_+^2,~\forall \textbf{s} \in {\cal S}~.
\end{align}
Equations \eqref{Eq:pwl:inf_acoe_half} and \eqref{Eq:pwl:inf_acoe_other_half} yield the ACOE \eqref{Eq:inf_acoe}. Moreover, from \eqref{Eq:wsubascoli}
and the fact that convexity and supermodularity are preserved under pointwise limits, we conclude that for every $\textbf{s} \in {\cal S}$, $w_{\infty,1}(\textbf{x},\textbf{s})$ is convex and supermodular in $\textbf{x}$. Then, by the same argument as the one used in Theorems \ref{Th:two_users:structure} and \ref{Th:two:infd}, there exists an optimal stationary policy with the same structure as statement (b) in Theorem \ref{Th:two:infd} that minimizes the right hand side of the ACOE. 

Thus, it just remains to show our model satisfies the five conditions. We proceed in order, beginning with Condition (G).
Consider again the policy ${\boldsymbol{\pi}}^{\textbf{d}}$ transmitting $d^1$ packets to user 1 and $d^2$ packets to user 2 in every slot, regardless of channel condition. Let the initial vector of buffer levels $\textbf{x}_0 = (0,0)$, and let the initial vector of channel conditions $\textbf{s}_0$ be arbitrary. Then we have:
\begin{eqnarray*}
\rho:=\inf\limits_{\boldsymbol{\pi} \in \boldsymbol{\Pi}}\inf\limits_{\substack{x \in {\Real_+^2}\\{s \in {\cal S}}}} \left\{ \limsup\limits_{N \rightarrow \infty} \frac{1}{N} V^{\boldsymbol{\pi}}_{N,1}(\textbf{x},\textbf{s}) \right\} \leq \limsup\limits_{N \rightarrow \infty} \frac{1}{N} V^{\boldsymbol{\pi}^{\textbf{d}}}_{N,1}(\textbf{x}_0,\textbf{s}_0)\leq \textbf{c}_{\max}^{\transpose}\textbf{d}<\infty~,  
\end{eqnarray*}
where $\textbf{c}_{\max}^{\transpose}$ is defined in \eqref{Eq:ctmax}.\footnote{For the proof of Theorem \ref{Th:pwl:infa}, we use $\tilde{c}_{\max}$ defined in \eqref{Eq:ctildemaxdef} instead.}

The only nontrivial statement in Condition (W) is the weak continuity of the transition law. Let $\{\textbf{x}_i\}_{i=1,2,\ldots}, \{\textbf{s}_i\}_{i=1,2,\ldots}$, and $\{\textbf{y}_i\}_{i=1,2,\ldots}$ be sequences approaching $\textbf{x}, \textbf{s}$, and $\textbf{y}$, respectively, and let $\Gamma$ be a bounded, continuous function on $\Real_+^2 \times {\cal S}$. We need to show:
\begin{eqnarray*} 
\lim\limits_{i \rightarrow \infty}\Expectation\left[\Gamma\left(\textbf{X}^{\prime},\textbf{S}^{\prime}\right) \middle|\textbf{X}=\textbf{x}_i,\textbf{S}=\textbf{s}_i,\textbf{Y}=\textbf{y}_i \right] = \Expectation\left[\Gamma\left(\textbf{X}^{\prime},\textbf{S}^{\prime}\right) \middle|\textbf{X}=\textbf{x},\textbf{S}=\textbf{s},\textbf{Y}=\textbf{y} \right]~.
\end{eqnarray*}
This is true, as 
\begin{align*} 
&\lim\limits_{i \rightarrow \infty}\Expectation\left[\Gamma\left(\textbf{X}^{\prime},\textbf{S}^{\prime}\right) \middle|\textbf{X}=\textbf{x}_i,\textbf{S} =\textbf{s}_i,\textbf{Y}=\textbf{y}_i \right] \\
&= \lim\limits_{i \rightarrow \infty}\sum\limits_{\textbf{s}^{\prime} \in {\cal S}} \hbox{Pr}\left(\textbf{S}^{\prime}=\textbf{s}^{\prime}\mid \textbf{S}=\textbf{s}_i\right) \cdot \Gamma\left(\textbf{y}_i-\textbf{d},\textbf{s}^{\prime}\right)  \\
&= \sum\limits_{\textbf{s}^{\prime} \in {\cal S}} \left[\lim\limits_{i \rightarrow \infty}\hbox{Pr}\left(\textbf{S}^{\prime}=\textbf{s}^{\prime}\mid \textbf{S}=\textbf{s}_i\right)\right] \cdot \left[\lim\limits_{i \rightarrow \infty}\Gamma\left(\textbf{y}_i-\textbf{d},\textbf{s}^{\prime}\right)\right]  \\
&= \sum\limits_{\textbf{s}^{\prime} \in {\cal S}} \hbox{Pr}\left(\textbf{S}^{\prime}=\textbf{s}^{\prime}\mid \textbf{S}=\textbf{s}\right) \cdot \Gamma\left(\textbf{y}-\textbf{d},\textbf{s}^{\prime}\right) \\
&= \Expectation\left[\Gamma\left(\textbf{X}^{\prime},\textbf{S}^{\prime}\right) \middle|\textbf{X}=\textbf{x},\textbf{S}=\textbf{s},\textbf{Y}=\textbf{y} \right]~.
\end{align*}

Next, we prove Conditions (B) and (B2). Let $\alpha \in [0,1)$ be arbitrary. For every $\textbf{s} \in {\cal S}$, $V_{\infty,\alpha}(\textbf{x},\textbf{s})$ is convex in $\textbf{x}$, and
\begin{eqnarray*}
\lim\limits_{||\textbf{x}||\rightarrow \infty} V_{\infty,\alpha}(\textbf{x},\textbf{s}) \geq \lim\limits_{||\textbf{x}||\rightarrow \infty} h(\textbf{x}-\textbf{d}) =\infty~,
 \end{eqnarray*}
 so there exists an $\textbf{x}^*(\textbf{s}) \in {\Real_+^2}$ such that:
 \begin{eqnarray*}
 \min\limits_{\textbf{x} \in {\Real_+^2}}\left\{V_{\infty,\alpha}(\textbf{x},\textbf{s})\right\}=V_{\infty,\alpha}\bigl(\textbf{x}^*(\textbf{s}),\textbf{s}\bigr)~.
 \end{eqnarray*}
Define:
\begin{eqnarray*}
\textbf{s}^*:=\argmin\limits_{\textbf{s} \in {\cal S}}\left\{V_{\infty,\alpha}(\textbf{x}^*(\textbf{s}),\textbf{s}) \right\}~,
\end{eqnarray*}
so that
\begin{eqnarray*}
m_{\infty,\alpha} = V_{\infty,\alpha}(\textbf{x}^*(\textbf{s}^*),\textbf{s}^*)~.
\end{eqnarray*}
Define also the stationary policy $\breve{\boldsymbol{\pi}}=\left(\breve{\textbf{y}},\breve{\textbf{y}},\ldots \right)$, where:
$\breve{\textbf{y}}(\textbf{x},\textbf{s}):=\Bigl(\breve{y}^1(x^1,s^1),\breve{y}^2(x^2,s^2)\Bigr)$, and for $m\in\{1,2\}$,
\begin{eqnarray*}
\breve{y}^m(x^m,s^m):=\left\{
\begin{array}{ll}
   x^m , & \mbox{if } ~ x^{m^*}(\textbf{s}^*)+d^m \leq x^m \\
   x^{m^*}(\textbf{s}^*)+d^m  , & \mbox{if } ~ x^{m^*}(\textbf{s}^*)+d^m-\frac{\frac{P}{2}}{c_{s^m}} \leq x^m < x^{m^*}(\textbf{s}^*)+d^m  \\
   x^m+\frac{\frac{P}{2}}{c_{s^m}} , & \mbox{if } ~ x^m < x^{m^*}(\textbf{s}^*)+d^m-\frac{\frac{P}{2}}{c_{s^m}}
\end{array} \right..
\end{eqnarray*}
The stationary policy $\breve{\boldsymbol{\pi}}$ calls for the scheduler to allocate at most $\frac{P}{2}$ units of power for transmission to each user, and tries to bring receiver $m$'s buffer towards $x^{m^*}(\textbf{s}^*)+d^m$ (before transmission), regardless of the random channel conditions.\footnote{For the proof of Theorem \ref{Th:pwl:infa}, the policy $\breve{\boldsymbol{\pi}}$ calls for the scheduler to allocate the full $P$ units of power for transmission to the single receiver when its buffer is below $x^*(s^*)+d$. The bounds are adjusted accordingly.}
For $m \in \{1,2\}$, let $\tau^m(x^m,s^m)$ be the random number of time slots until receiver $m$'s buffer level at the beginning of a slot reaches $x^{m^*}(\textbf{s}^*)$ under policy $\breve{\boldsymbol{\pi}}$, starting from state $(x^m,s^m)$. Define also $\tau_{\max}(\textbf{x},\textbf{s}):=\max\left\{\tau^1(x^1,s^1),\tau^2(x^2,s^2)\right\}$, and $\tau_{\min}:=\min\left\{\tau^1(x^1,s^1),\tau^2(x^2,s^2)\right\}$.
Note that if $x^m > x^{m^*}(\textbf{s}^*)$, then $\tau^m(x^m,s^m) = \left\lceil\frac{x-x^{m^*}(\textbf{s}^*)}{d^m}\right\rceil$, and the total discounted expected transmission and holding cost associated with receiver $m$ for the first $\tau^m(x^m,s^m)$ slots is upper bounded by:
\begin{align}\label{Eq:bound1}
&\alpha^{\tau^m(x^m,s^m)-1}\cdot c_{\max}^m \cdot d^m + \sum\limits_{t=1}^{\tau^m(x^m,s^m)} \alpha^{t-1} \cdot h^m\left(x-t\cdot d^m\right) \nonumber \\
&\leq c_{\max}^m \cdot d^m + \sum\limits_{t=1}^{\left\lceil\frac{x-x^{m^*}(\textbf{s}^*)}{d^m}\right\rceil} h^m\left(x-t\cdot d^m\right)~.
\end{align}
On the other hand, if $x^m \leq x^{m^*}(\textbf{s}^*)$, $\Expectation[\tau^m(x^m,s^m)]$ is finite.\footnote{In order to guarantee $\Expectation[\tau^m(x^m,s^m)]$ is finite, we actually need an additional assumption that Pr$\left(\frac{\frac{P}{2}}{c_{\max}^m}= d^m\right) <1$. However, this assumption is harmless, for if it is not true, the channel condition does not vary over time, a scenario outside of our scope of interest.}
Therefore, by Wald's Lemma, the total discounted expected transmission and holding cost associated with receiver $m$ for the first $\tau^m(x^m,s^m)$ slots is upper bounded by:
\begin{eqnarray}\label{Eq:bound2}
\sum\limits_{t=1}^{\tau^m(x^m,s^m)} \alpha^{t-1} \cdot \left[\frac{P}{2}+h^m\bigl(x^{m^*}(\textbf{s}^*)\bigr)\right] \leq \Expectation[\tau^m(x^m,s^m)] \cdot \left[\frac{P}{2}+h^m\bigl(x^{m^*}(\textbf{s}^*)\bigr)\right]~.
\end{eqnarray}
So for $m \in \{1,2\}$, we define:
\begin{eqnarray*}
\bar{\kappa}^m(x^m,s^m):=\left\{
\begin{array}{ll}
   c_{\max}^m \cdot d^m + \sum\limits_{t=1}^{\left\lceil\frac{x-x^{m^*}(\textbf{s}^*)}{d^m}\right\rceil} h^m\left(x-t\cdot d^m\right) , & \mbox{if } ~ x^{m^*}(\textbf{s}^*) < x^m \\
   \Expectation[\tau^m(x^m,s^m)] \cdot \left[\frac{P}{2}+h^m\bigl(x^{m^*}(\textbf{s}^*)\bigr)\right]  , & \mbox{if } ~ x^m \leq x^{m^*}(\textbf{s}^*)
\end{array} \right..
\end{eqnarray*}
Next, let $\tau_{\switch}(\textbf{x},\textbf{s})$ be the random number of time slots until the state $\bigl(\textbf{x}^*(\textbf{s}^*),\textbf{s}^*\bigr)$ is reached at the beginning of a slot under policy $\breve{\boldsymbol{\pi}}$, starting from state $(\textbf{x},\textbf{s})$. We define a new policy
$\bar{\boldsymbol{\pi}}$ that follows $\breve{\boldsymbol{\pi}}$ for $\tau_{\switch}(\textbf{x},\textbf{s})$ slots (a random stopping time), and then behaves optimally.
Then we have:
\begin{eqnarray}\label{Eq:firstv}
V_{\infty,\alpha}(\textbf{x},\textbf{s})\leq V_{\infty,\alpha}^{\bar{\boldsymbol{\pi}}}(\textbf{x},\textbf{s}) \leq \bar{\kappa}(\textbf{x},\textbf{s}) + V_{\infty,\alpha}\bigl(\textbf{x}^*(\textbf{s}^*),\textbf{s}^*\bigr)~,
\end{eqnarray}
where
\begin{align}\label{Eq:barkappa}
\bar{\kappa}(\textbf{x},\textbf{s}):=& \bar{\kappa}^1(x^1,s^1)+\bar{\kappa}^2(x^2,s^2) \nonumber \\
&+\Expectation\bigl[\tau_{\switch}(\textbf{x},\textbf{s})-\tau_{\min}(\textbf{x},\textbf{s})\bigr]\cdot \left[\textbf{c}_{\max}^{\transpose}\textbf{d}+h^1\bigl(x^{1^*}(\textbf{s}^*)\bigr)+h^2\bigl(x^{2^*}(\textbf{s}^*)\bigr)\right]~.
\end{align}
The third term in \eqref{Eq:barkappa} is an upper bound on the transmission and holding costs required to keep the vector of buffer levels at $\textbf{x}^*(\textbf{s}^*)$ while waiting for the vector of channel condition realizations to reach $\textbf{s}^*$. Since the vector of channel conditions is a finite-state ergodic Markov process, this quantity is finite.
Equation \eqref{Eq:firstv} implies:
\begin{eqnarray*}
w_{\infty,\alpha}(\textbf{x},\textbf{s}) &=& V_{\infty,\alpha}(\textbf{x},\textbf{s}) - m_{\infty,\alpha} \\
&=& V_{\infty,\alpha}(\textbf{x},\textbf{s}) - V_{\infty,\alpha}(\textbf{x}^*(\textbf{s}^*),\textbf{s}^*) \\
&\leq& \bar{\kappa}(\textbf{x},\textbf{s})<\infty~.
\end{eqnarray*}
The important thing to note here is that the bounding function $\bar{\kappa}(\textbf{x},\textbf{s})$ is independent of $\alpha$, so Condition (B) holds. The function $\bar{\kappa}(\textbf{x},\textbf{s})$ is also measurable and satisfies \eqref{Eq:b2cond}, so Condition (B2) also holds.

Finally, Condition (E) follows from the fact that for every $l \in \{1,2,\ldots\}$ and $\textbf{s} \in {\cal S}$, $w_{\infty,\alpha(l)}(\cdot,\textbf{s})$ is convex. Thus, by the finiteness of ${\cal S}$ and essentially the same argument used by Fern\'{a}ndez-Gaucherand, Marcus, and Arapostathis in \cite[pp. 178-179]{fernandez}, $\left\{w_{\infty,\alpha(l)}(\cdot,\cdot)\right\}_{l=1,2,\ldots}$ is locally equi-Lipschitzian and equicontinuous.
